\documentclass[12pt]{amsart}

\numberwithin{equation}{section}

\usepackage{amsthm}
\usepackage{amsmath}
\usepackage{amssymb}
\usepackage{verbatim}
\usepackage{enumerate}
\usepackage[lite, alphabetic, nobysame]{amsrefs}
\usepackage{graphics}
\usepackage{graphicx}
\usepackage{color}

\usepackage{hyperref}
\hypersetup{pdfstartview={XYZ null null 1.00}, pdfpagemode=UseNone, colorlinks,breaklinks, linkcolor=blue,urlcolor=blue, anchorcolor=blue,citecolor=blue}

\usepackage[capitalise]{cleveref}
\crefname{defn}{Definition}{Definitions}

\usepackage{xspace}

\usepackage{geometry}
\newcommand{\N}{\mathbb{N}}
\newcommand{\Z}{\mathbb{Z}}
\newcommand{\Q}{\mathbb{Q}}
\newcommand{\R}{\mathbb{R}}

\newcommand{\ga}{\gamma}
\newcommand\al{\alpha}

\newcommand\de{\delta}
\newcommand{\e}{\epsilon}

\newcommand\si{\sigma}
\newcommand{\Ga}{\Gamma}
\newcommand{\De}{\Delta}

\newcommand{\0}{\mathbb{\emptyset}}
\newcommand{\w}{\infty}
\newcommand{\Fr}{\mathbb{F}}

\newcommand{\proj}{\text{\textnormal{proj}}}

\newcommand{\cl}{\overline}

\newcommand{\actson}{\curvearrowright}

\newcommand{\imp}{\Rightarrow}

\newcommand{\shortiff}{\Leftrightarrow}

\newcommand{\set}[1]{\{#1\}}
\newcommand{\gen}[1]{\langle #1 \rangle}
\newcommand{\rest}[1]{\! \downharpoonright_{#1}}

\newcommand{\Slawek}{S\l{}awek\xspace}
\newcommand{\Slawomir}{S\l{}awomir\xspace}

\crefname{subsection}{Subsection}{Subsections}

\newtheorem{defn}[equation]{Definition}
\newtheorem*{defn*}{Definition}

\crefname{prop}{Proposition}{Propositions}
\newtheorem{prop}[equation]{Proposition}
\newtheorem*{propo*}{Proposition}

\newtheorem{lemma}[equation]{Lemma}

\newtheorem{obs}[equation]{Observation}

\crefname{cor}{Corollary}{Corollaries}
\newtheorem{cor}[equation]{Corollary}
\newtheorem*{cor*}{Corollary}

\newtheorem{theorem}[equation]{Theorem}
\newtheorem*{theorem*}{Theorem}

\newtheorem{claim}{Claim}
\newtheorem*{claim*}{Claim}

\newtheorem*{fact*}{Fact}

\crefname{question}{Question}{Questions}
\newtheorem{question}[equation]{Question}
\newtheorem*{question*}{Question}

\newcommand{\fntsz}[1][10]{\fontsize{#1}{#1}\selectfont}

\newenvironment{remark}[1][]{\refstepcounter{equation}\par\medskip\noindent \textbf{Remark~\theequation}#1\textbf{.} \rmfamily}{\medskip}

\newenvironment{remark*}[1][Remark]{\begin{trivlist}
\item[\hskip \labelsep {\bfseries #1.}]}{\end{trivlist}}

\newenvironment{remarklike}[2][]{\refstepcounter{equation}\par\medskip\noindent \textbf{#2~\theequation}#1\textbf{.} \rmfamily}{\medskip}

\newenvironment{remarklike*}[2][]{\par\medskip\noindent \textbf{#2}#1\textbf{.} \rmfamily}{\medskip}

\newenvironment{namedthm*}[2][]{\par\medskip\noindent \textbf{#2}#1\textbf{.}\itshape}{\medskip}

\newcommand{\iformat}[1][(\alph{enumi})]{
	\renewcommand{\theenumi}{\normalfont #1}
	\renewcommand{\labelenumi}{\theenumi}
}

\newenvironment{example}[1][]{\refstepcounter{equation}\par\medskip\noindent \textbf{Example~\theequation}#1\textbf{.} \rmfamily}{}

\newenvironment{examplelist*}[1][a]{\refstepcounter{equation}\par\medskip\noindent \textbf{Examples~\theequation.} 
\vspace{.6em}
\begin{enumerate}[\bfseries (#1)]
\setlength{\itemsep}{.8em}
\rmfamily}{\end{enumerate}\smallskip}

\creflabelformat{enumi}{#2(#1)#3}
\newcommand{\examplelistref}[2]{Example \ref{#1}\labelcref{#2}}

\newenvironment{case*}[1]{\smallskip\par\noindent \textit{Case~#1}: \rmfamily}{\smallskip}

\newenvironment{prop*}[1][]{\par\medskip\noindent \textbf{Proposition #1.} \itshape}{}

\newenvironment{pfof}[1][Claim]{\noindent \textit{Proof of #1.} \rmfamily}{\hfill{$\dashv$} \smallskip}
\newenvironment{skofpf}{\noindent \textit{Sketch of Proof.} \rmfamily}{\hfill{\qed} \medskip}

\definecolor{gris}{RGB}{90,90,90}

\usepackage[all]{xy} 
\usepackage{MnSymbol} 

\newtheorem{ergdecthm}[equation]{Ergodic Decomposition Theorem}
\newtheorem{KriegersThm}[equation]{Krieger's Finite Generator Theorem}

\newcommand{\QB}{\N^{<\N}}
\newcommand{\Br}{\mathcal{N}}
\newcommand{\Brr}{\Br_2}

\newcommand{\Uni}{\bigcup_{n \in \N}}
\newcommand{\Uone}{\bigcup_{n \geq 1}}

\newcommand{\F}{\mathcal{F}}
\newcommand{\I}{\mathcal{I}}
\newcommand{\J}{\mathcal{J}}
\newcommand{\Pa}{\mathcal{P}}
\newcommand{\Qa}{\mathcal{Q}}
\newcommand{\M}{\mathcal{M}}
\newcommand{\K}{\mathcal{K}}

\newcommand{\U}{\mathcal{U}}
\newcommand{\W}{\mathcal{W}}
\newcommand{\BA}{\mathcal{B}}

\newcommand{\Erg}{\mathcal{E}}
\newcommand{\ErgX}{\Erg_{\Z}(X)}

\newcommand{\Sm}{\mathfrak{S}}
\newcommand{\Afrak}{\mathfrak{A}}
\newcommand{\Bfrak}{\mathfrak{B}}
\newcommand{\Cfrak}{\mathfrak{C}}
\newcommand{\Id}{\mathfrak{I}}

\newcommand{\fin}{\text{fin}}
\newcommand{\lfin}{\text{lfin}}

\newcommand{\GN}{G^{\N}}
\newcommand{\sone}{\mathbf{\Sigma}_1^1}
\newcommand{\s}{\sigma(\sone)}
\newcommand{\EI}{{F_{\I}}}
\newcommand{\f}{{f_{\I}}}

\newcommand{\n}{\bar{n}}
\newcommand{\Ez}{\tilde{E}_0}

\newcommand{\MEAGER}{\text{\textnormal{MEAGER}}}
\newcommand{\NULL}{\text{\textnormal{NULL}}}

\title[\tiny Finite generators for countable group actions]{Finite generators for countable group actions in the Borel and Baire category settings}
\author{Anush Tserunyan}
\address{Department of Mathematics, University of Illinois at Urbana-Champaign, IL, 61801, USA}
\email{anush@illinois.edu}
\thanks{This work is part of the author's Ph.D. thesis completed at University of California at Los Angeles in 2013.}
\date{}

\begin{document}
\maketitle

\begin{abstract}
For a continuous action of a countable discrete group $G$ on a Polish space $X$, a countable Borel partition $\Pa$ of $X$ is called a \textit{generator} if $G \Pa := \set{gP : g \in G, P \in \Pa}$ generates the Borel $\si$-algebra of $X$. For $G = \Z$, the Kolmogorov--Sinai theorem gives a measure-theoretic obstruction to the existence of finite generators: they do not exist in the presence of an invariant probability measure with infinite entropy. It was asked by Benjamin Weiss in the late 80s whether the nonexistence of any invariant probability measure guarantees the existence of a finite generator. We show that the answer is positive (in fact, there is a $32$-generator) for an arbitrary countable group $G$ and $\si$-compact $X$ (in particular, for locally compact $X$). We also show that any continuous aperiodic action of $G$ on an arbitrary Polish space admits a $4$-generator on a comeager set, thus giving a positive answer to a question of Alexander Kechris asked in the mid-90s.

Furthermore, assuming a positive answer to Weiss's question for arbitrary Polish spaces and $G=\Z$, we prove the following dichotomy: every aperiodic Borel action of $\Z$ on a Polish space $X$ admits either an invariant probability measure of infinite entropy or a finite generator. As an auxiliary lemma, we prove the following statement, which may be of independent interest: every aperiodic Borel action of a countable group $G$ on a Polish space $X$ admits a $G$-equivariant Borel map to the aperiodic part of the shift action $G \actson 2^G$.

We also obtain a number of other related results, among which is a criterion for the nonexistence of non-meager weakly wandering sets for continuous actions of $\Z$. A consequence of this is a negative answer to a question asked by Eigen--Hajian--Nadkarni, which was also independently answered by Benjamin Miller.
\end{abstract}

\tableofcontents

\section{Introduction}

\subsection{Definition and examples}
Throughout the paper let $G$ denote a countably infinite discrete group and let $1_G$ denote the identity of $G$. Let $X$ be a \emph{Borel $G$-space}, i.e. a standard Borel space (see \cite[12.5]{bible}) equipped with a Borel action of $G$. If $X$ is a Polish topological space and $G$ acts continuously, we call it a \emph{Polish $G$-space}. For $g \in G$ and $x \in X$, we denote the result of the action of $g$ on $x$ by $gx$ (instead of $g \cdot x$).

Consider the following game: Player I chooses a finite or countable Borel partition\footnote{We call a countable partition Borel if each piece of it is a Borel set.} $\I = \set{A_n}_{n<k}$ of $X$ ($k \le \w$), then Player II chooses $x \in X$ and Player I tries to guess $x$ by asking questions to Player II regarding which piece of the partition $x$ lands in when moved by a certain group element. More precisely, for every $g \in G$, Player I asks ``To which $A_n$ does $gx$ belong?'' and Player II gives an index $n_g < k$ as an answer. Whether or not Player I can uniquely determine $x$ from the sequence $(n_g)_{g \in G}$ of responses depends on how cleverly he chose the partition $\I$. We call $\I$ a \emph{generator} (or a \emph{generating partition}) if it guarantees that Player I will determine $x$ correctly no matter which $x$ Player II chooses. Here is the precise definition, which also explains the terminology.

\begin{defn}[Generator]
For a Borel $G$-space $X$, a countable Borel partition $\I = \set{A_n}_{n<k}$ of $X$ ($k \le \w$) is called a generator if $G \I := \set{gA_n : g \in G, n<k}$ generates the Borel $\sigma$-algebra of $X$. We also call $\I$ a $k$-generator, and, if $k < \w$, a finite generator.
\end{defn}

\begin{examplelist*}\label{example:generators}
	\item\label{item:shift} For $k \le \w$, let $X = k^G$ (with the understanding that $X = \N^G$ when $k = \w$) be equipped with the product topology and the shift action of $G$, i.e. $g x (h) := x (g^{-1}h)$, for $g,h \in G, x \in X$. For each $n < k$, put $V_n = \set{x \in X : x(1_G) = n}$. The partition $\I = \set{V_n}_{n<k}$ of $X$ is clearly a generator because $G \I$ is a sub-basis for the product topology on $k^G$.
	
	\item Let $X = S^1$ be the unit circle in the complex plain, $\al \in \R$ be irrational relative to $\pi$ (i.e. $\al / \pi \notin \Q$), and $T_\al : X \to X$ be the rotation by angle $\al$, i.e. $T_\al(z) = e^{i\al}z$, for $z \in X$. This induces a continuous action $\Z \actson X$ and we claim that the partition $\I = \set{C_u, C_l}$ of $X$ into upper and lower half circles is a generator. This is not hard to verify directly, using the fact that every orbit is dense, but it easier to check the equivalent condition to being a generator given in \labelcref{item:separating_points} of \cref{st:equivalences_to_generator} below.
	
	\item If a Borel $G$-space $X$ admits a Borel weakly wandering (see \cref{defn of weakly wandering set}) complete section\footnote{For an action $G \actson X$, a subset of $X$ is called a \emph{complete section} if it non-trivially intersects every orbit.} $W$, e.g. the translation action of $\Z$ on $\R$ with $W = [0,2]$, then it is not hard to prove that it admits a $3$-generator (see \cref{st:ww_imp_3-gen} for a short and self-contained proof).
	
	\item The translation action of $\Z$ on $\R$ actually admits a $2$-generator because the exponentiation function $2^x$ makes it Borel isomorphic to the Borel $\Z$-space $\R^{>0}$, where the action of $1_\Z$ is given by multiplication by $2$; the latter action in its turn is Borel isomorphic to an invariant subset of the shift action $\Z \actson 2^\Z$ via the binary representation, which, by \labelcref{item:shift} above, has a $2$-generator. The fact that the translation action $\Z \actson \R$ has a $2$-generator is true more generally (\cref{2-generator for free smooth}) for smooth (\cref{defn:smoothness_via_transversal}) free actions.
\end{examplelist*}

For a Borel $G$-space $X$ and a Borel partition $\I = \set{A_n}_{n < k}$ of $X$,  $k \le \w$, let $\f : X \rightarrow k^G$ be defined by $x \mapsto (n_g)_{g \in G}$, where $n_g$ is such that $gx \in A_{n_g}$. This is often called the \emph{symbolic representation} or \emph{coding map} for the process $(X, G, \I)$. Clearly $\f$ is a $G$-equivariant Borel map and, for every $x \in X$, $\f(x)$ is the sequence of responses of Player I in the above game.

\begin{prop}\label{st:equivalences_to_generator}
For a Borel $G$-space $X$ and a Borel partition $\I = \set{A_n}_{n<k}$, $k \le \w$, the following are equivalent:
\begin{enumerate}[(1)]
\item $\I$ is a generator.

\item\label{item:separating_points} $G \I$ separates points\footnote{A collection $\F$ of subsets of $X$ \emph{separates points} if for all distinct $x,y \in X$ there is $A \in \F$ such that $x \in A \nLeftrightarrow y \in A$}.

\item The coding map $\f : X \to k^G$ is one-to-one.
\end{enumerate}
\end{prop}
\begin{proof}
	The only part worth proving is (3)$\imp$(1), so suppose $\f$ is one-to-one. Then, by the Luzin--Souslin theorem (see \cite[15.1]{bible}), $\f$ is a $G$-equivariant Borel isomorphism between $X$ and $Y := \f(X)$, so it is enough to show that the partition $\f(\I) := \set{\f(A_n)}_{n<k}$ of $Y$ is a generator for the shift action of $G$ on $Y$. But, in the notation of \examplelistref{example:generators}{item:shift}, $\f(A_n) = V_n \cap Y$, for each $n<k$, so $\f(\I)$ is indeed a generator for $G \actson Y$.
\end{proof}

Conversely, given a $G$-equivariant Borel map $f : X \rightarrow k^G$ for some $k \le \w$, define a partition $\I_f = \set{A_n}_{n<k}$ by $A_n = f^{-1}(V_n)$, where $V_n$ is as in \examplelistref{example:generators}{item:shift}. Note that $f_{\I_f} = f$. This and \cref{st:equivalences_to_generator} imply the following.

\begin{cor}
For $k \le \w$, $X$ admits a $k$-generator if and only if there is a $G$-equivariant Borel embedding of $X$ into $k^G$.
\end{cor}

In all arguments to follow in this paper, we use these equivalent descriptions of a generator without comment.

\subsection{Countable generators}

In \cite{Weiss}, it was shown that every aperiodic (i.e. having no finite orbits) $\Z$-space admits a countable generator. This was later generalized to any countable group in \cite{JKL}.

\begin{theorem}[Weiss, Jackson--Kechris--Louveau]\label{JKL thm}
	Every aperiodic Borel $G$-space $X$ admits a countable generator. In particular, there is a Borel $G$-equivariant embedding of $X$ into $\N^G$.
\end{theorem}

This is sharp in the sense that we could not hope to obtain a finite generator solely from the aperiodicity assumption because of the measure-theoretic obstruction provided by the Kolmogorov--Sinai theorem (see \labelcref{Kolmogorov--Sinai}) as explained below. For example, the action of $\Z$ on $[0,1]^{\Z} \setminus A$ by shift, where $A$ is the set of periodic points, is aperiodic, but it does not admit a finite generator since it admits an invariant probability measure of infinite entropy, namely, the Lebesgue measure.

Thus, the question of the existence of countable generators is completely resolved and the current paper studies the existence of finite generators.

\subsection{Entropy and finite generators}

Generators arose in the study of entropy in ergodic theory. Let $(X, \mu, T)$ be a dynamical system, i.e. $(X, \mu)$ is a standard probability space and $T$ is a Borel measure preserving automorphism of $X$. We can interpret the above game as follows: 
\begin{itemize}
\item $X$ is the set of possible pictures of the world,
\item $\I$ is an experiment that Player I conducts,
\item the point $x \in X$ that Player II chooses is the true picture of the world,
\item $T$ is the unit of time.
\end{itemize}
Assume that $\I$ is finite (indeed, we want our experiment to have finitely many possible outcomes). Player I repeats the experiment every day (also going back in time) and Player II tells its outcome. The goal is to find the true picture of the world $x$ with probability $1$. This happens exactly when $\I$ is a generator $\mu$-a.e.

The \emph{static entropy} of the experiment $\I = \set{P_n}_{n < k}$ is defined by
\begin{equation}\label{eq:static_entropy}
	h_{\mu}(\I) = - \sum_{n<k} \mu(A_n) \log_2 \mu(A_n),
\end{equation}
and intuitively, it measures our probabilistic uncertainty about the outcome of the experiment. For example, if for some $n < k$, $A_n$ had measure $1$, then we would be probabilistically certain that the outcome is going to be in $A_n$ and $h_\mu(\I) = 0$. On the other hand, if all of $A_n$ had probability ${1 \over k}$, then our uncertainty would be the highest, namely, $h_\mu(\I) = \log_2 k$. Equivalently, according to Shannon's interpretation, $h_{\mu}(\I)$ measures how much information we gain from learning the outcome of the experiment.

We now define the \emph{time average} or \emph{dynamic} entropy of $\I$ by
\begin{equation}\label{eq:dynamic_entropy}
	h_{\mu}(\I, T) = \lim_{n \to \infty} {1 \over 2n+1} h_{\mu}(\bigvee_{i=-n}^n T^i \I),
\end{equation}
where $\bigvee$ denotes the join (the least common refinement) of the partitions. The sequence in the limit is decreasing and hence the limit always exists and is finite (see \cite{Glasner} or \cite{Rudolph}).

Finally the \emph{entropy of the dynamical system} $(X,\mu,T)$ is defined as the supremum over all (finite) experiments:
\begin{equation}\label{eq:sup_entropy}
	h_{\mu}(T) = \sup_{\I} h_{\mu}(\I, T),
\end{equation}
and it could be finite or infinite. Now it is plausible that if $\I$ is a finite generator (and hence Player I wins the above game), then $h_{\mu}(\I, T)$ should be all the information there is to obtain about $X$ and hence $\I$ achieves the supremum above. This is indeed the case as the following theorem (\cite[Theorem 14.33]{Glasner}) shows.

\begin{theorem}[Kolmogorov--Sinai, '58--59]\label{Kolmogorov--Sinai}
If $\I$ is a finite generator modulo $\mu$-$\NULL$, then $h_{\mu}(T) = h_{\mu}(\I, T)$. In particular, $h_{\mu}(T) \le \log_2(|\I|) < \infty$.
\end{theorem}

Here $\mu$-$\NULL$ denotes the $\sigma$-ideal of $\mu$-null sets and, by definition, a statement holds modulo a $\sigma$-ideal $\Id$ if it holds on $X \setminus Z$, for some $Z \in \Id$. We will also use this for $\Id = \MEAGER$ (the $\sigma$-ideal of meager sets in a Polish space).

In case of ergodic systems, i.e. dynamical systems where every measurable invariant set is either null or conull, the converse of Kolmogorov--Sinai theorem is true (see \cite{Krieger}):

\begin{theorem}[Krieger, '70]\label{Krieger's thm}
 Suppose $(X, \mu, T)$ is ergodic. If $h_{\mu}(T) < \log_2 k$, for some $k \ge 2$, then there is a $k$-generator modulo $\mu$-$\NULL$.
\end{theorem}

\subsection{Weiss's question and potential dichotomy theorems}

Now let $X$ be just a Borel $\Z$-space with no measure specified. By the Kolmogorov--Sinai theorem, if there exists an invariant Borel probability measure on $X$ with infinite entropy, then $X$ does not admit a finite generator. Is this the only obstruction to having a finite generator? More precisely:

\begin{question}\label{natural_question}
If a Borel $\Z$-space $X$ does not admit any invariant Borel probability measure of infinite entropy, does it admit a finite generator?
\end{question}

The following seemingly simpler question was first asked by Weiss in \cite{Weiss}:

\begin{question}[Weiss, '87]\label{Weiss's question_for_Z}
If a Borel $\Z$-space $X$ does not admit any invariant Borel probability measure, does it admit a finite generator?
\end{question}

It is shown below in \cref{section_dichotomy} that these two questions are actually equivalent, and thus, a positive answer to Weiss's question would imply the following dichotomy theorem:

\begin{namedthm*}{\cref*{dichotomyII}}
Suppose the answer to \cref{Weiss's question_for_Z} is positive and let $X$ be an aperiodic Borel $\Z$-space. Then exactly one of the following holds:
\begin{enumerate}[(1)]
\item there exists an invariant Borel probability measure with infinite entropy;

\item $X$ admits a finite generator.
\end{enumerate}
\end{namedthm*}

We remark that the nonexistence of an invariant ergodic probability measure of infinite entropy does not guarantee the existence of a finite generator. For example, let $X$ be a direct sum of uniquely ergodic actions $\Z^{\curvearrowright} \! X_n$ such that the entropy $h_n$ of each $X_n$ is finite but $h_n \to \w$. Then $X$ does not admit an invariant ergodic probability measure with infinite entropy since otherwise it would have to be supported on one of the $X_n$, contradicting the unique ergodicity. Neither does $X$ admit a finite generator since that would contradict the Kolmogorov--Sinai theorem applied to $X_n$ for large enough $n$.

However, assuming again that the answer to \cref{Weiss's question_for_Z} is positive, we prove the following dichotomy suggested by Kechris:
\begin{namedthm*}{\cref*{dichotomyI}}
Suppose the answer to \cref{Weiss's question_for_Z} is positive and let $X$ be an aperiodic Borel $\Z$-space. Then exactly one of the following holds:
\begin{enumerate}[(1)]
\item there exists an invariant ergodic Borel probability measure with infinite entropy,

\item there exists a partition $\set{Y_n}_{n \in \N}$ of $X$ into invariant Borel sets such that each $Y_n$ admits a finite generator.
\end{enumerate}
\end{namedthm*}

The proofs of these dichotomies (presented in \cref{section_dichotomy}) use the Ergodic Decomposition Theorem and a uniform version of Krieger's theorem in tandem with \cref{separating smooth-many invariant sets}, which provides a finite partition that separates the equivalence classes of an invariant smooth equivalence relation (\cref{defn:smooth_eq_rel}).

\subsection{Weiss's question for an arbitrary group and the main result of the paper}

Because \cref{natural_question,Weiss's question_for_Z} are equivalent, we may focus on answering the latter. Moreover, since the statement of \cref{Weiss's question_for_Z} does not use the notion of entropy, one may as well state it for an arbitrary countable group $G$ as it is done in \cite{JKL}:

\begin{question}[Weiss '87, Jackson--Kechris--Louveau '02]\label{Weiss's question}
	Let $G$ be a countable group and $X$ be a Borel $G$-space. If $X$ does not admit any invariant Borel probability measure, does it admit a finite generator?
\end{question}

In order to state our answer, we need the following:

\begin{defn}
	Let $X$ be a Borel $G$-space and denote its Borel $\sigma$-algebra by $\Bfrak(X)$. For a topological property $P$ (e.g. Polish, $\sigma$-compact, etc.), we say that $X$ admits a $P$ topological realization, if there exists a Hausdorff second countable topology on $X$ satisfying $P$ such that it makes the $G$-action continuous and its induced Borel $\sigma$-algebra is equal to $\Bfrak(X)$.
\end{defn}

We remark that every Borel $G$-space admits a Polish topological realization\footnote{This is actually true also for an uncountable Polish group $G$, but it is a highly non-trivial result of Becker and Kechris, see \cite[5.2]{BK}}, see, \cite[13.12.ii]{bible}. The main result of this paper is a positive answer to \cref{Weiss's question} in case $X$ has a $\sigma$-compact realization:

\begin{namedthm*}{\cref*{sigma-compact compressible => finite generator}}
Let $X$ be a Borel $G$-space that admits a $\sigma$-compact realization. If there is no $G$-invariant Borel probability measure on $X$, then $X$ admits a Borel $32$-generator.
\end{namedthm*}

For example, \cref{Weiss's question} has a positive answer when $G$ acts continuously on a locally compact or even just $\sigma$-compact Polish space.

\begin{remark}
	The number $32$ in the above theorem comes from the fact that the generator is constructed as the partition generated by $5$ Borel sets.
\end{remark}

\begin{remark}
	The fact that a concrete numerical bound of $32$ is obtained in the conclusion of the above theorem is somewhat surprising. However, Robin Tucker-Drob pointed out that if \cref{Weiss's question} had a positive answer, then automatically there would be a uniform finite bound on the number generators; indeed, otherwise, there is an unbounded sequence $(k_n)_{n \in \N}$ of natural numbers such that for each $n \in \N$, there is a Borel $G$-space $X_n$ that
	\begin{enumerate}[(i)]
		\item does not admit an invariant probability measure,
		\item admits a $k_n$-generator, 
		\item does not admit a $k$-generator for $k<k_n$.
	\end{enumerate}
	Then, letting $X$ be the disjoint union of the $X_n$, $n \in \N$, we see that $X$ still does not admit an invariant probability measure, but neither does it admit a finite generator, contradicting the assumption that the answer to \cref{Weiss's question} is positive.
\end{remark}

\begin{remarklike}{Addendum}
	In the original version of the current paper, it was also asked whether every Borel $G$-space admits a $\sigma$-compact realization. However, this was later answered negatively by Conley, Kechris and Miller in \cite{CKM}.
\end{remarklike}

Before explaining the idea of the proof of the above theorem, we present previously known results as well as other related results obtained in this paper.

\subsection{Finite generators in the measure-theoretic setting}

The following result gives a positive answer to a version of \cref{Weiss's question} in the measure-theoretic context.

\begin{theorem}[Krengel--Kuntz \cite{Krengel,Kuntz}]\label{Kuntz's thm}
	Let $(X,\mu)$ be a standard probability space equipped with a nonsingular\footnote{An action is \emph{nonsingular} if it preserves the $\mu$-null sets.} Borel action of $G$. If there is no invariant Borel probability measure absolutely continuous with respect to $\mu$, then $X$ admits a $2$-generator modulo $\mu$-$\NULL$.
\end{theorem}

The proof of this uses a version of the Hajian--Kakutani--It\^{o} theorem (see \labelcref{st:HKI_theorem} below), which states that the hypothesis of the Krengel--Kuntz theorem is equivalent to the existence of a weakly wandering Borel set (see \cref{defn of weakly wandering set}) of arbitrarily large $\mu$-measure\footnote{Here, by arbitrarily large we mean arbitrarily close to $1$.}. However, as explained below in \cref{subsection:traveling/ww_vs_fin-gen}, the analogues of the Hajian--Kakutani--It\^{o} theorem fail in the Borel and Baire category settings. Therefore, other means needed to be used to obtain the results of the current paper.

\subsection{Finite generators in the Baire category setting}

In the mid-'90s, Kechris asked whether an analogue of the Krengel--Kuntz theorem holds in the context of Baire category (see \cite[6.6.(B)]{JKL} for a more particular version of this question), more precisely:

\begin{question}[Kechris, mid-'90s]\label{Kechris's question}
Does every aperiodic Polish $G$-space admit a finite generator on an invariant comeager set?
\end{question}

The nonexistence of invariant measures is not mentioned in the hypothesis of the question because it is automatic in the context of Baire category due to the following:

\begin{theorem}[Kechris--Miller {\cite[13.1]{KM}}]\label{Kechris-Miller thm}
For any aperiodic Polish $G$-space, there is an invariant comeager set that does not admit any invariant probability measure.
\end{theorem}

Thus, a positive answer to \cref{Weiss's question} for \emph{all} Borel $G$-spaces would imply a positive answer to \cref{Kechris's question}, but the latter does not follow from \cref{sigma-compact compressible => finite generator}. However, using entirely different techniques, we still give an affirmative answer to \cref{Kechris's question}:

\begin{namedthm*}{\cref*{4-generator modulo a meager set}}
Any aperiodic Polish $G$-space admits a $4$-generator on an invariant comeager set.
\end{namedthm*}

The proof of this is entirely contained in \cref{section_finite gen modulo meager} and can be read independently from the rest of the paper. It uses the Kuratowski--Ulam method introduced in the proofs of Theorems 12.1 and 13.1 in \cite{KM}. This method was inspired by product forcing and its idea is as follows. Suppose we want to prove the existence of an object $A$ that satisfies a certain condition on a comeager set (in our case a finite partition). We give a parametrized construction of such objects $A_{\alpha}$, where the parameter $\alpha$ ranges over $2^{\N}$ or $\N^{\N}$ (or any other Polish space), and then try to show that for comeager many values of the parameter $\alpha$, the object $A_{\alpha}$ has the desired property $\Phi$ on a comeager set. In other words, we want to prove 
$$
\forall^* \alpha \forall^* x \Phi(\alpha, x),
$$ 
where $\forall^*$ means ``for comeager many''. Now the key point is that the Kuratowski--Ulam theorem allows us to switch the order of the quantifiers and prove 
$$
\forall^* x \forall^* \alpha \Phi(\alpha, x)
$$
instead. The latter is often an easier task since it allows us to work locally with a fixed $x \in X$.

\subsection{Connections with weakly wandering and traveling sets}\label{subsection:traveling/ww_vs_fin-gen}

Going back to the Borel setting and \cref{Weiss's question}, we now explore the hypothesis of the question, namely, the nonexistence of an invariant probability measure. In the measure-theoretic setting, the latter is tightly connected to weakly wandering sets (\cref{defn of weakly wandering set}):

\begin{theorem}[Hajian--Kakutani--It\^{o} \cite{HK,HI}]\label{st:HKI_theorem}
	Let $(X,\mu)$ be a standard probability space equipped with a nonsingular Borel action of $G$. There is no invariant Borel probability measure equivalent to $\mu$ if and only if there is a $\mu$-positive weakly wandering Borel set.
\end{theorem}

In the Borel setting, the following statement would be an analogue of \cref{st:HKI_theorem}: \emph{a Borel $G$-space does not admit any invariant Borel probability measure if and only if it admits a weakly wandering Borel complete section}. However, in \cite{EHN} the authors show that it does not hold. Nevertheless, their counterexample does not rule out local versions of the statement, leaving the following question unsettled \cite[question (ii) on page 9]{EHN}:

\begin{question}[Eigen--Hajian--Nadkarni, '93]\label{EHN's question}
	Let $X$ be a Borel $\Z$-space. If $X$ does not admit an invariant probability measure, is there a countably generated (by Borel sets) partition of $X$ into invariant sets, each of which admits a weakly wandering Borel complete section?
\end{question}

We define a related notion of \emph{locally weakly wandering} sets (\cref{defn of weakly wandering set}) --- a generalization of weakly wandering sets, and show that the presence of a locally weakly wandering Borel complete section still gives a 4-generator (\cref{finite generators for lww}). However, we give an example (\cref{failure of having lww mod meager}) of a Polish $G$-space that does not support an invariant probability measure, but neither does it admit a locally weakly wandering Borel complete section nor a partition as in \cref{EHN's question}, thus providing a negative answer to the latter question. This result is a consequence of a criterion for the nonexistence of non-meager weakly wandering sets (\cref{syndetic open sets => no ww set}), which also implies that the analogues of the Hajian--Kakutani--It\^{o} theorem fail not only in the Borel, but also in the Baire category setting. Ben Miller pointed out that he had also obtained a negative answer to \cref{EHN's question} in \cite[Example 3.13]{Miller_thesis}.

We further generalize locally weakly wandering sets to \emph{traveling sets} (\cref{defn of traveling sets}) and observe as an immediate consequence of Nadkarni's theorem (see \cref{subsection:Nadkarni's theorem,Nadkarni's thm}) that a Borel $G$-space does not admit an invariant probability measure if and only if it admits a Borel traveling complete section (\cref{equivalences to compressibility}). Thus, \cref{Weiss's question} amounts to whether the existence of a Borel traveling complete section implies the existence of a finite generator. Regarding this formulation of the question, the strongest result obtained in the current paper is as follows: we define an intermediate (between traveling and locally weakly wandering) notion of \emph{locally finitely traveling} sets (\cref{defn of fin traveling sets}) and prove that the presence of a Borel locally finitely traveling complete section implies the existence of a $32$-generator (\cref{lfin traveling => 2^5-generator}). Unlike the analogous results with weakly wandering and locally weakly wandering sets (\cref{st:ww_imp_3-gen,finite generators for lww}), this result is nonconstructive --- its proof uses the same machinery as that of \cref{sigma-compact compressible => finite generator}.

The following diagram summarizes the results mentioned in this subsection; here ``c.s.'' stands for ``complete section''.

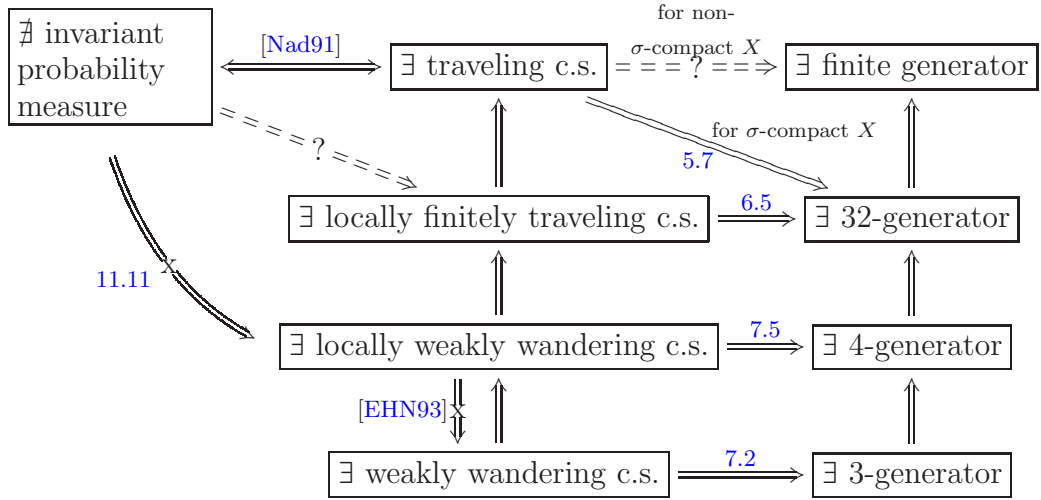
\begin{displaymath}\xymatrix@C-.2cm{
	\fbox{\parbox{.15\textwidth}{$\nexists$ invariant \\ probability \\ measure}}\ar@{<=>}[r]^-{\txt{\scriptsize \cite{Nadkarni}}} \ar@{=>}@<-2ex>@/^-3.5pc/[2,1]|-{\txt{x}}_-{\txt{\scriptsize \labelcref{failure of having lww mod meager} \ }}
	\ar@{==>}[dr]|-*+[o]{?} & \framebox{$\exists$ traveling c.s.}
	\ar@{=>}[dr]_-{\txt{\scriptsize \labelcref{sigma-compact compressible => finite generator}}}^(.55){\txt{\tiny \hspace{.58cm} for $\si$-compact $X$}}
	\ar@{==>}[r]|-*+[o]{?}^-{\txt{\tiny for non- \\ \tiny $\si$-compact $X$}} & \framebox{$\exists$ finite generator} \\
	& \framebox{$\exists$ locally finitely traveling c.s.}\ar@{=>}[r]^-{\txt{\scriptsize \labelcref{lfin traveling => 2^5-generator}}} \ar@{=>}[u] & \framebox{$\exists$ $32$-generator} \ar@{=>}[u]\\
	& \framebox{$\exists$ locally weakly wandering c.s.} \ar@{=>}@<-3ex>[d]|-{\txt{x}}_-{\txt{\scriptsize \cite{EHN}}}\ar@{=>}[r]^-{\txt{\scriptsize \labelcref{finite generators for lww}}} \ar@{=>}[u] & \framebox{$\exists$ $4$-generator} \ar@{=>}[u]\\
	& \framebox{$\exists$ weakly wandering c.s.} \ar@{=>}[r]^-{\txt{\scriptsize \labelcref{st:ww_imp_3-gen}}} \ar@{=>}[u] & \framebox{$\exists$ $3$-generator} \ar@{=>}[u]
}\end{displaymath}

\subsection{Separating smooth-many sets and maps to the aperiodic part of the shift action}

In the proofs of the dichotomy theorems mentioned above (\labelcref{dichotomyII,dichotomyI}), we needed to apply the following result to the equivalence relation $E$ of being in the same component of the ergodic decomposition; we mention it here as it may be of interest on its on.

\begin{namedthm*}{\cref*{separating smooth-many invariant sets}}
	Let $X$ be an aperiodic Borel $G$-space and let $E$ be a smooth (see \cref{defn:smooth_eq_rel}) $G$-invariant equivalence relation on $X$. There exists a partition $\Pa$ of $X$ into $4$ Borel sets such that $G \Pa$ separates any two $E$-nonequivalent points in $X$, i.e. for all $x, y \in X$, $[x]_E \ne [y]_E$ implies $f_{\Pa}(x) \neq f_{\Pa}(y)$.
\end{namedthm*}

The proof of this theorem in its turn uses the following result, which may also be of independent interest.

\begin{namedthm*}{\cref*{map_to_aperiodic_part_of_shift}}
	Any aperiodic Borel $G$-space admits a $G$-equivariant Borel map to the aperiodic part of the shift action $G \actson 2^G$.
\end{namedthm*}

\subsection{Nadkarni's theorem}\label{subsection:Nadkarni's theorem}

In order to explain the idea behind the proof of \cref{sigma-compact compressible => finite generator}, we now present an equivalent condition to the hypothesis of \cref{Weiss's question}, i.e. to the nonexistence of invariant measures. It was proved by Nadkarni in \cite{Nadkarni} and it is the analogue of Tarski's theorem about paradoxical decompositions (see \cite{Wagon}) for countably additive measures.

Let $X$ be a Borel $G$-space and denote the set of invariant Borel probability measures on $X$ by $\M_G(X)$. Also, for $S \subseteq X$, let $[S]_G$ denote the saturation of $S$, i.e. $[S]_G = \bigcup_{g \in G} gS$.

The following definition makes no reference to any invariant measure on $X$, yet provides a sufficient condition for the measure of two sets to be equal (resp. $\le$ or $<$).

\begin{defn}\label{defn of equidec}
	Two Borel sets $A,B \subseteq X$ are said to be equidecomposable (denoted by $A \sim B$) if there are Borel partitions $\{A_n\}_{n \in \N}$ and $\{B_n\}_{n \in \N}$ of $A$ and $B$, respectively, and $\{g_n\}_{n \in \N} \subseteq G$ such that $g_n A_n = B_n$. We write $A \preceq B$ if $A \sim B' \subseteq B$, and we write $A \prec B$ if moreover $[B \setminus B']_G = [B]_G$.
\end{defn}

The following explains the above definition.
\begin{prop}\label{st:equidec_vs_measure}
	Let $A,B \subseteq X$ be Borel sets and $\mu \in \M_G(X)$.
	\begin{enumerate}[(a)]
		\item If $A \sim B$, then $\mu(A) = \mu(B)$.
		\item If $A \preceq B$, then $\mu(A) \le \mu(B)$.
		\item\label{item:Dedekind_containment} If $A \prec B$, then either $\mu(A) = \mu(B) = 0$ or $\mu(A) < \mu(B)$.
	\end{enumerate}
\end{prop}
\begin{proof}
	The only part worth proving is \labelcref{item:Dedekind_containment} and to this end, let $B' \subseteq B$ be such that $A \sim B'$ and $[B \setminus B']_G = [B]_G$. If $C := B \setminus B'$ is $\mu$-null, then so is $B$, and hence $A$, because $B \subseteq [C]_G = \bigcup_{g \in G} g C$. On the other hand, if $\mu(C)>0$, then, because $\mu$ is a finite measure, we have $\mu(A) = \mu(B') = \mu(B) - \mu(C) < \mu(B)$.
\end{proof}

\begin{defn}\label{defn of compressibility}
	A Borel set $A \subseteq X$ is called compressible if $A \prec A$.
\end{defn}

In other words, $A \prec A$ means that there is a Borel witness to the fact that each orbit of $A$ is Dedekind infinite, i.e. is equinumerous to its proper subset.

It is clear from \labelcref{item:Dedekind_containment} of \cref{st:equidec_vs_measure} that if a Borel set $A \subseteq X$ is compressible, then $\mu(A) = 0$ for all $\mu \in \M_G(X)$. In particular, if $X$ itself is compressible then $\M_G(X) = \0$. Thus compressibility is an apparent obstruction to having an invariant probability measure. It turns out that it is the only one:
\begin{theorem}[Nadkarni, '91]\label{Nadkarni's thm}
	A Borel $G$-space $X$ admits an invariant Borel probability measure if and only if it is not compressible.
\end{theorem}

The proof of this first appeared in \cite{Nadkarni} for $G = \Z$ and is also presented in Chapter 4 of \cite{BK} for an arbitrary countable group $G$. Although we do not explicitly use this theorem in our arguments, we largely use ideas from its proof.

\subsection{Outline of the proof of \cref{sigma-compact compressible => finite generator}}

In our attempt to positively answer \cref{Weiss's question}, we take the non-constructive approach and try to prove the contrapositive:
\begin{center}
	No finite generator $\imp$ $\exists$ an invariant probability measure.
\end{center}

When constructing an invariant measure (e.g. Haar measure), one usually needs some notion of ``largeness'' so that $X$ is ``large'' (e.g. having nonempty interior, being incompressible). So we aim at something like this:
$$\begin{array}{rcl}
\text{No 32-generator} & & \text{$\exists$ an invariant probability measure} \\
\Searrow & & \Nearrow \\
& \text{$X$ is not ``small'' $=$ $X$ is ``large''} &
\end{array}$$

In the definition of equidecomposability of sets $A$ and $B$, the partitions $\set{A_n}_{n \in \N}$ and $\set{B_n}_{n \in \N}$ belong to the Borel $\sigma$-algebra. For $i \ge 1$, we define a finer notion of equidecomposability by restricting the Borel $\sigma$-algebra to \emph{some} $\sigma$-algebra that is generated by the $G$-translates of $i$-many Borel sets. In this case we say that $A$ and $B$ are \emph{$i$-equidecomposable} and denote by $A \sim_i B$. In other words, $A \sim_i B$ if $i$-many Borel sets are enough to generate a $G$-invariant $\sigma$-algebra that is sufficiently fine to carve out partitions $\set{A_n}_{n \in \N}$ and $\set{B_n}_{n \in \N}$ witnessing $A \sim B$.

As before, we say that a set $A$ is \emph{$i$-compressible} if $A \prec_i A$. Taking $i$-compressibility as our notion of ``smallness'', we prove the following:
$$\begin{array}{rcl}
\text{No 32-generator} & & \text{$\exists$ an invariant probability measure} \\
\text{\footnotesize{(1)}} \Searrow & & \Nearrow \text{\footnotesize{(2)}} \\
& \text{$X$ is not 4-compressible} &
\end{array}$$

We prove the contrapositive of Step (1). More precisely, assuming $i$-compressibility, we construct a $2^{i+1}$-generator by hand (see \cref{i-compressible => finite generator}), thus obtaining:
\begin{center}
	No $2^5$-generator $\imp$ $X$ is not $4$-compressible.
\end{center}

Step (2) is proving an analogue of Nadkarni's theorem for $i$-compressibility:
\begin{center}
	$X$ is not $4$-compressible $\imp$ $\exists$ an invariant probability measure.
\end{center}
To accomplish this step, firstly, we show that $i$-compressibility is indeed a notion of ``smallness'', i.e. that the set of $i$-compressible sets (roughly speaking) forms a $\sigma$-ideal (see \cref{C_i is a sigma-ideal}). The difficulty here is to prevent $i$ from growing when taking unions.

Secondly, we assume that $X$ is not $4$-compressible and give a construction of a measure (see \cref{section_Nadkarni's proof}) reminiscent of the one in the proof of Nadkarni's theorem or the existence of Haar measure. But unfortunately, our proof only yields a family of \emph{finitely additive} invariant probability measures because here we cannot prevent $i$ from growing when taking countable unions. However, with the additional assumption that $X$ is $\sigma$-compact, we are able to concoct a countably additive invariant probability measure out of this family of finitely additive measures, thus obtaining \cref{sigma-compact compressible => finite generator}.

\subsection{Open questions}
Here are some open questions that arose in this research. Let $X$ denote a Borel $G$-space.
\begin{enumerate}[(A)]
	\item\label{q:compressibility=i-compressibility} Is $X$ being compressible equivalent to $X$ being $i$-compressible for some $i \ge 1$?
	
	\item\label{q:traveling=locally-finitely-traveling} Does the existence of a Borel traveling complete section imply the existence of a Borel locally finitely traveling complete section?
	
	\item Can we get a $2$-generator instead of a $32$-generator in \cref{sigma-compact compressible => finite generator}?
\end{enumerate}

A positive answer to any of Questions \labelcref{q:compressibility=i-compressibility} and \labelcref{q:traveling=locally-finitely-traveling} would imply a positive answer to \cref{Weiss's question} since \labelcref{q:compressibility=i-compressibility} is just a rephrasing of \cref{Weiss's question} because of \cref{finite generator <=> i-compressibility}, and for \labelcref{q:traveling=locally-finitely-traveling}, it follows from \cref{i-compressibility <=> i-traveling,lfin traveling => 2^5-generator}.

\subsection*{Acknowledgements}
First and foremost, I thank my advisor Alexander Kechris for his help, support and encouragement, and for suggesting the problems and guiding me throughout the research. Many thanks to \Slawek Solecki for his positive feedback and valuable remarks, and especially for pointing out a nice way of thinking about the notion of $i$-equidecomposability. I also thank Benjamin Weiss, Mahendra Nadkarni and Shashi Srivastava for their positive feedback, useful remarks and corrections, as well as for pointing out that an instance of $\mathbf{\Pi_1^1}$-reflection in one of the proofs could be replaced with analytic separation. Many thanks to the UCLA and Caltech logic groups for running a series of seminars in which I presented my work. I also thank Patrick Allen, Ben Miller, Yiannis Moschovakis, Itay Neeman, Justin Palumbo, Todor Tsankov, and Robin Tucker-Drob for useful conversations and comments. Finally, I thank the referee for very helpful remarks, suggestions and corrections, which led to improvements in many places.

\section{The theory of $i$-compressibility}

Throughout this section, let $X$ be a Borel $G$-space and let $E_G$ denote the orbit equivalence relation on $X$ induced by the action of $G$. For a set $A \subseteq X$ and $G$-invariant set $P \subseteq X$, put $A^P := A \cap P$.

For an equivalence relation $E$ on $X$ and $A \subseteq X$, let $[A]_E$ denote the saturation of $A$ with respect to $E$, i.e. $[A]_E = \set{x \in X : \exists y \in A (x E y)}$. In case $E = E_G$, we use $[A]_G$ instead of $[A]_{E_G}$.

Let $\Bfrak$ denote the (proper) class of all Borel subsets of standard Borel spaces, i.e.
$$
\Bfrak = \set{B : \text{$B$ is a Borel subset of some standard Borel space $X$}}.
$$
Also, let $\Ga$ be a class $\sigma$-algebra of subsets of standard Borel spaces containing $\Bfrak$ and closed under Borel preimages, i.e. if $X,Y$ are standard Borel spaces and $f : X \to Y$ is a Borel map, then for a subset $A \subseteq Y$, if $A \in \Ga$ then $f^{-1}(A)$ is also in $\Ga$. For example, $\Ga = \Bfrak$, $\s$, universally measurable sets.

For a standard Borel space $X$, let $\Ga(X)$ denote the set of all subsets of $X$ that belong to $\Ga$. In particular, $\Bfrak(X)$ denotes the set of all Borel subsets of $X$.

\subsection{The notion of $\I$-equidecomposability}\label{subsection_I-equidec}

A countable partition of $X$ is called Borel if all the sets in it are Borel. For a finite Borel partition $\I = \{A_i : i < k\}$ of $X$, let $\EI$ denote the equivalence relation of not being separated by $G  \I := \set{g A_i : g \in G, i < k}$, more precisely, $\forall x,y\in X$,
$$
x \EI y \Leftrightarrow \f(x) = \f(y),
$$
where $\f$ is the symbolic representation map for $(X, G, \I)$ defined above. Note that if $\I$ is a generator, then $\EI$ is just the equality relation.

For $A \subseteq X$, put
$$
\Ga(X) \rest{A} = \set{A' \subseteq A : \exists B \in \Ga(X) \ (A' = B \cap A)}.
$$
Also, for an equivalence relation $E$ on $X$ and $A,B \subseteq X$, say that $A$ is $E$-invariant relative to $B$ or just $E \rest{B}$-invariant if $[A]_{E} \cap B = A \cap B$.

\begin{defn}[$\I$-equidecomposability]\label{I-equidecomposability}
	Let $A,B \subseteq X$, and $\I$ be a finite Borel partition of $X$. $A$ and $B$ are said to be equidecomposable with $\Ga$ pieces (denote by $A \sim^{\Ga} B$) if there are $\{g_n\}_{n \in \N} \subseteq G$ and partitions $\{A_n\}_{n \in \N}$ and $\{B_n\}_{n \in \N}$ of $A$ and $B$, respectively, such that for all $n \in \N$
	\begin{itemize}
		\item $g_n A_n = B_n$,
		\item $A_n \in \Ga(X) \rest{A}$ and $B_n \in \Ga(X) \rest{B}$.
	\end{itemize}
	If moreover,
	\begin{itemize}
		\item $A_n$ and $B_n$ are $\EI$-invariant relative to $A$ and $B$, respectively,
	\end{itemize}
	then we will say that $A$ and $B$ are $\I$-equidecomposable with $\Ga$ pieces and denote it by $A \sim_{\I}^{\Ga} B$. If $\Ga = \Bfrak$, we will not mention $\Ga$ and will just write $\sim$ and $\sim_{\I}$.
\end{defn}

Note that for any finite Borel partition $\I$ of $X$ and Borel sets $A, B \subseteq X$, $A$ and $B$ are $\I$-equidecomposable if and only if $\f(A)$ and $\f(B)$ are equidecomposable (although the images of Borel sets under $\f$ are analytic, they are Borel relative to $\f(X)$ due to the Luzin Separation Theorem for analytic sets, see \cite[14.7]{bible}). Also note that if $\I$ is a generator, then $\sim_{\I}$ coincides with $\sim$.

\begin{obs}\label{properties of I-equidecomposability}
	Below let $\I,\I_0,\I_1$ denote finite Borel partitions of $X$, and $A,B,C \in \Ga(X)$.
	\begin{enumerate}[(a)]
		\item (Quasi-transitivity) If $A \sim_{\I_0}^{\Ga} B \sim_{\I_1}^{\Ga} C$, then $A \sim_{\I}^{\Ga} C$ with $\I = \I_0 \vee \I_1$ (the least common refinement of $\I_0$ and $\I_1$).
		
		\item ($\EI$-disjoint countable additivity) Let $\{A_n\}_{n \in \N}, \{B_n\}_{n \in \N}$ be partitions of $A$ and $B$, respectively, into $\Ga$ sets such that $\forall n \neq m$, $[A_n]_{\EI} \cap [A_m]_{\EI} = [B_n]_{\EI} \cap [B_m]_{\EI} = \0$. If $\forall n\in \N$, $A_n \sim_{\I}^{\Ga} B_n$, then $A \sim_{I}^{\Ga} B$.
	\end{enumerate}
\end{obs}

If $A \sim B$, then there is a Borel isomorphism $\phi$ of $A$ onto $B$ with $\phi(x) E_G x$ for all $x \in A$; namely $\phi(x) = g_n x$ for all $x \in A_n$, where $A_n, g_n$ are as in \cref{defn of equidec}. It is easy to see that the converse is also true, i.e. if such $\phi$ exists, then $A \sim B$. In \cref{witnessing map} we prove the analogue of this for $\sim_{\I}^{\Ga}$, but first we need the following lemma and definition that take care of definability and $\EI$-invariance, respectively.

\medskip

For a Polish space $Y$, $f : X \to Y$ is said to be $\Ga$-measurable if the preimages of open sets under $f$ are in $\Ga$. For $A \in \Ga(X)$ and $h : A \rightarrow G$, define $\hat{h} : A \rightarrow X$ by $x \mapsto h(x)x$.

\begin{lemma}\label{measurability of roadmap}
	If $h : A \rightarrow G$ is $\Ga$-measurable, then the images and preimages of sets in $\Ga$ under $\hat{h}$ are in $\Ga$.
\end{lemma}
\begin{proof}
	Let $B \subseteq A$, $C \subseteq X$ be in $\Ga$. For $g \in G$, set $A_g = h^{-1}(g)$ and note that $\hat{h}(B) = \bigcup_{g \in G} g(A_g \cap B)$ and $\hat{h}^{-1}(C) = \bigcup_{g \in G} g^{-1}(gA_g \cap C)$. Thus $\hat{h}(B)$ and $\hat{h}^{-1}(C)$ are in $\Ga$ by the assumptions on $\Ga$.
\end{proof}

The following technical definition is needed in the proofs of \cref{witnessing map,orbit-disjoint ctbl unions}.
\begin{defn}\label{defn of sensitivity}
	For $A \subseteq X$ and a finite Borel partition $\I$ of $X$, we say that $\I$ is $A$-sensitive or that $A$ respects $\I$ if $A$ is $\EI$-invariant relative to $[A]_G$, i.e. $[A]_{\EI}^{[A]_G} = A$.
\end{defn}

For example, if $\I$ is finer than $\{A, A^c\}$, then $\I$ is $A$-sensitive. Note that if $A \sim_{\I} B$ and $A$ respects $\I$, then so does $B$.

\begin{prop}\label{witnessing map}
	Let $A,B \in \Ga(X)$ and let $\I$ be a Borel partition of $X$ that is $A$-sensitive. Then, $A \sim_{\I}^{\Ga} B$ if and only if there is an $\EI$-invariant $\Ga$-measurable map $\ga : A \to G$ such that $\hat{\ga}$ is a bijection between $A$ and $B$. We refer to such $\ga$ as a witnessing map for $A \sim_{\I}^{\Ga} B$. The same holds if we delete ``$\EI$-invariant'' and ``$\I$'' from the statement.
\end{prop}
\begin{proof}
	\noindent $\Rightarrow$: If $\{g_n\}_{n \in \N}$, $\{A_n\}_{n \in \N}$ and $\{B_n\}_{n \in \N}$ are as in \cref{I-equidecomposability}, then define $\ga : A \rightarrow G$ by setting $\ga \rest{A_n} \equiv g_n$.
	
	\noindent $\Leftarrow$: Let $\ga$ be as in the lemma. Fixing an enumeration $\{g_n\}_{n \in \N}$ of $G$ with no repetitions, put $A_n = \ga^{-1}(g_n)$ and $B_n = g_n A_n$. It is clear that $\{A_n\}_{n \in \N}, \{B_n\}_{n \in \N}$ are partitions of $A$ and $B$, respectively, into $\Ga$ sets. Since $\ga$ is $\EI$-invariant, each $A_n$ is $\EI$-invariant relative to $A$ and hence relative to $P := [A]_G = [B]_G$ because $A$ respects $\I$. It remains to show that each $B_n$ is $\EI$-invariant relative to $B$. To this end, let $y \in [B_n]_{\EI} \cap B$ and thus there is $x \in A_n$ such that $y \EI g_n x$. Hence $z := g_n^{-1} y \ \EI \ g_n^{-1} g_n x = x$ and therefore $z \in A_n$ because $A_n$ is $\EI$-invariant relative to $P$. Thus $y = g_n z \in B_n$.
\end{proof}

In the rest of the subsection we work with $\Ga = \Bfrak$.

Next we prove that $\I$-equidecomposability can be extended to $\EI$-invariant Borel sets. First we need the following separation lemma for analytic sets\footnote{My original argument used $\Pi_1^1$ reflection principles, but it was pointed out to me by Shashi Srivastava that one could use analytic separation instead. I chose to present this latter argument here since analytic separation may be more transparent for non-logicians than $\Pi_1^1$ reflection principles.}:
\begin{lemma}[Invariant analytic separation]\label{analytic_separation}
	Let $E$ be an analytic equivalence relation on $X$. For any disjoint family $\set{A_n}_{n \in \N}$ of $E$-invariant analytic sets, there exists a disjoint family $\set{B_n}_{n \in \N}$ of $E$-invariant Borel sets such that $A_n \subseteq B_n$.
\end{lemma}
\begin{proof}[Proof \emph{(Vaught)}]
	We give the proof for two disjoint $E$-invariant analytic sets $A_0, A_1$ since this easily implies the statement for countably many. Recursively define analytic sets $C_n \subseteq X$ and Borel sets $D_n \subseteq X$ such that for every $n \in \N$ we have
	\begin{enumerate}[(i)]
		\item $A_0 \subseteq C_n \subseteq D_n \subseteq C_{n+1} \subseteq A_1^c$,
		\item $C_n$ is $E$-invariant.
	\end{enumerate}
	To do this, let $C_0 = A_0$, and, assuming that $C_n$ is defined, define $D_n, C_{n+1}$ as follows: since $C_n$ and $A_1$ are disjoint analytic sets, there is a Borel set $D_n$ separating them (by the Luzin separation theorem), i.e. $D_n \supseteq C_n$ and $D_n \cap A_1 = \0$. Let $C_{n+1} = [D_n]_E$, and note that $C_{n+1}$ is analytic and disjoint from $A_1$ since $A_1$ is $E$-invariant and disjoint from $D_n$. This finishes the construction.
	
	Now let $B = \Uni D_n$; hence $B$ is Borel, contains $A_0$ and is disjoint from $A_1$. On the other hand, $B = \Uni C_n$ and thus is $E$-invariant.
\end{proof}

\begin{prop}[$\EI$-invariant extensions]\label{saturated supersets}
	Let $\I$ be a Borel partition of $X$ and let $A,B \subseteq X$ be Borel sets with $[A]_\EI \cap [B]_\EI = \0$. If $A \sim_{\I} B$, then there exists Borel sets $A' \supseteq A$ and $B' \supseteq B$ such that $A',B'$ are $\EI$-invariant and $A' \sim_{\I} B'$. In fact, if $\{g_n\}_{n \in \N}, \{A_n\}_{n \in \N}, \{B_n\}_{n \in \N}$ witness $A \sim_{\I} B$, then there are $\EI$-invariant Borel partitions $\{A'_n\}_{n \in \N}, \{B'_n\}_{n \in \N}$ of $A'$ and $B'$ respectively, such that $g_n A'_n = B'_n$ and $A'_n \supseteq A_n$ (and hence $B'_n \supseteq B_n$).
\end{prop}
\begin{proof}
	Let $\{g_n\}_{n \in \N}, \{A_n\}_{n \in \N}, \{B_n\}_{n \in \N}$ be as in \cref{I-equidecomposability} and put $\cl{A}_n = [A_n]_{\EI}$. It is easy to see that for $n \neq m \in \N$,
	\begin{enumerate}[(i)]
		\item $\cl{A}_n \cap \cl{A}_m = \0$;
		\item $g_n \cl{A}_n \cap g_m \cl{A}_m = \0$.
	\end{enumerate}
	
	Put $\cl{A} = [A]_{\EI}$ and note that $\{\cl{A}_n\}_{n \in \N}$ is a partition of $\cl{A}$. Although $\cl{A}_n$ and $\cl{A}$ are $\EI$-invariant, they are analytic and in general not Borel. We obtain Borel analogues of these sets using invariant analytic separation as follows: \cref{analytic_separation} applied to $\set{A_n}_{n \in \N}$ and $\set{g_n A_n}_{n \in \N}$ (separately), gives us sequences $\set{C_n}_{n \in \N}$ and $\set{D_n}_{n \in \N}$ of $\EI$-invariant Borel sets such that $C_n \supseteq A_n$, $D_n \supseteq g_n A_n$ and $C_n \cap C_m = D_n \cap D_m = \0$ for $n \ne m$. Taking $A'_n = C_n \cap g_n^{-1} D_n$, we see that $\set{A'_n}_{n \in \N}$ is a pairwise disjoint family of $\EI$-invariant Borel sets such that $A'_n \supseteq A_n$. Moreover, $\set{g_n A'_n}_{n \in \N}$ is also a pairwise disjoint family. Thus, taking $B'_n = g_n A'_n$, we are done.
\end{proof}

\begin{lemma}[Orbit-disjoint unions]\label{orbit-disjoint union}
	Let $A_k,B_k \in \Bfrak(X)$, $k=0,1$, be such that $[A_0]_G$ and $[A_1]_G$ are disjoint and put $A = A_0 \cup A_1$ and $B = B_0 \cup B_1$. If $\I$ is an $A,B$-sensitive finite Borel partition of $X$ such that $A_k \sim_{\I} B_k$ for $k=0,1$, then $A \sim_{\I} B$. Moreover, if $\ga_0 : A_0 \rightarrow G$ is a Borel map witnessing $A_0 \sim_{\I} B_0$, then there exists a Borel map $\ga : A \rightarrow G$ extending $\ga_0$ that witnesses $A \sim_{\I} B$.
\end{lemma}
\begin{proof}
	First assume without loss of generality that $X = [A]_G$ ($= [B]_G$) since the statement of the lemma is relative to $[A]_G$. Thus $A,B$ are $\EI$-invariant.
	
	Applying \cref{saturated supersets} to $A_0 \sim_{\I} B_0$, we get $\EI$-invariant $A'_0 \supseteq A_0, B'_0 \supseteq B_0$ such that $A' \sim_{\I} B'$. Moreover, by the second part of the same lemma, if $\ga_0 : A_0 \rightarrow G$ is a witnessing map for $A_0 \sim_{\I} B_0$, then there is a witnessing map $\de : A'_0 \rightarrow G$ for $A' \sim_{\I} B'$ extending $\ga_0$. Put $C = A'_0 \cap A$ and note that $C$ is $\EI$-invariant since so are $A'_0$ and $A$. Finally, put $\cl{A}_0 = \{x \in C : C^{[x]_G} = A^{[x]_G} \wedge \hat{\de}(C^{[x]_G}) = B^{[x]_G}\}$ and note that $\cl{A}_0 \supseteq A_0$ since $\de \supseteq \ga_0$ and $[A_0]_G \cap [A_1]_G = \0$.
	
	\begin{claim*}
		$\cl{A}_0$ is $\EI$-invariant.
	\end{claim*}
	\begin{pfof}
		First note that for any $\EI$-invariant $D \subseteq X$ and $z \in X$, $[D^{[z]_G}]_{\EI} = D^{[[z]_{\EI}]_G}$. Furthermore, if $D \subseteq C$, then $[\hat{\de}(D)]_{\EI} = \hat{\de}([D]_{\EI})$ since $\hat{\de}$ and its inverse map $\EI$-invariant sets to $\EI$-invariant sets.
		
		Now take $x \in \cl{A}_0$ and let $Q = [[x]_{\EI}]_G$. Since $A, B, C$ are $\EI$-invariant, $C^Q = [C^{[x]_G}]_{\EI} = [A^{[x]_G}]_{\EI} = A^Q$. Furthermore, $\hat{\de}(C^Q) = \hat{\de}([C^{[x]_G}]_{\EI}) = [\hat{\de}(C^{[x]_G})]_{\EI} = [B^{[x]_G}]_{\EI} = B^Q$. Thus, $\forall y \in [x]_{\EI}$, $C^{[y]_G} = A^{[y]_G}$ and $\hat{\de}(C^{[y]_G}) = B^{[y]_G}$; hence $[x]_{\EI} \subseteq \cl{A}_0$.
	\end{pfof}
	
	Put $\cl{A}_1 = A \setminus \cl{A}_0$, $\alpha_0 = \de \rest{\cl{A}_0}$, $\alpha_1 = \ga_1 \rest{\cl{A}_1}$, where $\ga_1$ is a witnessing map for $A_1 \sim_{\I} B_1$. It is clear from the definition of $\cl{A}_0$ that $\cl{A}_0$ is $E_G$-invariant relative to $A$ and hence $[\cl{A}_0]_G \cap [\cl{A}_1]_G = \0$. Thus, for $k=0,1$, it follows that $\alpha_k$ witnesses $\cl{A}_k \sim_{\I} \cl{B}_k$, where $\cl{B}_k = \hat{\alpha_k}(\cl{A}_k)$. Furthermore, it is clear that $B^{[\cl{A}_k]_G} = \cl{B}_k$ and, since $[\cl{A}_0]_G \cup [\cl{A}_1]_G = X$, $\cl{B}_0 \cup \cl{B}_1 = B$. Now since $\cl{A}_k$ are $\EI$-invariant, $\ga = \alpha_0 \cup \alpha_1$ is $\EI$-invariant and hence witnesses $A \sim_{\I} B$. Finally, $\alpha_0 \rest{A_0} = \de \rest{A_0} = \ga_0$ and hence $\alpha_0 \supseteq \ga_0$.
\end{proof}

\begin{prop}[Orbit-disjoint countable unions]\label{orbit-disjoint ctbl unions}
	For $k \in \N$, let $A_k,B_k \in \Bfrak(X)$ be such that $[A_k]_G$ are disjoint and put $A = \bigcup_{k \in \N} A_k$, $B = \bigcup_{k \in \N} B_k$. Suppose that $\I$ is an $A,B$-sensitive finite Borel partition of $X$ such that $A_k \sim_{\I} B_k$ for all $k$. Then $A \sim_{\I} B$.
\end{prop}
\begin{proof}
	We recursively apply \cref{orbit-disjoint union} as follows. Put $\cl{A}_n = \bigcup_{k \leq n} A_k$ and $\cl{B}_n = \bigcup_{k \leq n} B_k$. Inductively define Borel maps $\ga_n : \bigcup_{k \leq n} A_k \rightarrow G$ such that $\ga_n$ is a witnessing map for $\cl{A}_n \sim_{\I} \cl{B}_n$ and $\ga_n \sqsubseteq \ga_{n+1}$. Let $\ga_0$ be a witnessing map for $A_0 \sim_{\I} B_0$. Assume $\ga_n$ is defined. Then $\ga_{n+1}$ is provided by \cref{orbit-disjoint union} applied to $\cl{A}_n$ and $A_{n+1}$ with $\ga_n$ as a witness for $\cl{A}_n \sim_{\I} \cl{B}_n$. Thus $\ga_n \sqsubseteq \ga_{n+1}$ and $\ga_{n+1}$ witnesses $\cl{A}_{n+1} \sim_{\I} \cl{B}_{n+1}$.
	
	Now it just remains to show that $\ga := \Uni \ga_n$ is $\EI$-invariant since then it follows that $\ga$ witnesses $A \sim_{\I} B$. Let $x,y \in A$ be $\EI$-equivalent. Then there is $n$ such that $x,y \in \cl{A}_n$. By induction on $n$, $\ga_n$ is $\EI$-invariant and, since $\ga \rest{\cl{A}_n} = \ga_n$, $\ga(x) = \ga(y)$.
\end{proof}

\begin{cor}[Finite quasi-additivity]\label{quasi-additivity}
	For $k=0,1$, let $A_k,B_k \in \Bfrak(X)$ be such that $A_0 \cap A_1 = B_0 \cap B_1 = \0$ and put $A = A_0 \cup A_1$, $B = B_0 \cup B_1$. Let $\I_k$ be an $A_k,B_k$-sensitive finite Borel partition of $X$. If $A_0 \sim_{\I_0} B_0$ and $A_1 \sim_{\I_1} B_1$, then $A \sim_{\I_0 \vee \I_1} B$.
\end{cor}
\begin{proof}
	Put $\I = \I_0 \vee \I_1$, $P = [A_0]_G \cap [A_1]_G$, $Q = [A_0]_G \setminus [A_1]_G$ and $R = [A_1]_G \setminus [A_0]_G$. Then $A_k^P, B_k^P$ respect $\I$, and thus $[A_0]_{\EI}^P \cap [A_1]_{\EI}^P = \0$, $[B_0]_{\EI}^P \cap [B_1]_{\EI}^P = \0$. Hence $A^P \sim_{\I} B^P$ since the sets that are $\EI$-invariant relative to $A_k^P$ are also $\EI$-invariant relative to $A^P$, and the same is true for $B_k^P$ and $B^P$. Also, $A^Q \sim_{\I} B^Q$ and $A^R \sim_{\I} B^R$ because $A^Q = A_0$, $B^Q = B_0$, $A^R = A_1$, $B^R = B_1$. Now since $P,Q,R$ are pairwise disjoint, it follows from \cref{orbit-disjoint ctbl unions} that $A \sim_{\I} B$.
\end{proof}

\subsection{The notion of $i$-compressibility}\label{subsection_i-compressibility}

For a finite collection $\F$ of subsets of $X$, let $\gen{\F}$ denote the partition of $X$ generated by $\F$.

\begin{defn}[$i$-equidecomposability]\label{i-equidecomposability}
	For $i \ge 1$, $A,B \subseteq X$, we say that $A$ and $B$ are $i$-equidecomposable with $\Ga$ pieces (write $A \sim_i^{\Ga} B$) if there is an $A$-sensitive partition $\I$ of $X$ generated by $i$ Borel sets such that $A \sim_{\I}^{\Ga} B$ (in particular, $\I$ must also be $B$-sensitive). For a collection $\F$ of Borel sets, we say that $\F$ witnesses $A \sim_i^{\Ga} B$ if $|\F| = i$, $\I := \gen{\F}$ is $A$-sensitive and $A \sim_{\I}^{\Ga} B$.
\end{defn}

\begin{remark}
	In the above definition, it might seem more natural to have $i$ be the cardinality of the partition $\I$ instead of the cardinality of the collection $\F$ generating $\I$. However, our definition above of $i$-equidecomposability is needed in order to show that the collection $\Cfrak_i$ defined below forms a $\sigma$-ideal. More precisely, the presence of $\F$ is needed in the definition of $i^*$-compressibility, which ensures that the partition $\I$ in the proof of \cref{C_i is a sigma-ideal} is $B$-sensitive.
\end{remark}

For a family $\F$ of subsets of $X$, let $\sigma_G(\F)$ denote the $\sigma$-algebra generated by $G \F$.

\begin{remark}
	\Slawomir Solecki pointed out that for $i \ge 1$ and Borel sets $A,B \subseteq X$, $A \sim_i B$ if and only if $A \sim B$ and the partitions $\set{A_n}_{n \in \N}$, $\set{B_n}_{n \in \N}$ witnessing the equidecomposability of $A$ and $B$ can be taken from a $\sigma$-algebra generated by the $G$-translates of $i$-many Borel sets. More precisely, $A \sim_i B$ if and only if there are a family $\F$ of $i$-many Borel sets, a sequence $\set{g_n}_{n \in \N} \subseteq G$, and partitions $\set{A_n}_{n \in \N}$ and $\set{B_n}_{n \in \N}$ of $A$ and $B$, respectively, such that $A_n,B_n \in \sigma_G (\F)$ and $g_n A_n = B_n$. Thus, $i$-equidecomposability is obtained from equidecomposability by restricting the Borel $\sigma$-algebra to \emph{some} $\sigma$-algebra generated by the $G$-translates of $i$-many Borel sets. Finally, note that every instance of $\sim_i$ uses a (potentially) different $\sigma$-algebra.
\end{remark}

For $i \ge 1$, $A,B \subseteq X$, we write $A \preceq_i^{\Ga} B$ if there is a $\Ga$ set $B' \subseteq B$ such that $A \sim_i^{\Ga} B'$. If moreover $[A \setminus B]_G = [A]_G$, then we write $A \prec_i^{\Ga} B$. If $\Ga = \Bfrak$, we simply write $\sim_i, \preceq_i, \prec_i$.

\begin{defn}[$i$-compressibility]\label{defn i-compressibility}
	For $i \in \N$, $A \subseteq X$, we say that $A$ is $i$-compressible with $\Ga$ pieces if $A \prec_i^{\Ga} A$.
\end{defn}

Unless specified otherwise, we will be working with $\Ga = \Bfrak$, in which case we simply say $i$-compressible.

For a collection of sets $\F$ and a $G$-invariant set $P$, set $\F^P = \{A^P : A \in \F\}$. We will use the following observations without mentioning.

\begin{obs}\label{properties of i-equidecomposability}
	Let $i,j \geq 2$, $A,A',B,B',C \in \Bfrak$. Let $P \subseteq [A]_G$ denote a $G$-invariant Borel set and $\F, \F_0, \F_1$ denote finite collections of Borel sets.
	\begin{enumerate}[(a)]
		\item If $A \sim_i B$ then $A^P \sim_i B^P$.
		\item If $\F$ witnesses $A \sim_i B$, then so does $\F^{[A]_G}$.
		\item If $A \sim_i B \sim_j C$, then $A \sim_{(i+j)} C$. In fact, $\F_0$ and $\F_1$ witness $A \sim_i B$ and $B \sim_j C$, respectively, then $\F = \F_0 \cup \F_1$ witnesses $A \sim_{(i+j)} C$.
		\item If $A \preceq_i B \preceq_j C$, then $A \preceq_{(i+j)} C$. If one of the first two $\preceq$ is $\prec$ then $A \prec_{(i+j)} C$.
		\item If $A \sim_i B$ and $A' \sim_j B'$ with $A \cap A' = B \cap B' = \0$, then $A \cup A' \sim_{(i+j)} B \cup B'$.
	\end{enumerate}
\end{obs}
\begin{proof}
	Part (e) follows from \cref{quasi-additivity}, and the rest follows directly from the definition of $i$-equidecomposability and \cref{properties of I-equidecomposability}.
\end{proof}

\begin{lemma}\label{A compressible => [A]_G compressible}
	If a Borel set $A \subseteq X$ is $i$-compressible, then so is $[A]_G$. In fact, if $\F$ is a finite collection of Borel sets witnessing the $i$-compressibility of $A$, then it also witnesses that of $[A]_G$.
\end{lemma}
\begin{proof}
	Let $B \subseteq A$ be a Borel set such that $[A \setminus B]_G = [A]_G$ and $A \sim_i B$. Furthermore, let $\I$ be an $A,B$-sensitive partition generated by a collection $\F$ of $i$ Borel sets such that $A \sim_{\I} B$. Let $\ga : A \rightarrow G$ be a witnessing map for $A \sim_{I} B$. Put $A' = [A]_G$, $B' = B \cup (A' \setminus A)$ and note that $A',B'$ respect $\I$. Define $\ga' : A' \rightarrow G$ by setting $\ga' \rest{A' \setminus A} = id \rest{A' \setminus A}$ and $\ga' \rest{A} = \ga$. Since $A'$ respects $\I$ and $id \rest{A' \setminus A}, \ga$ are $\EI$-invariant, $\ga'$ is $\EI$-invariant and thus clearly witnesses $A' \sim_{\I} B'$.
\end{proof}

The following is a technical refinement of the definition of $i$-compressibility that is (again) necessary for $\Cfrak_i$, defined below, to be a $\sigma$-ideal.
\begin{defn}[$i^*$-compressibility]\label{defn of i*-compressibility}
	For $i \ge 1$, we say that a Borel set $A$ is $i^*$-compressible if there is a Borel set $B \subseteq A$ such that $[A \setminus B]_G = [A]_G =: P$, $A \sim_i B$, and the latter is witnessed by a collection $\F$ of Borel sets such that $B \in \F^{P}$.
\end{defn}

Finally, for $i \geq 1$, put
$$
\Cfrak_i = \set{A \subseteq X : \text{there is a $G$-invariant Borel set $P \supseteq A$ such that $P$ is $i^*$-compressible}}.
$$

\begin{lemma}\label{i-compressible is in C(i+1)}
	Let $i \ge 1$ and $A \subseteq X$ be Borel. If $A \prec_i A$, then $A \in \Cfrak_{i+1}$.
\end{lemma}
\begin{proof}
	Setting $P = [A]_G$ and applying \cref{A compressible => [A]_G compressible}, we get that $P \prec_i P$, i.e. there is $B \subseteq P$ such that $[P \setminus B]_G = P$ and $P \sim_i B$. Let $\F$ be a collection of Borel sets witnessing the latter fact. Then $\F' = \F \cup \{B\}$ witnesses $P \sim_{(i+1)} B$ and contains $B$.
\end{proof}

\begin{prop}\label{C_i is a sigma-ideal}
	For all $i \ge 1$, $\Cfrak_i$ is a $\sigma$-ideal.
\end{prop}
\begin{proof}
	We only need to show that $\Cfrak_i$ is closed under countable unions. For this it is enough to show that if $A_n \in \Bfrak(X)$ are $i^*$-compressible $G$-invariant Borel sets, then so is $A := \Uni A_n$.
	
	We may assume that $A_n$ are pairwise disjoint since we could replace each $A_n$ by $A_n \setminus (\bigcup_{k<n} A_k)$. Let $B_n \subseteq A_n$ be a Borel set and $\F_n = \{F^n_k\}_{k<i}$ be a collection of Borel sets with $(F^n_0)^{A_n} = B_n$ such that $\F_n$ witnesses $A_n \sim_i B_n$ and $[A_n \setminus B_n]_G = A_n$. Using part (b) of \cref{properties of i-equidecomposability}, we may assume that $\F_n^{A_n} = \F_n$; in particular, $F^n_0 = B_n$.
	
	Put $B = \Uni B_n$ and $F_k = \Uni F^n_k$, $\forall k<i$; note that $F_0 = B$. Set $\F = \{F_k\}_{k<i}$ and $\I = \gen{\F}$. Since $B \in \F$ and $A$ is $G$-invariant, $\I$ is $A,B$-sensitive. Furthermore, since $\F^{A_n} = \F_n$, $A_n \sim_{\I} B_n$ for all $n \in \N$. Thus, by \cref{orbit-disjoint ctbl unions}, $A \sim_{\I} B$ and hence $A$ is $i^*$-compressible.
\end{proof}


\section{Traveling sets and finite generators}

Throughout this section, we again let $X$ be a Borel $G$-space and $\Ga$ be a $\sigma$-algebra of subsets of standard Borel spaces as in the previous section.

\subsection{Traveling and $i$-traveling sets}\label{subsection_traveling sets}

\begin{defn}\label{defn of traveling sets}
	Let $A \in \Ga(X)$.
	\begin{itemize}
		\item We call $A$ a traveling set with $\Ga$ pieces if there exists pairwise disjoint sets $\{A_n\}_{n \in \N}$ in $\Ga(X)$ such that $A_0 = A$ and $A \sim^{\Ga} A_n$, $\forall n \in \N$.
		
		\item For a finite Borel partition $\I$, we say that $A$ is $\I$-traveling with $\Ga$ pieces if $A$ respects $\I$ and the above condition holds with $\sim^{\Ga}$ replaced by $\sim_{\I}^{\Ga}$.
		
		\item For $i \ge 1$, we say that $A$ is $i$-traveling if it is $\I$-traveling for some $A$-sensitive partition $\I$ generated by a collection of $i$ Borel sets.
	\end{itemize}
\end{defn}

\begin{defn}
	For a set $A \subseteq X$, a function $\ga : A \to \GN$ is called a travel guide for $A$ if $\forall x \in A, \ga(x)(0) = 1_G$ and $\forall (x,n) \neq (y,m) \in A \times \N$, $\ga(x)(n)x \neq \ga(y)(m)y$.
\end{defn}

For $A \in \Ga(X)$, a $\Ga$-measurable map $\ga : A \rightarrow \GN$ and $n \in \N$, set $\ga_n := \ga(\cdot)(n) : A \to G$ and note that $\ga_n$ is also $\Ga$-measurable.

\begin{obs}
	Suppose $A \in \Ga(X)$ and $\I$ is an $A$-sensitive finite Borel partition of $X$. Then $A$ is $\I$-traveling with $\Ga$ pieces if and only if it has a $\Ga$-measurable $\EI$-invariant travel guide.
\end{obs}
\begin{proof}
	Follows from definitions and \cref{witnessing map}.
\end{proof}

Now we establish the connection between compressibility and traveling sets.
\begin{lemma}\label{lemma i-compressibility <=> i-traveling}
	Let $\I$ be a finite Borel partition of $X$, $P \in \Ga(X)$ be a Borel $G$-invariant set and let $A,B$ be $\Ga$ subsets of $P$. If $P \sim_{\I}^{\Ga} B$, then $P \setminus B$ is $\I$-traveling with $\Ga$ pieces. Conversely, if $A$ is $\I$-traveling with $\Ga$ pieces, then $P \sim_{\I}^{\Ga} (P \setminus A)$. The same is true if we replace $\sim_{\I}^{\Ga}$ and ``$\I$-traveling'' with $\sim^{\Ga}$ and ``traveling'', respectively.
\end{lemma}
\begin{proof}
	For the first statement, let $\ga : P \to G$ be a witnessing map for $P \sim_{\I}^{\Ga} B$. Put $A' = P \setminus B$ and note that $A'$ respects $\I$ since so does $P$ and hence $B$. We show that $A'$ is $\I$-traveling. Put $A_n = (\hat{\ga})^n(A')$, for each $n \ge 0$. It follows from injectivity of $\hat{\ga}$ that the $A_n$ are pairwise disjoint. For $n \in \N$, recursively define $\de_n : A' \to G$ as follows
	$$
	\left\{ \begin{array}{l}
	\de_0 = \ga \rest{A'} \\
	\de_{n+1} = \ga \circ \hat{\de}_n
	\end{array}\right..
	$$
	It follows from $\EI$-invariance of $\ga$ that each $\de_n$ is $\EI$-invariant. It is also clear that $\hat{\de}_n = (\hat{\ga})^n$ and hence $\de_n$ is a witnessing map for $A' \sim_{\I}^{\Ga} A_n$. Thus $A'$ is $i$-traveling with $\Ga$ pieces.
	
	For the converse, assume that $A$ is $\I$-traveling and let $\{A_n\}_{n \in \N}$ be as in \cref{defn of traveling sets}. In particular, each $A_n$ respects $\I$ and $A_n \sim_{\I}^{\Ga} A_m$, for all $n,m \in \N$. Let $P' = \Uni A_n$ and $B' = \bigcup_{n \ge 1} A_n$. Since $A_n \sim_{\I}^{\Ga} A_{n+1}$, part (b) of \cref{properties of I-equidecomposability} implies that $P' \sim_{\I}^{\Ga} B'$. Moreover, since $P \setminus P' \sim_{\I}^{\Ga} P \setminus P'$, we get $P \sim_{\I}^{\Ga} (B' \cup (P \setminus P')) = P \setminus A$.
\end{proof}

For a $G$-invariant set $P$ and $A \subseteq P$, we say that $A$ is a complete section for $P$ if $[A]_G = P$. The above lemma immediately implies the following.
\begin{prop}\label{i-compressibility <=> i-traveling}
	Let $P \in \Ga(X)$ be $G$-invariant and $i \ge 1$. $P$ is $i$-compressible with $\Ga$ pieces if and only if there exists a complete section for $P$ that is $i$-traveling with $\Ga$ pieces. The same is true with ``$i$-compressible'' and ``$i$-traveling'' replaced by ``compressible'' and ``traveling''.
\end{prop}

We need the following lemma in the proofs of \cref{equivalences to compressibility,equivalences to i-compressibility}.
\begin{lemma}\label{reflection of not having measure}
	Suppose $A \subseteq X$ is an invariant analytic set that does not admit an invariant Borel probability measure. Then there is an invariant Borel set $A' \supseteq A$ that still does not admit an invariant Borel probability measure.
\end{lemma}
\begin{proof}
	Let $\M$ denote the standard Borel space of $G$-invariant Borel probability measures on $X$ (see \cite[Section 17]{bible}). Let $\Phi \subseteq Pow(X)$ be the following predicate:
	$$
	\Phi(W) \Leftrightarrow \forall \mu \in \M (\mu(W) = 0).
	$$
	\begin{claim*}
		There is a Borel set $B \supseteq A$ with $\Phi(B)$.
	\end{claim*}
	\begin{pfof}
		By the dual form of the First Reflection Theorem for $\mathbf{\Pi}_1^1$ (see \cite[the discussion following 35.10]{bible}), it is enough to show that $\Phi$ is $\mathbf{\Pi}_1^1$ on $\mathbf{\Sigma}_1^1$. To this end, let $Y$ be a Polish space and $D \subseteq Y \times X$ be analytic. Then, for any $n \in \N$, the set
		$$
		H_n = \set{(\mu, y) \in \M \times Y : \mu(D_y) > {1 \over n}},
		$$
		is analytic by a theorem of Kond\^{o}--Tugu\'{e} (see \cite[29.26]{bible}), and hence so are the sets $H'_n := \proj{Y}(H_n)$ and $H := \Uni H'_n$. Finally, note that
		$$
		\set{y \in Y : \Phi(A_y)} = \set{y \in Y : \exists \mu \in \M \exists n \in \N (\mu(A_y) > {1 \over n})}^c = H^c,
		$$
		and so $\set{y \in Y : \Phi(A_y)}$ is $\mathbf{\Pi}_1^1$.
	\end{pfof}
	
	Now put $A' = (B)_G$, where $(B)_G = \set{x \in B : [x]_G \subseteq B}$. Clearly, $A'$ is an invariant Borel set, $A' \supseteq A$, and $\Phi(A')$ since $A' \subseteq B$ and $\Phi(B)$.
\end{proof}

\begin{prop}\label{equivalences to compressibility}
	Let $X$ be a Borel $G$-space. The following are equivalent:
	\begin{enumerate}[(1)]
		\item $X$ is compressible with universally measurable pieces;
		\item There is a universally measurable complete section that is a traveling set with universally measurable pieces;
		\item There is no $G$-invariant Borel probability measure on $X$;
		\item $X$ is compressible with Borel pieces;
		\item There is a Borel complete section that is a traveling set with Borel pieces.
	\end{enumerate}
\end{prop}
\begin{proof}
	Equivalence of (1) and (2) as well as (4) and (5) is asserted in \cref{i-compressibility <=> i-traveling}, (4)$\Rightarrow$(1) is trivial, and (3)$\Rightarrow$(4) follows from Nadkarni's theorem (see \labelcref{Nadkarni's thm}). It remains to show (1)$\Rightarrow$(3). To this end, suppose $X \sim^{\Ga} B$, where $B^c = X \setminus B$ is a complete section and $\Ga$ is the class of universally measurable sets. If there was a $G$-invariant Borel probability measure $\mu$ on $X$, then $\mu(X) = \mu(B)$ and hence $\mu(B^c) = 0$. But since $B^c$ is a complete section, $X = \bigcup_{g \in G} gB^c$, and thus $\mu(X) = 0$, a contradiction.
\end{proof}

Now we prove an analogue of this for $i$-compressibility.

\begin{prop}\label{equivalences to i-compressibility}
	Let $X$ be a Borel $G$-space. For $i \ge 1$, the following are equivalent:
	\begin{enumerate}[(1)]
		\item $X$ is $i$-compressible with universally measurable pieces;
		\item There is a universally measurable complete section that is an $i$-traveling set with universally measurable pieces;
		\item There is a partition $\I$ of $X$ generated by $i$ Borel sets such that $Y = \f(X) \subseteq |\I|^G$ does not admit a $G$-invariant Borel probability measure;
		\item $X$ is $i$-compressible with Borel pieces;
		\item There is a Borel complete section that is an $i$-traveling set with Borel pieces.
	\end{enumerate}
\end{prop}
\begin{proof}
	Equivalence of (1) and (2) as well as (4) and (5) is asserted in \cref{i-compressibility <=> i-traveling} and (4)$\Rightarrow$(1) is trivial. It remains to show (1)$\Rightarrow$(3)$\Rightarrow$(5).
	
	\noindent (1)$\Rightarrow$(3): Suppose $X \sim_{\I}^{\Ga} B$, where $B^c = X \setminus B$ is a complete section, $\I$ is a partition of $X$ generated by $i$ Borel sets, and $\Ga$ denotes the class of universally measurable sets. Let $\ga : X \to G$ be a witnessing map for $X \sim_i^{\Ga} B$. By the Jankov-von Neumann uniformization theorem (see \cite[18.1]{bible}), $\f$ has a $\s$-measurable (hence universally measurable) right inverse $h : Y \to X$. Define $\de : Y \to G$ by $\de(y) = \ga(h(y))$ and note that $\de$ is universally measurable being a composition of such functions. Letting $B' = \hat{\de}(Y)$, it is straightforward to check that $\hat{\de} \circ \f = \f \circ \hat{\ga}$ and thus $B' = \f(\hat{\ga}(X)) = \f(B)$. Now it follows that $\de$ is a witnessing map for $Y \sim^{\Ga} B'$ and hence $Y$ is compressible with universally measurable pieces. Finally, (1)$\Rightarrow$(3) of \cref{equivalences to compressibility} implies that $Y$ does not admit an invariant Borel probability measure.
	
	\noindent (3)$\Rightarrow$(5): Assume $Y$ is as in (3). Then by \cref{reflection of not having measure}, there is a Borel $G$-invariant $Y' \supseteq Y$ that does not admit a $G$-invariant Borel probability measure. Viewing $Y'$ as a Borel $G$-space, we apply (3)$\Rightarrow$(4) of \cref{equivalences to compressibility} and get that $Y'$ is compressible with Borel pieces; thus there is a Borel $B' \subseteq Y'$ with $[Y' \setminus B']_G = Y'$ such that $Y' \sim B'$. Let $\de : Y' \to G$ be a witnessing map for $Y' \sim B'$. Put $B = \f^{-1}(B')$ and $\ga = \de \circ \f$. By definition, $\ga$ is $\EI$-invariant. In fact, it is straightforward to check that $\ga$ is a witnessing map for $X \sim_{\I} B$ and $[X \setminus B]_G = [\f^{-1}(Y \setminus B')]_G = \f^{-1}([Y \setminus B']_G) = \f^{-1}(Y) = X$. Hence $X$ is $\I$-compressible.
\end{proof}

\medskip

We now give an example of a $1$-traveling set. First we need some definitions.

\begin{defn}\label{defn:smoothness_via_transversal}
	Let $X$ be a Borel $G$-space. A set $A \subseteq X$ is called
	\begin{itemize}
		\item aperiodic if it intersects every orbit in either $0$ or infinitely many points;
		
		\item a partial transversal if it intersects every orbit in at most one point;
		
		\item a transversal if it intersects every orbit in exactly one point;
		
		\item smooth if there is a Borel partial transversal $T \subseteq A$ such that $[T]_G = [A]_G$.
	\end{itemize}
	Finally, the action $G \actson X$ is called smooth if $X$ is a smooth set, i.e. admits a Borel transversal.
\end{defn}

\begin{prop}\label{transversals are 1-traveling}
	Let $X$ be an aperiodic Borel $G$-space and $T \subseteq X$ be Borel. If $T$ is a partial transversal, then $T$ is $\gen{T}$-traveling.
\end{prop}
\begin{proof}
	Let $G = \{g_n\}_{n \in \N}$ with $g_0 = 1_G$. For each $n \in \N$, define $\bar{n} : X \rightarrow \N$ and $\ga_n : T \to G$ recursively in $n$ as follows:
	$$
	\left\{ \begin{array}{l}
	\bar{n}(x) = \text{the least $k$ such that } g_k x \notin \{\hat{\ga}_i(x) : i<n\} \\
	\ga_n(x) = g_{\n(x)}
	\end{array}\right..
	$$
	Clearly, $\bar{n}$ and $\ga_n$ are well-defined and Borel. Define $\ga : T \to \GN$ by setting $\ga(\cdot)(n) = \ga_n$. It follows from the definitions that $\ga$ is a Borel travel guide for $T$ and hence, $T$ is a traveling set. It remains to show that $\ga$ is $\EI$-invariant, where $\I = \gen{T}$. For this it is enough to show that $\n$ is $\EI$-invariant, which we do by induction on $n$. Since it trivially holds for $n = 0$, we assume it is true for all $0 \le k < n$ and show it for $n$. To this end, suppose $x,y \in T$ with $x \EI y$, and assume for contradiction that $m := \n(x) < \n(y)$. Thus it follows that $g_m y = \hat{\ga}_k(y) \in \hat{\ga}_k(T)$, for some $k < n$. By the induction hypothesis, $\hat{\ga}_k(T)$ is $\EI$-invariant and hence, $g_m x \in \hat{\ga}_k(T)$, contradicting the definition of $\n(x)$.
\end{proof}

\begin{cor}\label{smooth sets are in C_1}
	Let $X$ be an aperiodic Borel $G$-space. If a Borel set $A \subseteq X$ is smooth, then $A \in \Cfrak_1$.
\end{cor}
\begin{proof}
	Let $P = [A]_G$ and let $T$ be a Borel partial transversal with $[T]_G = P$. By \cref{transversals are 1-traveling}, $T$ is $\I$-traveling, where $\I = \gen{T}$. Hence, $P \sim_{\I} P \setminus T$, by \cref{lemma i-compressibility <=> i-traveling}. This implies that $P$ is $1^*$-compressible since $\I = \gen{T^c}$ and $P \setminus T \in \set{T^c}^P$.
\end{proof}

\subsection{Constructing finite generators using $i$-traveling sets}\label{subsection_i-traveling => fin gen}

\begin{lemma}\label{I-traveling implies finite generator}
	Let $A \in \Bfrak(X)$ be a complete section and $\I$ be an $A$-sensitive finite Borel partition of $X$. If $A$ is $\I$-traveling (with Borel pieces), then there is a Borel $2|\I|$-generator. If moreover $A \in \I$, then there is a Borel $(2|\I|-1)$-generator.
\end{lemma}
\begin{proof}
	Let $\ga$ be an $\EI$-invariant Borel travel guide for $A$. Fix a countable family $\{U_n\}_{n \in \N}$ generating the Borel structure of $X$ and let $B = \Uone \hat{\ga}_n(A \cap U_n)$. By \cref{measurability of roadmap}, each $\hat{\ga}_n$ maps Borel sets to Borel sets and hence $B$ is Borel. Set $\J = \gen{B}$ , $\Pa = \I \vee \J$ and note that $|\Pa| \le 2 |\I|$. $A$ and $B$ are disjoint since $\{\hat{\ga}_n(A)\}_{n \in \N}$ is a collection of pairwise disjoint sets and $\hat{\ga}_0(A) = A$; thus if $A \in \I$, $|\Pa| \leq 1 + 2(|\I|-1) = 2|\I|-1$. We show that $\Pa$ is a generator, that is $G \Pa$ separates points in $X$.
	
	Let $x \neq y \in X$ and assume they are not separated by $G \I$, thus $x \EI y$. We show that $G \J$ separates $x$ and $y$. Because $A$ is a complete section, multiplying $x$ by an appropriate group element, we may assume that $x \in A$. Since $A$ respects $\I$, $A$ is $\EI$-invariant and thus $y \in A$. Also, because $\ga$ is $\EI$-invariant, $\ga_n(x) = \ga_n(y)$, $\forall n \in \N$. Let $n \ge 1$ be such that $x \in U_n$ but $y \notin U_n$. Put $g = \ga_n(x)$($= \ga_n(y)$). Then $gx = \hat{\ga}_n(x) \in \hat{\ga}_n(A \cap U_n)$ while $gy = \hat{\ga}_n(y) \notin \hat{\ga}_n(A \cap U_n)$. Hence, $gx \in B$ and $gy \notin B$ because $\ga_m(A) \cap \ga_n(A) = \0$ for all $m \ne n$ and $gy = \hat{\ga}_n(y) \in \hat{\ga}_n(A)$. Thus $G \J$ separates $x$ and $y$.
\end{proof}

Now \cref{equivalences to i-compressibility,I-traveling implies finite generator} together imply the following.

\begin{prop}\label{i-compressible => finite generator}
	Let $X$ be a Borel $G$-space and $i \ge 1$. If $X$ is $i$-compressible then there is a Borel $2^{i+1}$-generator.
\end{prop}
\begin{proof}
	By \cref{equivalences to i-compressibility}, there exists a Borel $i$-traveling complete section $A$. Let $\I$ witness $A$ being $i$-traveling and thus, by \cref{I-traveling implies finite generator}, there is a $2|\I| \le 2 \cdot 2^i = 2^{i+1}$-generator.
\end{proof}

\begin{example}
	For $2 \le n \le \w$, let $\Fr_n$ denote the free group on $n$ generators and let $X$ be the boundary of $\Fr_n$, i.e. the set of infinite reduced words. Clearly, the product topology makes $X$ a Polish space and $\Fr_n$ acts continuously on $X$ by left concatenation and cancellation. We show that $X$ is $1$-compressible and thus admits a Borel $2^2=4$-generator by \cref{i-compressible => finite generator}. To this end, let $a,b$ be two of the $n$ generators of $\Fr_n$ and let $X_a$ be the set of all words in $X$ that start with $a$. Then $X = (X_{a^{-1}} \cup X_{a^{-1}}^c) \sim_{\I} Y$, where $Y = bX_{a^{-1}} \cup aX_{a^{-1}}^c$ and $\I \gen{X_{a^{-1}}}$. Hence $X \sim_1 Y$. Since $X \setminus Y \supseteq X_{a^{-1}}$, $[X \setminus Y]_{\Fr_n} = X$ and thus $X$ is $1$-compressible.
\end{example}

Now we obtain a sufficient condition for the existence of an embedding into a finite Bernoulli shift.
\begin{cor}\label{k-surjection => 2k-injection}
	Let $X$ be a Borel $G$-space and $k \in \N$. If there exists a $G$-equivariant Borel map $f : X \rightarrow k^G$ such that $Y = f(X)$ does not admit a $G$-invariant Borel probability measure, then there is a $G$-equivariant Borel embedding of $X$ into $(2k)^G$.
\end{cor}
\begin{proof}
	Let $\I = \I_{f}$ and hence $f = \f$. By (3)$\Rightarrow$(5) of \cref{equivalences to i-compressibility} (or rather the proof of it), $X$ admits a Borel $\I$-traveling complete section. Thus by \cref{I-traveling implies finite generator}, $X$ admits a $2|\I| = 2k$-generator and hence, there is a $G$-equivariant Borel embedding of $X$ into $(2k)^G$.
\end{proof}

\begin{lemma}\label{I generated by log(|I|)}
	Let $\I$ be a partition of $X$ into $n$ Borel sets. Then $\I$ is generated by $k = \lceil \log_2(n) \rceil$ Borel sets.
\end{lemma}
\begin{proof}
	Since $2^k \ge n$, we can index $\I$ by the set $\mathbf{2^k}$ of all $k$-tuples of $\set{0,1}$, i.e. $\I = \{A_{\sigma}\}_{\sigma \in \mathbf{2^k}}$. For all $i < k$, put
	$$
	B_i = \bigcup_{\sigma \in \mathbf{2^k} \wedge \sigma(i) = 1} A_{\sigma}.
	$$
	Now it is clear that for all $\sigma \in \mathbf{2^k}$, $A_{\sigma} = \bigcap_{i<k} B_i^{\sigma(i)}$, where $B_i^{\sigma(i)}$ is equal to $B_i$ if $\sigma(i) = 1$, and equal to $B_i^c$, otherwise. Thus $\I = \gen{B_i : i<k}$.
\end{proof}

\begin{prop}\label{n-generator => log(n)-compressible}
	If $X$ is compressible and there is a Borel $n$-generator, then $X$ is $\lceil \log_2(n) \rceil$-compressible.
\end{prop}
\begin{proof}
	Let $\I$ be an $n$-generator and hence, by \cref{I generated by log(|I|)}, $\I$ is generated by $\lceil \log_2(n) \rceil$ Borel sets. Since $G \I$ separates points in $X$, each $\EI$-class is a singleton and hence $X \prec X$ implies $X \prec_{\I} X$.
\end{proof}

From \cref{i-compressible => finite generator,n-generator => log(n)-compressible} we immediately get the following corollary, which justifies the use of $i$-compressibility in studying \cref{Weiss's question}.

\begin{cor}\label{finite generator <=> i-compressibility}
	Let $X$ be a Borel $G$-space that is compressible (equivalently, does not admit an invariant Borel probability measure). $X$ admits a finite generator if and only if $X$ is $i$-compressible for some $i \ge 1$.
\end{cor}


\section{Finitely additive invariant measures and $i$-compressibility}\label{section_Nadkarni's proof}

This section is mainly devoted to the following theorem, together its corollaries and proof.

\begin{theorem}\label{Nadkarni for C_4}
	Let $X$ be a Borel $G$-space. If $X$ is aperiodic, then there exists a function $m : \Bfrak(X) \times X \to [0,1]$ satisfying the following properties for all $A,B \in \Bfrak(X)$:
	\begin{enumerate}[(a)]
		\item $m(A, \cdot)$ is Borel;
		\item $m(X, x) = 1$, $\forall x \in X$;
		\item If $A \subseteq B$, then $m(A, x) \le m(B,x)$, $\forall x \in X$;
		\item $m(A, x) = 0$ off $[A]_G$;
		\item $m(A, x) > 0$ on $[A]_G$ modulo $\Cfrak_4$;
		\item $m(A,x) = m(gA, x)$, for all $g \in G$, $x \in X$ modulo $\Cfrak_3$;
		\item If $A \cap B = \0$, then $m(A \cup B,x) = m(A,x) + m(B,x)$, $\forall x \in X$ modulo $\Cfrak_4$.
	\end{enumerate}
\end{theorem}

\begin{remark}
	A version of this theorem is what lies at the heart of the proof of Nadkarni's theorem. The conclusions of our theorem are modulo $\Cfrak_4$, which is potentially a smaller $\sigma$-ideal than the $\sigma$-ideal of sets contained in compressible Borel sets used in Nadkarni's version. However, the price we pay for this is that part (g) asserts only finite additivity instead of countable additivity asserted by Nadkarni's version.
\end{remark}

Before proceeding with the proof of this theorem, we draw a couple of corollaries. \cref{Nadkarni for C_4} will only be used via \cref{measures on ctbl Boolean algebras}.

\begin{defn}
	Let $X$ be a Borel $G$-space. $\BA \subseteq \Bfrak(X)$ is called a Boolean $G$-algebra, if it is a Boolean algebra, i.e. is closed under finite unions and complements, and is closed under the $G$-action, i.e. $G \BA = \BA$.
\end{defn}

\begin{cor}\label{measures on ctbl Boolean algebras}
	Let $X$ be a Borel $G$-space and let $\BA \subseteq \Bfrak(X)$ be a countable Boolean $G$-algebra. For any $A \in \BA$ with $A \notin \Cfrak_4$, there exists a $G$-invariant finitely additive probability measure $\mu$ on $\BA$ with $\mu(A)>0$. Moreover, $\mu$ can be taken such that there is $x \in A$ such that $\forall B \in \BA$ with $B \cap [x]_G = \0$, $\mu(B)=0$.
\end{cor}
\begin{proof}
	Let $A \in \BA$ be such that $A \notin \Cfrak_4$. We may assume that $X = [A]_G$ by setting the (to be constructed) measure to be $0$ outside $[A]_G$.
	
	If $X$ is not aperiodic, then by assigning equal point masses to the points of a finite orbit, we will have a probability measure on all of $\Bfrak(X)$, so assume $X$ is aperiodic.
	
	Since $\Cfrak_4$ is a $\sigma$-ideal and $\BA$ is countable, \cref{Nadkarni for C_4} implies that there is a $P \in \Cfrak_4$ such that (a)-(g) of the same theorem hold on $X \setminus P$ for all $A,B \in \BA$. Since $A \notin \Cfrak_4$, there exists $x_A \in A \setminus P$. Hence, letting $\mu(B) = m(B,x_A)$ for all $B \in \BA$, conditions (b),(f) and (g) imply that $\mu$ is a $G$-invariant finitely additive probability measure on $\BA$. Moreover, since $x_A \in [A]_G \setminus P$, $\mu(A) = m(A, x_A) > 0$. Finally, the last assertion follows from condition (d).
\end{proof}

\begin{cor}
	Let $X$ be a Borel $G$-space. For every Borel set $A \subseteq X$ with $A \notin \Cfrak_4$, there exists a $G$-invariant finitely additive Borel probability measure $\mu$ (defined on all Borel sets) with $\mu(A)>0$.
\end{cor}
\begin{proof}
	The statement follows from \cref{measures on ctbl Boolean algebras} and a standard application of the Compactness Theorem of propositional logic. Here are the details.
	
	We fix the following set of propositional variables
	$$
	\Pa = \set{P_{A,r} : A \in \Bfrak(X), r \in [0,1]},
	$$
	with the following interpretation in mind:
	$$
	P_{A,r} \Leftrightarrow \text{``the measure of $A$ is $\ge r$''}.
	$$
	Define the theory $T$ as the following set of sentences: for each $A,B \in \Bfrak(X)$, $r,s \in [0,1]$ and $g \in G$,
	\begin{enumerate}[(i)]
		\item ``$P_{A,0}$''$\in T$;
		
		\item if $r > 0$, then ``$\neg P_{\0, r}$''$\in T$;
		
		\item if $s \ge r$, then ``$P_{A,s} \to P_{A,r}$''$\in T$;
		
		\item if $A \cap B = \0$, then ``$(P_{A,r} \wedge P_{B,s}) \to P_{A \cup B, r+s}$'', ``$(\neg P_{A,r} \wedge \neg P_{B,s}) \to \neg P_{A \cup B, r+s}$''$\in T$;
		
		\item ``$P_{X,1}$''$\in T$;
		
		\item ``$P_{A,r} \to P_{gA,r}$''$\in T$.
	\end{enumerate}
	
	If there is an assignment of the variables in $\Pa$ satisfying $T$, then for each $A \in \Bfrak(X)$, we can define
	$$
	\mu(A) = \sup \set{r \in [0,1] : P_{A,r}}.
	$$
	Note that due to (i), $\mu$ is well defined for all $A \in \Bfrak(X)$. In fact, it is straightforward to check that $\mu$ is a finitely additive $G$-invariant probability measure. Thus, we only need to show that $T$ is satisfiable, for which it is enough to check that $T$ is finitely satisfiable, by the Compactness Theorem of propositional logic (or by Tychonoff's theorem).
	
	Let $T_0 \subseteq T$ be finite and let $\Pa_0$ be the set of propositional variables that appear in the sentences in $T_0$. Let $\BA$ denote the Boolean $G$-algebra generated by the sets that appear in the indices of the variables in $\Pa_0$. By \cref{measures on ctbl Boolean algebras}, there is a finitely additive $G$-invariant probability measure $\mu$ defined on $\BA$. Consider the following assignment of the variables in $\Pa_0$: for all $P_{A,r} \in \Pa_0$,
	$$
	P_{A,r} :\Leftrightarrow \mu(A) \ge r.
	$$
	It is straightforward to check that this assignment satisfies $T_0$, and hence, $T$ is finitely satisfiable.
\end{proof}

\bigskip

We now start working towards the proof of \cref{Nadkarni for C_4}, following the general outline of Nadkarni's proof of \cref{Nadkarni's thm}. The construction of $m(A,x)$ is somewhat similar to that of Haar measure. First, for sets $A,B$, we define a Borel function $[A/B] : X \to \N \cup \set{-1, \infty}$ that basically gives the number of copies of $B^{[x]_G}$ that fit in $A^{[x]_G}$ when moved by group elements (piecewise). Then we define a decreasing sequence of complete sections (called a fundamental sequence below), which serves as a gauge to measure the size of a given set.

Assume throughout that $X$ is an aperiodic Borel $G$-space (although we only use the aperiodicity assumption in \cref{fundamental sequence exists} to assert that smooth sets are in $\Cfrak_1$).

\subsection{Measuring the size of a set relative to another}

\begin{lemma}[Comparability]\label{comparability}
	$\forall A,B \in \Bfrak(X)$, there is a partition $X = P \cup Q$ into $G$-invariant Borel sets such that for any $A,B$-sensitive finite Borel partition $\I$ of $X$, $A^P \prec_{\I} B^P$ and $B^Q \preceq_{\I} A^Q$.
\end{lemma}
\begin{proof}
	It is enough to prove the lemma assuming $X = [A]_G \cap [B]_G$ since we can always include $[B]_G \setminus [A]_G$ in $P$ and $X \setminus [B]_G$ in $Q$.
	
	Fix an enumeration $\{g_n\}_{n \in \N}$ for $G$. We recursively construct Borel sets $A_n,B_n,A_n',B_n'$ as follows. Set $A_0' = A$ and $B_0' = B$. Assuming $A_n', B_n'$ are defined, set $B_n = B_n' \cap g_n A_n'$, $A_n = g_n^{-1} B_n$, $A_{n+1}' = A_n' \setminus A_n$ and $B_{n+1}' = B_n' \setminus B_n$.
	
	It is easy to see by induction on $n$ that for any $A,B$-sensitive $\I$, $A_n,B_n$ are $\EI$-invariant since so are $A,B$. Thus, setting $A^* = \Uni A_n$ and $B^* = \Uni B_n$, we get that $A^* \sim_{\I} B^*$ since $B_n = g_n A_n$.
	
	Let $A' = A \setminus A^*$, $B' = B \setminus B^*$ and set $P = [B']_G$, $Q = X \setminus P$.
	
	\begin{claim*}
		$[A']_G \cap [B']_G = \0$.
	\end{claim*}
	\begin{pfof}
		Assume for contradiction that $\exists x \in A'$ and $n \in \N$ such that $g_n x \in B'$. It is clear that $A' = \bigcap_{k \in \N} A_k'$, $B' = \bigcap_{k \in \N} B_k'$; in particular, $x \in A_n'$ and $g_n x \in B_n'$. But then $g_n x \in B_n$ and $x \in A_n$, contradicting $x \in A'$.
	\end{pfof}
	
	Let $\I$ be an $A,B$-sensitive partition. Then $A^P = (A^*)^P$ and hence $A^P \prec_{\I} B^P$ since $(A^*)^P \sim_{\I} (B^*)^P \subseteq B^P$ and $[B^P \setminus (B^*)^P]_G = [B']_G = P = [B^P]_G$. Similarly, $B^Q = (B^*)^Q$ and hence $B^Q \preceq_{\I} A^Q$ since $(B^*)^Q \sim_{\I} (A^*)^Q \subseteq A^Q$.
\end{proof}

\begin{defn}[Divisibility]\label{defn of divisibility}
	Let $n \le \w$, $A,B,C \in \Bfrak(X)$ and $\I$ be a finite Borel partition of $X$.
	\begin{itemize}
		\item Write $A \sim_{\I} nB \oplus C$ if there are Borel sets $A_k \subseteq A$, $k<n$, such that $\set{A_k}_{k<n} \cup \set{C}$ is a partition of $A$, each $A_k$ is $\EI$-invariant relative to $A$ and $A_k \sim_{\I} B$.
		
		\item Write $n B \preceq_{\I} A$ if there is $C \subseteq A$ with $A \sim_{\I} n B \oplus C$, and write $n B \prec_{\I} A$ if moreover $[C]_G = [A]_G$.
		
		\item Write $A \preceq_{\I} nB$ if there is a Borel partition $\{A_k\}_{k<n}$ of $A$ such that each $A_k$ is $\EI$-invariant relative to $A$ and $A_k \preceq_{\I} B$. If moreover, $A_k \prec_{\I} B$ for at least one $k < n$, we write $A \prec_{\I} nB$.
	\end{itemize}
	For $i \ge 1$, we use the above notation with $\I$ replaced by $i$ if there is an $A,B$-sensitive partition $\I$ generated by $i$ sets for which the above conditions hold.
\end{defn}

\begin{prop}[Euclidean decomposition]\label{Euclidean decomposition}
	Let $A,B \in \Bfrak(X)$ and put $R = [A]_G \cap [B]_G$. There exists a partition $\{P_n\}_{n \le \w}$ of $R$ into $G$-invariant Borel sets such that for any $A,B$-sensitive finite Borel partition $\I$ of $X$ and $n \leq \w$, $A^{P_n} \sim_{\I} nB^{P_n} \oplus C_n$ for some $C_n$ such that $C_n \prec_{\I} B^{P_n}$, if $n < \w$.
\end{prop}
\begin{proof}
	We repeatedly apply \cref{comparability}. For $n < \w$, recursively define $R_n, P_n, A_n, C_n$ satisfying the following:
	\begin{enumerate}[(i)]
		\item $R_n$ are invariant decreasing Borel sets such that $n B^{R_n} \preceq_{\I} A^{R_n}$ for any $A,B$-sensitive $\I$;
		
		\item $P_n = R_n \setminus R_{n+1}$;
		
		\item $A_n \subseteq R_{n+1}$ are pairwise disjoint Borel sets such that for any $A,B$-sensitive $\I$, every $A_n$ respects $\I$ and $A_n \sim_{\I} B^{R_{n+1}}$;
		
		\item $C_n \subseteq P_n$ are Borel sets such that for any $A,B$-sensitive $\I$, every $C_n$ respects $\I$ and $C_n \prec_{\I} B^{P_n}$.
	\end{enumerate}
	
	Set $R_0 = R$. Given $R_n$, $\{A_k\}_{k<n}$ satisfying the above properties, let $A' = A^{R_n} \setminus \bigcup_{k<n} A_k$. We apply \cref{comparability} to $A'$ and $B^{R_n}$, and get a partition $R_n = P_n \cup R_{n+1}$ such that $(A')^{P_n} \prec_{\I} B^{P_n}$ and $B^{R_{n+1}} \preceq_{\I} (A')^{R_{n+1}}$. Set $C_n = (A')^{P_n}$. Let $A_n \subseteq (A')^{R_{n+1}}$ be such that $B^{R_{n+1}} \sim_{\I} A_n$. It is straightforward to check (i)-(iv) are satisfied.
	
	Now let $P_\w = \bigcap_{n \in \N} R_n$ and $C_{\w} = (A \setminus \Uni A_n)^{P_\w}$. It follows from (i)-(iv) that for all $n \le \w$, $\{A_k^{P_n}\}_{k<n} \cup \{C_n\}$ is a partition of $A^{P_n}$ witnessing $A^{P_n} \sim_{\I} nB \oplus C_n$, and for all $n < \w$, $C_n \prec B^{P_n}$.
\end{proof}

For $A,B \in \Bfrak(X)$, let $\{P_n\}_{n \le \w}$ be as in the above proposition. Define
$$[A / B](x) = \left\{
\begin{array}{ll}
n & \text{if } x \in P_n, n < \w \\
\infty & \text{if } x \in P_{\w} \text{ or } x \in [A]_G \setminus [B]_G \\
0 & \text{if } x \in [B]_G \setminus [A]_G \\
-1 & \text{otherwise}
\end{array}\right..$$
Note that $[A/B] : X \to \N \cup \set{-1, \infty}$ is a Borel function by definition.

\subsection{Properties of [A/B]}

\begin{lemma}[Infinite divisibility $\Rightarrow$ compressibility]\label{infinite divisibility => compressibility}
	Let $A,B \in \Bfrak(X)$ with $[A]_G = [B]_G$, and let $\I$ be a finite Borel partition of $X$. If $\w B \preceq_{\I} A$, then $A \prec_{\I} A$.
\end{lemma}
\begin{proof}
	Let $C \subseteq A$ be such that $A \sim_{\I} \w B \oplus C$ and let $\{A_k\}_{k < \w}$ be as in \cref{defn of divisibility}. $A_k \sim_{\I} B \sim_{\I} A_{k+1}$ and hence $A_k \sim_{\I} A_{k+1}$. Also trivially $C \sim_{\I} C$. Thus, letting $A' = \bigcup_{k < \w} A_{k+1} \cup C$, we apply (b) of \cref{properties of I-equidecomposability} to $A$ and $A'$, and get that $A \sim_{\I} A'$. Because $[A \setminus A']_G = [A_0]_G = [B]_G = [A]_G$, we have $A \prec_{\I} A$.
\end{proof}

\begin{lemma}[Ambiguity $\Rightarrow$ compressibility]\label{ambiguous divisibility => compressibility}
	Let $A,B \in \Bfrak(X)$ and $\I$ be a finite Borel partition of $X$. If $nB \preceq_{\I} A \prec_{\I} nB$ for some $n \ge 1$, then $A \prec_{\I} A$.
\end{lemma}
\begin{proof}
	Let $C \subseteq A$ be such that $A \sim_{\I} nB \oplus C$ and let $\{A_k\}_{k<n}$ be a partitions of $A \setminus C$ witnessing $A \sim_{\I} nB \oplus C$. Also let $\{A'_k\}_{k<n}$ be witnessing $A \prec_{\I} nB$ with $A'_0 \prec_{\I} B$. Since $A'_k \preceq_{\I} B \sim_{\I} A_k$, $A'_k \preceq_{\I} A_k$, for all $k<n$ and $A'_0 \prec_{\I} A_0$. Note that it follows from the hypothesis that $[A]_G = [B]_G$ and hence $[A_0]_G = [A]_G$ since $[A_0]_G = [B]_G$. Thus it follows from (b) of \cref{properties of I-equidecomposability} that $A = \bigcup_{k<n} A'_k \prec_{\I} \bigcup_{k<n} A_k \subseteq A$.
\end{proof}

\begin{prop}\label{properties of [A/B]}
	Let $n \in \N$ and $A,A',B,P \in \Bfrak(X)$, where $P$ is invariant.
	\begin{enumerate}[(a)]
		\item $[A/B] \in \N$ on $[B]_G$ modulo $\Cfrak_3$.
		
		\item If $A \subseteq A'$, then $[A / B] \le [A' / B]$.
		
		\item If $[A/B] = n$ on $P$ then $n B^P \preceq_{\I} A^P \prec_{\I} (n+1) B^P$, for any finite Borel partition $\I$ that is $A,B$-sensitive. In particular, $n B^P \preceq_2 A^P \prec_2 (n+1) B^P$ by taking $\I = \gen{A,B}$.
		
		\item For $n \ge 1$, if $A^P \prec_i nB^P$, then $[A/B] < n$ on $P$ modulo $\Cfrak_{i+1}$;
		
		\item If $A^P \subseteq [B]_G$ and $nB^P \preceq_i A^P$, then $[A/B] \ge n$ on $P$ modulo $\Cfrak_{i+1}$.
	\end{enumerate}
\end{prop}
\begin{proof}
	\noindent For (a), notice that \cref{infinite divisibility => compressibility,i-compressible is in C(i+1)} imply that $P_{\w} \in \Cfrak_3$. 
	
	For part (b), it is enough to note the following: if $X = P \cup Q$ and $X = P' \cup Q'$ are the partitions provided by \cref{comparability} when applied to $A,B$ and $A',B$, respectively, then it follows from the construction in the proof of that lemma that $Q' \supseteq Q$.
	
	Part (c) follows from the definition of $[A/B]$. 
	
	For (d), let $\I$ be an $A,B$-sensitive partition of $X$ generated by $i$ Borel sets such that $A^P \prec_{\I} nB^P$, and put $Q = \{x \in P : [A/B](x) \ge n\}$. By (c), $nB^Q \preceq_{\I} A^Q$. Thus, by \cref{ambiguous divisibility => compressibility}, $A^Q \prec_{\I} A^Q$ and hence, by \cref{i-compressible is in C(i+1)}, $[A^Q]_G = Q \in C_{i+1}$.
	
	For (e), let $\I$ be an $A,B$-sensitive partition of $X$ generated by $i$ Borel sets such that $nB^P \preceq_{\I} A^P$, and put $Q = \{x \in P : [A/B](x) < n\}$. By (c), $A^Q \prec_{\I} nB^Q$. Thus, by \cref{ambiguous divisibility => compressibility}, $A^Q \prec_{\I} A^Q$ and hence, by \cref{i-compressible is in C(i+1)}, $[A^Q]_G = Q \in C_{i+1}$.
\end{proof}

\begin{lemma}[Almost cancellation]\label{almost cancellation}
	For any $A,B,C \in X$,
	$$
	[A/B][B/C] \leq [A/C] < ([A/B] + 1)([B/C] + 1)
	$$
	on $R := [B]_G \cap [C]_G$ modulo $\Cfrak_4$.
\end{lemma}
\begin{proof}
	Let $\I = \gen{A,B,C}$.
	
	\medskip
	
	\noindent $[A/B][B/C] \leq [A/C]$: Fix integers $i,j > 0$ and let $P = \{x \in X : [A/B](x) = i \wedge [B/C](x) = j\}$. Since $i,j > 0$, $P \subseteq [A]_G \cap [B]_G \cap [C]_G$ and we work in $P$. By (c) of \cref{properties of [A/B]}, $i B \preceq_{\I} A$ and $j C \preceq_{\I} B$. Thus it follows that $ij C \preceq_{\I} A$ and hence $[A / C] \ge ij$ modulo $\Cfrak_4$ by (e) of \labelcref{properties of [A/B]}.
	
	\medskip
	
	\noindent $[A/C] < ([A/B] + 1)([B/C] + 1)$: By (a) of \labelcref{properties of [A/B]}, $[A/C], [A/B], [B/C] \in \N$ on $R$ modulo $\Cfrak_3$. Fix $i,j \in \N$ and let $Q = \{x \in R : [A/B](x) = i \wedge [B/C](x) = j\}$. We work in $Q$. By (c) of \labelcref{properties of [A/B]}, $A \prec_{\I} (i+1) B$ and $B \prec_{\I} (j+1) C$. Thus $A \prec_{\I} (i+1)(j+1) C$ and hence $[A/C] < (i+1)(j+1)$ modulo $\Cfrak_4$ by (d) of \labelcref{properties of [A/B]}.
\end{proof}

\begin{lemma}[Invariance]\label{invariance}
	For $A,F \in \Bfrak(X)$,  $\forall g \in G, [A / F] = [gA / F]$, modulo $\Cfrak_3$.
\end{lemma}
\begin{proof}
	We may assume that $X = [A]_G \cap [F]_G$. Fix $g \in G$, $n \in \N$, and put $Q = \{x \in X : [gA / F](x) = n\}$. We work in $Q$. Let $\I = \gen{A,F}$ and hence $A,gA,F$ respect $\I$. By (c) of \labelcref{properties of [A/B]}, $nF \preceq_{\I} gA$. But clearly $gA \sim_{\I} A$ and hence $nF \preceq_{\I} A$. Thus, by (e) of \labelcref{properties of [A/B]}, $[A / F] \geq n = [gA / F]$, modulo $\Cfrak_3$. By symmetry, $[gA / F] \geq  [A / F]$ (modulo $\Cfrak_3$) and the lemma follows.
\end{proof}

\begin{lemma}[Almost additivity]\label{almost additivity}
	For any $A,B,F \in X$ with $A \cap B = \0$, $[A/F] + [B/F] \leq [A \cup B / F] \leq [A/F] + [B/F] + 1$ modulo $\Cfrak_4$.
\end{lemma}
\begin{proof}
	Let $\I = \gen{A,B,F}$.
	
	\medskip
	
	\noindent $[A/F] + [B/F] \leq [A \cup B / F]$: Fix $i,j \in \N$ not both $0$, say $i>0$, and let $S = \{x \in X : [A/F](x) = i \wedge [B/F](x) = j\}$. Since $i>0$, $S \subseteq [A]_G \cap [F]_G$ and we work in $S$. By (c) of \labelcref{properties of [A/B]}, $iF^S \preceq_{\I} A^S$ and $jF^S \preceq_{\I} B^S$. Hence $(i+j)F^S \preceq_{\I} (A \cup B)^S$ and thus, by (e) of \labelcref{properties of [A/B]}, $[A \cup B / F] \ge i + j$, modulo $\Cfrak_4$.
	
	\medskip
	
	\noindent $[A \cup B / F] \leq [A/F] + [B/F] + 1$: Outside $[F]_G$, the inequality clearly holds. Fix $i,j \in \N$ and let $M = \{x \in [F]_G: [A/F](x) = i \wedge [B/F](x) = j\}$. We work in $M$. By (c) of \labelcref{properties of [A/B]}, $A \prec_{\I} (i+1)F$ and $B \prec_{\I} (j+1)F$. Thus it is clear that $A \cup B \prec_{\I} (i + j + 2)F$ and hence $[A \cup B / F] < i+j+2$, modulo $\Cfrak_4$, by (d) of \labelcref{properties of [A/B]}.
\end{proof}

\subsection{Fundamental sequence}

\begin{defn}
	A sequence $\{F_n\}_{n \in \N}$ of decreasing Borel complete sections with $F_0 = X$ and $[F_n/F_{n+1}] \geq 2$ modulo $\Cfrak_3$ is called fundamental.
\end{defn}

\begin{prop}\label{fundamental sequence exists}
	There exists a fundamental sequence.
\end{prop}
\begin{proof}
	Take $F_0 = X$. Given any complete Borel section $F$, its intersection with every orbit is infinite modulo a smooth set (if the intersection of an orbit with a set is finite, then we can choose an element from each such nonempty intersection in a Borel way and get a Borel transversal). Thus, by \cref{smooth sets are in C_1}, $F$ is aperiodic modulo $\Cfrak_1$. Now use \cref{markers} (the proof of this does not use any results from the current paper, so there is no loop) to write $F = A \cup B, A \cap B = \0$, where $A,B$ are also complete sections. Let now $P,Q$ be as in \cref{comparability} for $A,B$, and hence $A^P \prec_2 B^P, B^Q \preceq_2 A^Q$ because we can take $\I = \gen{A,B}$. Let $A' = A^P \cup B^Q, B' = B^P \cup A^Q$. Then $F = A' \cup B', A' \cap B' = \0$, $A' \preceq B'$ and $A'$ is also a complete Borel section. By (e) of \cref{properties of [A/B]}, $[F/A'] \geq 2$ modulo $\Cfrak_3$. Iterate this process to inductively define $F_n$.
\end{proof}

\subsection{Definition and properties of $m(A,x)$}

Fix a fundamental sequence $\{F_n\}_{n \in \N}$ and for any $A \in \Bfrak(X), x \in X$, define
\begin{equation}\label{eq:def_of_measure}
	m(A,x) = \lim_{n \rightarrow \infty} \frac{[A/F_n](x)}{[X/F_n](x)},
\end{equation}
if the limit exists, and $0$ otherwise. In the above fraction we define ${\infty \over \infty} = 1$. We will prove in \cref{limit exists} that this limit exists modulo $\Cfrak_4$. But first we need a lemma.

\begin{lemma}\label{limit is infinity}
	For any $A \in \Bfrak(A)$, $$\lim_{n \rightarrow \infty} [A / F_n] = \left\{\begin{array}{ll} \infty & \text{on } [A]_G \\ 0 & \text{on } X \setminus [A]_G \end{array}\right., \text{ modulo } \Cfrak_4.$$
\end{lemma}
\begin{proof}
	The part about $X \setminus [A]_E$ is clear, so work in $[A]_E$, i.e. assume $X = [A]_G$. By (a) of \labelcref{properties of [A/B]} and \cref{almost cancellation}, we have
	$$
	\infty > [F_1 /A] \ge [F_1 /F_n ] [F_n / A ] \ge 2^{n-1} [F_n /A], \text{ modulo } \Cfrak_4,
	$$
	which holds for all $n$ at once since $\Cfrak_4$ is a $\sigma$-ideal. Thus $[F_n /A] \to 0$ modulo $\Cfrak_4$ and hence, as $[F_n /A] \in \N$, $[F_n /A]$ is eventually $0$, modulo $\Cfrak_4$. So if
	$$
	B_k := \{x \in [A]_G : [F / A](x) = 0\},
	$$
	then $B_k \nearrow X$, modulo $\Cfrak_4$. Now it follows from \cref{comparability} that $[A / F_k] > 0$ on $B_k$ modulo $\Cfrak_4$. But
	$$
	[A/F_{k+n}] \ge [A / F_k ] [F_k / F_{k+n} ] \ge 2^n [A / F_k ], \text{ modulo } \Cfrak_4,
	$$
	so for every $k$, $[A/F_n] \to \infty$ on $B_k$ modulo $\Cfrak_4$. Since $B_k \nearrow X$ modulo $\Cfrak_4$, we have $[A/F_n] \to \infty$ on $X$, modulo $\Cfrak_4$.
\end{proof}

\begin{prop}\label{limit exists}
	For any Borel set $A \subseteq X$, the limit in \labelcref{eq:def_of_measure} exists and is positive on $[A]_G$, modulo $\Cfrak_4$.
\end{prop}
\begin{proof}
	\begin{claim*}
		Suppose $B,C \in \Bfrak(X)$, $i \in \N$ and $D_i = \{x \in X : [C / F_i](x) > 0\}$. Then
		$$
		\overline{\lim} {[B/F_n ] \over [C/F_n]} \le {[B/F_i] + 1 \over [C/F_i]}
		$$
		on $D_i$, modulo $\Cfrak_4$.
	\end{claim*}
	\begin{pfof}
		Working in $D_i$ and using \cref{almost cancellation}, $\forall j$ we have (modulo $\Cfrak_4$)
		\begin{align*}
		[B/F_{i+j}] & \le ([B/F_i ]+1) ([F_i / F_{i+j}] + 1) \\
		[C/F_{i+j}] & \ge [C/F_i ] [ F_i /F_{i+j}] > 0,
		\end{align*}
		so
		\begin{align*}
		{[B/F_{i+j}]\over [C/F_{i+j}]} & \le {[B/F_i ]+1 \over [C/F_i ]}
		\cdot {[F_i /F_{i+j}]+1\over [F_i /F_{i+j}]} \\
		& \le {[B/F_i ]+1\over [C/F_i ]} \cdot (1 + {1\over 2^j}),
		\end{align*}
		from which the claim follows.
	\end{pfof}
	
	Applying the claim to $B = A$ and $C = X$ (hence $D_i = X$), we get that for all $i \in \N$
	$$
	\overline{\lim_{n \to \infty}} {[A/F_n ](x)\over [X/F_n ](x)} \le {[A/F_i ](x)+1 \over [X/F_i ](x)} (\text{modulo } \Cfrak_4).
	$$
	Thus
	$$
	\overline{\lim_{n \to \infty}} {[A/F_n ]\over [X/F_n ]} \le \underline{\lim_{i \to \infty}} {[A/F_i
		]+1\over [X/F_i ]} = \underline{\lim_{i \to \infty}} {[A/F_i ]\over [X/F_i ]}
	$$
	since $\lim_{i \to \infty} {1 \over [X/F_i]} = 0$.
	
	To see that $m(A,x)$ is positive on $[A]_E$ modulo $\Cfrak_4$ we argue as follows. We work in $[A]_G$. Applying the above claim to $B = X$ and $C = A$, we get
	$$
	{1 \over m(A,x)} = \lim_{n \to \infty} {[X/F_n] \over [A/F_n]} \le {[X/F_i] + 1 \over [A/F_i]} < \infty \text{ on } D_i \text{ (modulo $\Cfrak_4$)}.
	$$
	Thus $m(A,x) > 0$ on $\bigcup_{i \in \N} D_i$, modulo $\Cfrak_4$. But $D_i \nearrow [A]_G$ because $[A/F_i] \to \infty$ as $i \to \infty$, and hence $m(A,x) > 0$ on $[A]_G$ modulo $\Cfrak_4$.
\end{proof}

\subsection{Proof of \cref{Nadkarni for C_4}}

Fix $A,B \in \Bfrak(X)$. The fact that $m(A,x) \in [0,1]$ and parts (b) and (d) follow directly from the definition of $m(A,x)$. Part (a) follows from the fact that $[A / F_n]$ is Borel for all $n \in \N$. (c) follows from (b) of \cref{properties of [A/B]}, and (e) and (f) are asserted by \cref{limit exists,invariance}, respectively.

To show (g), we argue as follows. By \cref{almost additivity}, $[A/F_n] + [B/F_n] \leq [A \cup B / F_n] \leq [A/F_n] + [B/F_n] + 1$, modulo $\Cfrak_4$, and thus $$\frac{[A/F_n]}{[X/F_n]} + \frac{[B/F_n]}{[X/F_n]} \leq \frac{[A \cup B / F_n]}{[X/F_n]} \leq \frac{[A/F_n]}{[X/F_n]} + \frac{[B/F_n]}{[X/F_n]} + \frac{1}{[X/F_n]},$$ for all $n$ at once, modulo $\Cfrak_4$ (using the fact that $\Cfrak_4$ is a $\sigma$-ideal). Since $[X/F_n] \geq 2^n$, passing to the limit in the inequalities above, we get $m(A,x) + m(B,x) \leq m(A \cup B,x) \leq m(A,x) + m(B,x)$. \hfill{QED (\cref{Nadkarni for C_4})}


\section{Finite generators in the case of $\sigma$-compact spaces}

In this section we prove that the answer to \cref{Weiss's question} is positive in case $X$ has a $\sigma$-compact realization. To do this, we first prove \cref{fin additive >0 on compact => ctbl additive}, which shows how to construct a countably additive invariant probability measure on $X$ using a finitely additive one. We then use \cref{measures on ctbl Boolean algebras} to conclude the result.

For the next two statements, let $X$ be a second countable Hausdorff topological space equipped with a continuous action of $G$.

\begin{lemma}\label{invariant metric outer measure}
	Let $\U \subseteq Pow(X)$ be a countable basis for $X$ closed under the $G$-action and finite unions/intersections. Let $\rho$ be a $G$-invariant finitely additive probability measure on the $G$-algebra generated by $\U$. For every $A \subseteq X$, define
	$$
	\mu^* (A) = \inf \{ \sum_{n \in \N} \rho(U_n ) : U_n \in \U \; \wedge \; A \subseteq \bigcup_{n \in \N} U_n\}.
	$$
	Then:
	\begin{enumerate}[(a)]
		\item $\mu^*$ is a $G$-invariant outer measure.
		
		\item If $K \subseteq X$ is compact, then $K$ is metrizable and $\mu^*$ is a metric outer measure on $K$ (with respect to any compatible metric).
	\end{enumerate}
\end{lemma}
\begin{proof}
	It is a standard fact from measure theory that $\mu^*$ is an outer measure. That $\mu^*$ is $G$-invariant follows immediately from $G$-invariance of $\rho$ and the fact that $\U$ is closed under the action of $G$.
	
	For (b), first note that by Urysohn metrization theorem, $K$ is metrizable, and fix a metric on $K$. If $E, F \subseteq K$ are a positive distance apart, then so are $\bar{E}$ and $\bar{F}$. Hence there exist disjoint open sets $U,V$ such that $\bar{E} \subseteq U$, $\bar{F} \subseteq V$. Because $\bar{E}$ and $\bar{F}$ are compact, $U,V$ can be taken to be finite unions of sets in $\U$ and therefore $U,V \in \U$.
	
	Now fix $\e>0$ and let $W_n \in \U$, be such that $E \cup F \subseteq \bigcup_n W_n$ and
	\begin{equation}\label{eq:outer_measure_almost_ctbl_additivity}
		\sum_n \rho(W_n) \le \mu^*(E \cup F) + \e \leq \mu^*(E) + \mu^*(F) +\e.
	\end{equation}
	Note that $\{W_n \cap U\}_{n \in \N}$ covers $E$, $\{W_n \cap V\}_{n \in \N}$ covers $F$ and $W_n \cap U, W_n \cap V \in \U$. Also, by finite additivity of $\rho$,
	$$
	\rho(W_n \cap U) + \rho(W_n \cap V) = \rho(W_n \cap (U \cup V)) \leq \rho(W_n).
	$$
	Thus
	$$
	\mu^*(E) + \mu^*(F) \le \sum_n \rho(W_n \cap U) + \sum_n \rho(W_n \cap V) \le \sum_n \rho(W_n),
	$$
	which, together with \labelcref{eq:outer_measure_almost_ctbl_additivity}, implies that $\mu^*(E \cup F) = \mu^*(E) + \mu^*(F)$ since $\e$ is arbitrary.
\end{proof}

\begin{prop}\label{fin additive >0 on compact => ctbl additive}
	Suppose there exist a countable basis $\U \subseteq Pow(X)$ for $X$ and a compact set $K \subseteq X$ such that the $G$-algebra generated by $\U \cup \set{K}$ admits a finitely additive $G$-invariant probability measure $\rho$ with $\rho(K)>0$. Then there exists a countably additive $G$-invariant Borel probability measure on $X$.
\end{prop}
\begin{proof}
	Let $K, \U$ and $\rho$ be as in the hypothesis. We may assume that $\U$ is closed under the $G$-action and finite unions/intersections. Let $\mu^*$ be the outer measure provided by \cref{invariant metric outer measure} applied to $\U$, $\rho$. Thus $\mu^*$ is a metric outer measure on $K$ and hence all Borel subsets of $K$ are $\mu^*$-measurable (see \cite[13.2]{Munroe}). This implies that all Borel subsets of $Y = [K]_G = \bigcup_{g \in G} gK$ are $\mu^*$-measurable because $\mu^*$ is $G$-invariant. By Carath\'{e}odory's theorem, the restriction of $\mu^*$ to the Borel subsets of $Y$ is a countably additive Borel measure on $Y$, and we extend it to a Borel measure $\mu$ on $X$ by setting $\mu(Y^c) = 0$. Note that $\mu$ is $G$-invariant and $\mu(Y) \le 1$.
	
	It remains to show that $\mu$ is nontrivial, which we do by showing that $\mu(K) \ge \rho(K)$ and hence $\mu(K)>0$. To this end, let $\{U_n\}_{n \in \N} \subseteq \U$ cover $K$. Since $K$ is compact, there is a finite subcover $\{U_n\}_{n < N}$. Thus $U := \bigcup_{n < N} U_n \in \U$ and $K \subseteq U$. By finite additivity of $\rho$, we have
	$$
	\sum_{n \in \N} \rho(U_n) \ge \sum_{n < N} \rho(U_n) \ge \rho(U) \ge \rho(K),
	$$
	and hence, it follows from the definition of $\mu^*$ that $\mu^*(K) \ge \rho(K)$. Thus $\mu(K) = \mu^*(K) > 0$.
\end{proof}

\begin{cor}\label{invariant measure for compact}
	Let $X$ be a second countable Hausdorff topological $G$-space whose Borel structure is standard. For every compact set $K \subseteq X$ not in $\Cfrak_4$, there is a $G$-invariant countably additive Borel probability measure $\mu$ on $X$ with $\mu(K) > 0$.
\end{cor}
\begin{proof}
	Fix any countable basis $\U$ for $X$ and let $\BA$ be the Boolean $G$-algebra generated by $\U \cup \set{K}$. By \cref{measures on ctbl Boolean algebras}, there exists a $G$-invariant finitely additive probability measure $\rho$ on $\BA$ such that $\rho(K) > 0$. Now apply \cref{fin additive >0 on compact => ctbl additive}.
\end{proof}

As a corollary, we derive the analogue of Nadkarni's theorem for $\Cfrak_4$ in case of $\sigma$-compact spaces.
\begin{cor}\label{Nadkarni for sigma-compact}
	Let $X$ be a Borel $G$-space that admits a $\sigma$-compact realization. $X \notin \Cfrak_4$ if and only if there exists a $G$-invariant countably additive Borel probability measure on $X$.
\end{cor}
\begin{proof}
	\noindent $\Leftarrow$: If $X \in \Cfrak_4$, then it is compressible in the usual sense and hence does not admit a $G$-invariant Borel probability measure.
	
	\noindent $\Rightarrow$: Suppose that $X$ is a $\sigma$-compact topological $G$-space and $X \notin \Cfrak_4$. Then, since $X$ is $\sigma$-compact and $\Cfrak_4$ is a $\sigma$-ideal, there is a compact set $K$ not in $\Cfrak_4$. Now apply \cref{invariant measure for compact}.
\end{proof}

\begin{remark}
	For a Borel $G$-space $X$, let $\K$ denote the collection of all subsets of invariant Borel sets that admit a $\sigma$-compact realization (when viewed as Borel $G$-spaces). Also, let $\Cfrak$ denote the collection of all subsets of invariant compressible Borel sets. It is clear that $\K$ and $\Cfrak$ are $\sigma$-ideals, and what \cref{Nadkarni for sigma-compact} implies is that $\Cfrak \cap \K \subseteq \Cfrak_4$.
\end{remark}

\begin{theorem}\label{sigma-compact compressible => finite generator}
	Let $X$ be a Borel $G$-space that admits a $\sigma$-compact realization. If there is no $G$-invariant Borel probability measure on $X$, then $X$ admits a Borel $32$-generator.
\end{theorem}
\begin{proof}
	By \cref{Nadkarni for sigma-compact}, $X \in \Cfrak_4$ and hence, $X$ is $4$-compressible. Thus, by \cref{i-compressible => finite generator}, $X$ admits a Borel $2^5$-generator.
\end{proof}

\begin{example}
	Let $LO$ denote the set of all linear orderings of $\N$ with the ordering relation symbol $<$; this can be modeled as a closed subset of $2^{\N^2}$, so it is a compact Polish space. We think of each $x \in LO$ as a structure $(\N,<_x)$, where $<_x$ is the linear ordering of $\N$ according to $x$ and we write, for example, $7 <_x 5$ to mean that $7$ is less than $5$ according to $x$.
	
	Letting $G$ be the group of finite permutations of elements of $\N$, we see that $G$ is countable and acts continuously on $LO$ in the natural way: 
	$$
	n <_{gx} m \shortiff g^{-1}(n) <_x g^{-1}(m),
	$$
	for $n, m \in \N, g \in G, x \in LO$; this is referred to as \emph{the logic action}. 
	
	Put $Y = LO \setminus DLO$, where $DLO$ denotes the set of all dense linear orderings without endpoints (copies of $\Q$). It is straightforward to see that $DLO$ is a $G_{\de}$ subset of $LO$, hence $Y$ is $F_{\sigma}$ and therefore $\sigma$-compact since $LO$ is compact. Also clearly $Y$ is $G$-invariant.
	
	Let $\mu$ be the unique measure on $LO$ defined by $\mu(V_{(F,<_F)}) = {1 \over n!}$, where $(F,<_F)$ is a finite linearly ordered subset of $\N$ of cardinality $n$ and $V_{(F,<_F)}$ is the set of all linear orderings of $\N$ extending the order $<_F$ on $F$. It is not hard to check (shown in \cite{GW}) that $\mu$ is the unique invariant measure for the action of $G$ on $LO$ and $\mu(Y) = 0$. Thus, there is no $G$-invariant Borel probability measure on $Y$ and hence, by \cref{sigma-compact compressible => finite generator}, $Y$ admits a $32$-generator. However, as pointed out by Todor Tsankov and the referee, $LO$ (and hence also $Y$) already has an obvious $2$-generator: namely, the partition generated by the set $\set{x \in LO : 0 <_x 1}$. 
	
	Nevertheless, we can modify this example to make the application of \cref{sigma-compact compressible => finite generator} more fruitful by considering multiple relations instead of just one. For example, for $n \ge 1$, let $LO_n$ be the set of all $n$-tuples of linear orderings of $\N$ with the ordering relation symbols $<^0, <^1, ..., <^{n-1}$, and, as before, consider the natural (logic) action of $G$ on $LO_n$; this can be modeled as the coordinatewise (diagonal) action of $G$ on $LO_n := LO^n$, i.e. $g (x_0,x_1,...,x_{n-1}) = (gx_0,gx_1,...,gx_{n-1})$. The Polish $G$-space $LO_n$ has an obvious $2^n$-generator: namely, the partition generated by the sets $\set{x \in LO_n : 0 <^i_x 1}$, $i<n$. However, letting $X = LO^n \setminus DLO^n$, we again see that it is $G$-invariant and $\sigma$-compact. Moreover, $X$ does not admit an invariant probability measure because otherwise, one of the sets of the form $LO^i \times Y \times LO^{n-i-1}$ would have positive measure, so the pushforward measure under the projection onto the $i^\text{th}$ coordinate would be a nontrivial finite invariant measure on $Y$, but we argued above that such a measure does not exist. Thus, by \cref{sigma-compact compressible => finite generator}, $X$ has a $32$-generator, which gives us something new when $n > 5$.
\end{example}

\section{Finitely traveling sets}

Throughout this section, let $X$ be a Borel $G$-space.

\begin{defn}\label{defn of lfin equidec}
	Let $A,B \in \Bfrak(X)$ be equidecomposable, i.e. there are $N \leq \infty$, $\{g_n\}_{n < N} \subseteq G$ and Borel partitions $\{A_n\}_{n < N}$ and $\{B_n\}_{n < N}$ of $A$ and $B$, respectively, such that $g_n A_n = B_n$ for all $n < N$. $A,B$ are said to be
	\begin{itemize}
		\item locally finitely equidecomposable (denote by $A \sim_{\lfin} B$), if $\{A_n\}_{n < N},\{B_n\}_{n < N},\{g_n\}_{n < N}$ can be taken so that for every $x \in A$, $A_n \cap [x]_G = \0$ for all but finitely many $n<N$;
		
		\item finitely equidecomposable (denote by $A \sim_{\fin} B$), if $N$ can be taken to be finite.
	\end{itemize}
\end{defn}

The notation $\prec_{\fin}$, $\prec_{\lfin}$ and the notions of finite and locally finite compressibility are defined analogous to \cref{defn of equidec,defn of compressibility}.

\begin{defn}\label{defn of fin traveling sets}
	A Borel set $A \subseteq X$ is called (locally) finitely traveling if there exists pairwise disjoint Borel sets $\{A_n\}_{n \in \N}$ such that $A_0 = A$ and $A \sim_{\fin} A_n$ ($A \sim_{\lfin} A_n$), $\forall n \in \N$.
\end{defn}

\begin{prop}\label{finitely compressible => finitely traveling}
	If $X$ is (locally) finitely compressible then $X$ admits a (locally) finitely traveling Borel complete section.
\end{prop}
\begin{proof}
	We prove for finitely compressible $X$, but note that everything below is also locally valid (i.e. restricted to every orbit) for a locally compressible $X$.
	
	Run the proof of the first part of \cref{lemma i-compressibility <=> i-traveling} noting that a witnessing map $\ga : X \to G$ of finite compressibility of $X$ has finite image and hence the image of each $\de_n$ (in the notation of the proof) is finite, which implies that the obtained traveling set $A$ is actually finitely traveling.
\end{proof}

\begin{prop}\label{lfin traveling => X in C4}
	If $X$ admits a locally finitely traveling Borel complete section, then $X \in \Cfrak_4$.
\end{prop}
\begin{proof}
	Let $A$ be a locally finitely traveling Borel complete section and let $\{A_n\}_{n \in \N}$ be as in \cref{defn of fin traveling sets}. Let $\I_n = \{C_k^n\}_{k \in \N}$, $\J_n = \{D_k^n\}_{k \in \N}$ be Borel partitions of $A$ and $A_n$, respectively, that together with $\{g_k^n\}_{k \in \N} \subseteq G$ witness $A \sim_{\lfin} A_n$ (as in \cref{defn of lfin equidec}). Let $\BA$ denote the Boolean $G$-algebra generated by $\{X\} \cup \bigcup_{n \in \N} (\I_n \cup \J_n \cup \{A_n\})$.
	
	Now assume for contradiction that $X \notin \Cfrak_4$ and hence, $A \notin \Cfrak_4$. Thus, applying \cref{measures on ctbl Boolean algebras} to $A$ and $\BA$, we get a $G$-invariant finitely additive probability measure $\mu$ on $\BA$ with $\mu(A)>0$. Moreover, there is $x \in A$ such that $\forall B \in \BA$ with $B \cap [x]_G = \0$, $\mu(B) = 0$.
	
	\begin{claim*}
		$\mu(A_n) = \mu(A)$, for all $n \in \N$.
	\end{claim*}
	\begin{pfof}
		For each $n$, let $\{C_{k_i}^n\}_{i < K_n}$ be the list of those $C_k^n$ such that $C_k^n \cap [x]_G \neq \0$ ($K_n < \w$ by the definition of locally finitely traveling). Set $B = A \setminus (\bigcup_{i < K_n} C_{k_i}^n)$ and note that by finite additivity of $\mu$,
		$$\mu(A) = \mu(B) + \sum_{i < K_n} \mu(C_{k_i}^n).$$
		Similarly, set $B' = A_n \setminus (\bigcup_{i < K_n} D_{k_i}^n)$ and hence
		$$\mu(A_n) = \mu(B') + \sum_{i < K_n} \mu(D_{k_i}^n).$$
		But $B \cap [x]_G = \0$ and $B' \cap [x]_G = \0$, and thus $\mu(B) = \mu(B') = 0$. Also, since $g_{k_i}^n C_{k_i}^n = D_{k_i}^n$ and $\mu$ is $G$-invariant, $\mu(C_{k_i}^n) = \mu(D_{k_i}^n)$. Therefore
		$$\mu(A) = \sum_{i < K_n} \mu(C_{k_i}^n) = \sum_{i < K_n} \mu(D_{k_i}^n) = \mu(A_n).$$
	\end{pfof}
	
	This claim contradicts $\mu$ being a probability measure since for large enough $N$, $\mu(\bigcup_{n < N} A_n) = N \mu(A) > 1$, contradicting $\mu(X) = 1$.
\end{proof}

This, together with \cref{i-compressible => finite generator}, implies the following.

\begin{cor}\label{lfin traveling => 2^5-generator}
	Let $X$ be a Borel $G$-space. If $X$ admits a locally finitely traveling Borel complete section, then there is a Borel $32$-generator.
\end{cor}

\section{Locally weakly wandering sets and other special cases}\label{section_weakly wandering sets}

Assume throughout the section that $X$ is a Borel $G$-space.

\begin{defn}\label{defn of weakly wandering set}
	We say that $A \subseteq X$ is
	\begin{itemize}
		\item weakly wandering with respect to $H \subseteq G$ if $(h A) \cap (h' A) = \0$, for all distinct $h, h' \in H$;
		
		\item weakly wandering, if it is weakly wandering with respect to an infinite subset $H \subseteq G$ (by shifting $H$, we can always assume $1_G \in H$);
		
		\item locally weakly wandering if for every $x \in X$, $A^{[x]_G}$ is weakly wandering.
	\end{itemize}
\end{defn}

\subsection{Weakly wandering sets and finite generators}

The following is a prototypical/toy example of a construction of a finite generator, and it has served as a driving idea for a number of constructions in the current paper.

\begin{prop}\label{st:ww_imp_3-gen}
	If a $G$-space $X$ admits a Borel weakly wandering complete section, then it admits a $3$-generator.
\end{prop}
\begin{proof}
	Let $W$ be a Borel complete section that is weakly wandering with respect to an infinite sequence $(g_n)_{n \in \N} \subseteq G$, where $g_0 = 1_G$. Let $(U_n)_{n \ge 1}$ be a sequence of Borel sets that generate the Borel $\si$-algebra of $X$ (e.g. a countable basis of open sets in a compatible Polish topology) and put
	$$
	V = \bigcup_{n \ge 1} (g_n W \cap g_n U_n).
	$$
	Thus, $W \cap V = \0$ and we claim that the partition $\set{W, V, (W \cup V)^c}$ is a generator. Indeed, fix distinct points $x,y \in X$. Because $W$ intersects the orbit of $x$, there is $g \in G$ with $gx \in W$, so by replacing $x,y$ with $gx,gy$, we may assume that $x$ was in $W$ to begin with. Now if $y$ is not in $W$, then $W$ separates $x$ and $y$, and we are done; so suppose $y$ is also in $W$. Then, since $\set{U_n}_{n \ge 1}$ separates points, there must be $n \ge 1$ such that $x \in U_n$ but $y \notin U_n$. But then $g_n x \in g_n W \cap g_n U_n$, and hence $g_n x \in V$, whereas $g_n y \notin g_n U_n$, so $g_n y \notin V$ because $g_n y \in g_n W$ and $g_n W \cap V \subseteq g_n U_n$. Thus $g_n^{-1} V$ separates $x$ and $y$.
\end{proof}

\begin{examplelist*}
	\item Let $X = \R$ and let $\Z$ act on $\R$ by translation. Then any interval is weakly wandering and any interval of length greater than $1$ is a complete section. Thus, by the above proposition, this Polish $\Z$-space admits a $3$-generator.
	
	\item Let $X = \Br$ (the Baire space) and $\Ez$ be the equivalence relation of eventual agreement of sequences of natural numbers. We find a countable group $G$ of homeomorphisms of $X$ such that $E_G = \Ez$. For each $s,t \in \QB$ with $s \perp t$ (i.e. $s \nsubseteq t$ and $t \nsubseteq s$) or $s=t$, let $\phi_{s,t} : X \to X$ be defined as follows:
	$$
	\phi_{s,t}(x) = 
	\left\{\begin{array}{ll} 
	t \!\smallfrown\! y & \text{if } x = s \!\smallfrown\! y \\
	s \!\smallfrown\! y & \text{if } x = t \!\smallfrown\! y \\
	x & \text{otherwise}
	\end{array}\right.,
	$$
	and let $G$ be the group generated by $\{\phi_{s,t} : s,t \in \QB, |s|=|t|\}$. It is clear that each $\phi_{s,t}$ is a homeomorphism of $X$ and $E_G = \Ez$. Now for $n \in \N$, let $X_n = \{x \in X : x(0) = n\}$ and let $g_n = \phi_{0,n}$. Then the $X_n$ are pairwise disjoint and $g_n X_0 = X_n$. Hence $X_0$ is a weakly wandering complete section and thus $X$ admits a Borel $3$-generator by \cref{st:ww_imp_3-gen}.

	\item Let $X = 2^{\N}$ (the Cantor space) and $E_t$ be the tail equivalence relation on $X$, that is:
	$$
	x E_t y \Leftrightarrow (\exists n,m \in \N) (\forall k \in \N) x(n+k) = y(m+k).
	$$ 
	Let $G$ be the group generated by $\{\phi_{s,t} : s,t \in 2^{<\N}, s \perp t\}$, where the $\phi_{s,t}$ are defined as above. To see that $E_G = E_t$ fix $x,y \in X$ with $x E_t y$. Thus, there are nonempty $s,t \in 2^{<\N}$ and $z \in X$ such that $x = s \!\smallfrown\! z$ and $y = t \!\smallfrown\! z$. If $s \perp t$, then $y = \phi_{s,t}(x)$. Otherwise, assume, say, $s \sqsubseteq t$ and let $s' \in 2^{<\N}$ be such that $s \perp s'$ (exists since $s \neq \0$). Then $s' \perp t$ and $y = \phi_{s',t} \circ \phi_{s,s'}(x)$.
	
	Now for $n \in \N$, let $s_n = \underbrace{11...1}_n 0$ and $X_n = \{x \in X : x = s_n \!\smallfrown\! y, \text{ for some } y \in X\}$. Note that the $s_n$ are pairwise incompatible and hence the $X_n$ are pairwise disjoint. Letting $g_n = \phi_{s_0,s_n}$, we see that $g_n X_0 = X_n$. Thus $X_0$ is a weakly wandering complete section and hence $X$ admits a Borel $3$-generator by \cref{st:ww_imp_3-gen}.
\end{examplelist*}

\subsection{Localization}

Let $F(\GN)$ denote the Effros space of $\GN$, i.e. the standard Borel space of closed subsets of $\GN$, see \cite[12.C]{bible}. Below we use a Borel selector for $F(\GN)$, i.e. a Borel function $s : F(\GN) \to \GN$ with $s(F) \in F$ for each nonempty $F \in F(\GN)$. It is a basic fact of descriptive set theory that such functions exist; see, for example, \cite[Theorem 12.13]{bible}.

For $A \subseteq X$ and $x \in A$, put
$$
\De_A(x) = \{(g_n)_{n \in \N} \in \GN : g_0 = 1_G \wedge \forall n \neq m (g_n A^{[x]_G} \cap g_m A^{[x]_G} = \0) \}.
$$

\begin{prop}\label{travel guide for lww}
	Let $A \in \Bfrak(X)$.
	\begin{enumerate}[(a)]
		\item $\forall x \in X$, $\De_A(x)$ is a closed set in $\GN$.
		\item $\De_A : A \rightarrow F(\GN)$ is $\s$-measurable and hence universally measurable.
		\item $\De_A$ is $F_A$-invariant, i.e. $\forall x,y \in A$, if $x F_A y$ then $\De_A(x) = \De_A(y)$.
		\item If $A$ is locally weakly wandering, then it is $1$-traveling with $\s$-pieces. In fact, for any Borel selector $s : F(\GN) \to \GN$, the function $\ga := s \circ \De_A$ is a $\s$-measurable $F_A$- and $G$-invariant travel guide for $A$.
	\end{enumerate}
\end{prop}
\begin{proof}
	\begin{enumerate}[(a)]
		\item $\De_A(x)^c$ is open since being in it is witnessed by two coordinates.
		
		\item For $s \in G^{<\N}$, let $B_s = \{F \in F(\GN) : F \cap V_s \neq \0\}$, where $V_s = \{\alpha \in \GN : \alpha \sqsupseteq s\}$. Since $\{B_s\}_{s \in G^{<\N}}$ generates the Borel structure of $F(\GN)$, it is enough to show that $\De_A^{-1}(B_s)$ is analytic, for every $s \in G^{<\N}$. But $\De_A^{-1}(B_s) = \{x \in X : \exists (g_n)_{n \in \N} \in V_s [g_0 = 1_G \wedge \forall n \neq m (g_n A^{[x]_G} \cap g_m A^{[x]_G} = \0)]\}$ is clearly analytic.
		
		\item Assume for contradiction that $x F_A y$, but $\De_A(x) \neq \De_A(y)$ for some $x,y \in A$. We may assume that there is $(g_n)_{n \in \N} \in \De_A(x) \setminus \De_A(y)$ and thus $\exists n \neq m$ such that $g_n A^{[y]_G} \cap g_m A^{[y]_G} \neq \0$. Hence $A^{[y]_G} \cap g_n^{-1}g_m A^{[y]_G} \neq \0$ and let $y',y'' \in A^{[y]_G}$ be such that $y'' = g_n^{-1}g_m y'$. Let $g \in G$ be such that $y' = gy$.
		
		Since $y' = gy$, $y'' = g_n^{-1}g_m g y$ are in $A$, $x F_A y$, and $A$ is $F_A$-invariant, $gx, g_n^{-1}g_m g x$ are in $A$ as well. Thus $A^{[x]_G} \cap g_n^{-1}g_m A^{[x]_G} \neq \0$, contradicting $g_n A^{[y]_G} \cap g_m A^{[y]_G} = \0$ (this holds since $(g_n)_{n \in \N} \in \De_A(x)$).
		
		\item Follows from parts (b) and (c), and the definition of $\De_A$.
	\end{enumerate}
\end{proof}

\begin{theorem}\label{finite generators for lww}
	Let $X$ be a Borel $G$-space. If there is a locally weakly wandering Borel complete section for $X$, then $X$ admits a Borel $4$-generator.
\end{theorem}
\begin{proof}
	By part (d) of \labelcref{travel guide for lww} and \cref{equivalences to i-compressibility}, $X$ is $1$-compressible. Thus, by \cref{i-compressible => finite generator}, $X$ admits a Borel $2^2$-finite generator.
\end{proof}

\subsection{Countable unions of weakly wandering sets}

We can do even better in the case when a locally weakly wandering complete section is actually a countable union of weakly wandering sets.

\begin{cor}\label{st:ctbl-U-ww_imp_3-gen}
	Let $X$ be a Borel $G$-space. If $X$ admits a complete section that is a countable union of weakly wandering Borel sets, then $X$ admits a Borel $3$-generator.
\end{cor}
\begin{proof}
	Let $(W_n)_{n \in \N}$ be a sequence of Borel weakly wandering sets such that $\bigcup_{n \in \N} W_n$ is a complete section. By replacing each $W_n$ with $W_n \setminus \bigcup_{i<n} [W_i]_G$, we may assume that the $[W_n]_G$ are pairwise disjoint and hence $A := \bigcup_{n \in \N} W_n$ is a locally weakly wandering complete section. Using countable choice, take a function $p : \N \rightarrow \GN$ such that $\forall n \in \N$, $p(n) \in \bigcap_{x \in W_n} \De_{W_n}(x)$ (we know that $\bigcap_{x \in W_n} \De_{W_n}(x) \neq \0$ since $W_n$ is weakly wandering).
	
	Define $\ga : A \rightarrow \GN$ by
	$$
	x \mapsto \text{the smallest $k$ such that } p(k) \in \De_A(x).
	$$
	The condition $p(k) \in \De_A(x)$ is Borel because it is equivalent to $\forall n,m \in \N, y,z \in A \cap [x]_G, p(k)(n)y = p(k)(m)z \Rightarrow n=m \wedge x=y$; thus $\ga$ is a Borel function. Note that $\ga$ is a travel guide for $A$ by definition. Moreover, it is $F_A$-invariant because if $\De_A(x) = \De_A(y)$ for some $x,y \in A$, then conditions $p(k) \in \De_A(x)$ and $p(k) \in \De_A(y)$ hold or fail together. Since $\De_A$ is $F_A$-invariant, so is $\ga$. Hence, \cref{I-traveling implies finite generator} applied to $\I = \gen{A}$ gives a Borel $(2 \cdot 2 - 1)$-generator.
\end{proof}

In the light of this last corollary, we now record a version of the Hajian--Kakutani--It\^{o} theorem (see \labelcref{st:HKI_theorem}) as a corollary of the same theorem.

\begin{cor}\label{st:HKI_with_ctbl-U-ww}
	Let $(X,\mu)$ be a standard probability space equipped with a nonsingular Borel action of $G$. There is no invariant Borel probability measure absolutely continuous with respect to $\mu$ if and only if, modulo $\mu$-$\NULL$, there is a complete section that is a countable union of weakly wandering Borel sets.
\end{cor}
\begin{proof}
	The right-to-left direction immediately follows from the fact that if $W$ is a weakly wandering Borel set and $\nu$ is an invariant finite measure, then $\nu([W]_G) = 0$. For the left-to-right direction, we use a standard measure exhaustion argument based on iterative applications of the Hajian--Kakutani--It\^{o} theorem. By recursion on $n \in \N$, we will define a decreasing sequence $(X_n)_{n \in \N}$ of invariant Borel sets as well as a disjoint sequence $(W_n)_{n \in \N}$ of $\mu$-positive weakly wandering Borel sets such that
	\begin{enumerate}[(i)]
		\item the $X_n$ are $\mu$-vanishing, i.e. $\bigcap_{n \in \N} X_n$ is $\mu$-null,
		\item $W_n$ is a complete section for $X_n \setminus X_{n+1}$, i.e. $[W_n]_G = X_n \setminus X_{n+1}$.
	\end{enumerate}
	To this end, put $X_0 := X$, and assuming that $X_n$ is defined, apply the Hajian--Kakutani--It\^{o} theorem to $X_n$ and get that the set
	$$
	\W_n := \set{W \subseteq X_n : W \text{ is Borel and } \mu(W) > 0}
	$$
	is nonempty. Thus there is $W_n \in \W_n$ with $\mu(W_n) > {1 \over 2} w_n$, where $w_n := \sup_{W \in \W_n} \mu(W)$. Putting $X_{n+1} := X_n \setminus [W_n]_G$, we are through with the construction. However, we still have to check that $X_\w := \bigcap_{n \in \N} X_n$ is $\mu$-null. If $\mu(X_\w)>0$, an application of the Hajian--Kakutani--It\^{o} theorem to $X_\w$ would provide a $\mu$-positive weakly wandering Borel set $W_\w \subseteq X_\w$. But the sequence $(w_n)_{n \in \N}$ is summable since the $W_n$ are pairwise disjoint and $\mu(W_n) > {1 \over 2} w_n$, so for large enough $n \in \N$, $w_n < \mu(W_\w)$, contradicting the definition of $w_n$.
	
	Finally, putting $W = \bigcup_{n \in \N} W_n$, we get a complete section for $X \setminus X_\w$ that is a countable union of weakly wandering Borel sets.
\end{proof}

\cref{st:ctbl-U-ww_imp_3-gen,st:HKI_with_ctbl-U-ww} immediately imply the following version of the Krengel--Kuntz theorem (see \labelcref{Kuntz's thm}) with a $3$-generator instead of $2$.

\begin{cor}\label{st:Krengel--Kuntz_with_3}
	Let $(X,\mu)$ be a standard probability space equipped with a nonsingular Borel action of $G$. If there is no invariant Borel probability measure absolutely continuous with respect to $\mu$, then $X$ admits a $3$-generator modulo $\mu$-$\NULL$.
\end{cor}

\subsection{Further special cases}

Using the function $\De$ defined above, we give another proof of \cref{transversals are 1-traveling}.
\begin{namedthm*}{\cref*{transversals are 1-traveling}}
	Let $X$ be an aperiodic Borel $G$-space and $T \subseteq X$ be Borel. If $T$ is a partial transversal, then $T$ is $\gen{T}$-traveling.
\end{namedthm*}
\begin{proof}
	By definition, $T$ is locally weakly wandering.
	\begin{claim*}
		$\De_T$ is Borel.
	\end{claim*}
	\begin{pfof}
		Using the notation of the proof of part (b) of \cref{travel guide for lww}, it is enough to show that $\De_T^{-1}(B_s)$ is Borel for every $s \in G^{<\N}$. But since $\forall x \in T$, $T \cap [x]_G$ is a singleton, $\De_T(x) \in B_s$ is equivalent to $s(0) = 1_G \wedge (\forall n < m < |s|)$ $s(m)x \neq s(n)x$. The latter condition is Borel, hence so is $\De_T^{-1}(B_s)$.
	\end{pfof}
	
	By part (d) of \labelcref{travel guide for lww}, $\ga = s \circ \De_T$ is a Borel $F_T$-invariant travel guide for $T$.
\end{proof}

\begin{cor}
	Every aperiodic and smooth Borel $G$-space $X$ admits a Borel $3$-generator.
\end{cor}
\begin{proof}
	Let $T \subseteq X$ be a Borel transversal. By \cref{transversals are 1-traveling}, $T$ is $\gen{T}$-traveling. Thus, by \cref{I-traveling implies finite generator}, there is a Borel $(2 \cdot 2 - 1)$-generator.
\end{proof}

Lastly, in case of smooth free actions, a direct construction gives the optimal result as the following proposition shows.
\begin{prop}\label{2-generator for free smooth}
	Let $X$ be a Borel $G$-space. If the $G$-action is free and smooth, then $X$ admits a Borel $2$-generator.
\end{prop}
\begin{proof}
	Let $T \subseteq X$ be a Borel transversal. Also let $G \setminus \set{1_G} = \{g_n\}_{n \in \N}$ be such that $g_n \neq g_m$ for $n \neq m$. Because the action is free, $g_n T \cap g_m T = \0$ for $n \neq m$.
	
	Define $\pi : \N \rightarrow \N$ recursively as follows:
	$$
	\pi(n) = \left\{\begin{array}{ll} \min\{m : g_m \notin \{g_{\pi(i)} : i < n\}\} & \text{if } n=3k \\
	\min\{m : g_m, g_m g_k \notin \{g_{\pi(i)} : i < n\}\} & \text{if } n=3k+1 \\
	\text{the unique $l$ s.t. } g_l = g_{\pi(3k+1)}g_k & \text{if }
	n=3k+2
	\end{array}\right..
	$$
	
	Note that $\pi$ is a bijection. Fix a countable family $\{U_n\}_{n \in \N}$ generating the Borel sets and put $A = \bigcup_{k \in \N} g_{\pi(3k)}(T \cap U_k) \cup \bigcup_{k \in \N} g_{\pi(3k+1)}T$. Clearly, $A$ is Borel, and we show that $\I = \gen{A}$ is a generator. Fix distinct $x, y \in X$. Note that since $T$ is a complete section, we can assume that $x \in T$.
	
	First assume $y \in T$. Take $k$ with $x \in U_k$ and $y \notin U_k$. Then $g_{\pi(3k)} x \in g_{\pi(3k)}(T \cap U_k) \subseteq A$ and $g_{\pi(3k)} y \in g_{\pi(3k)}(T \setminus U_k)$. However $g_{\pi(3k)}(T \setminus U_k) \cap A = \emptyset$ and hence $g_{\pi(3k)} y \notin A$.
	
	Now suppose $y \notin T$. Then there exists $y' \in T^{[y]_G}$ and $k$ such that $g_ky' = y$. Now $g_{\pi(3k+1)}x \in g_{\pi(3k+1)} T \subseteq A$ and $g_{\pi(3k+1)} y = g_{\pi(3k+1)}g_k y' = g_{\pi(3k+2)} y' \in g_{\pi(3k+2)} T$. But $g_{\pi(3k+2)} T \cap A = \emptyset$, hence $g_{\pi(3k+1)} y \notin A$.
\end{proof}

\begin{cor}
	Let $H$ be a Polish group and $G$ be a countable subgroup of $H$. If $G$ admits an infinite discrete subgroup, then the translation action of $G$ on $H$ admits a $2$-generator.
\end{cor}
\begin{proof}
	Let $G'$ be an infinite discrete subgroup of $G$. Clearly, it is enough to show that the translation action of $G'$ on $H$ admits a $2$-generator. Since $G'$ is discrete, it is closed. Indeed, if $d$ is a left-invariant compatible metric on $H$, then $B_d(1_H, \e) \cap G' = \set{1_H}$, for some $\e>0$. Thus every $d$-Cauchy sequence in $G'$ is eventually constant and hence $G'$ is closed. This implies that the translation action of $G'$ on $H$ is smooth and free (see \cite[12.17]{bible}), and hence \cref{2-generator for free smooth} applies.
\end{proof}

\section{Separating partitions}

Assume throughout this section that $X$ is a Borel $G$-space.

\subsection{Aperiodic separation and $G$-equivariant maps to $2^G$}

\begin{lemma}\label{markers}
	If $X$ is aperiodic then it admits a countably infinite partition into Borel complete sections.
\end{lemma}
\begin{proof}
	The following argument is also given in \cite[proof of Theorem 13.1]{KM}. By the marker lemma (see \cite[6.7]{KM}), there exists a vanishing sequence $\set{B_n}_{n \in \N}$ of decreasing Borel complete sections, i.e. $\bigcap_{n \in \N} B_n = \0$. For each $n \in \N$, define $k_n : X \to \N$ recursively as follows:
	$$
	\left\{\begin{array}{rcl}
	k_0(x) & = & 0 \\
	k_{n+1}(x) &= & \min \set{k \in \N : B_{k_n(x)} \cap [x]_G \nsubseteq B_k}
	\end{array}\right.,
	$$
	and define $A_n \subseteq X$ by
	$$
	x \in A_n \Leftrightarrow x \in B_{k_n(x)} \setminus B_{k_{n+1}(x)}.
	$$
	It is straightforward to check that $A_n$ are pairwise disjoint Borel complete sections.
\end{proof}

For $A \in \Bfrak(X)$, if $\I = \gen{A}$ then we use the notation $F_A$ and $f_A$ instead of $\EI$ and $\f$, respectively.

We now work towards strengthening the above lemma to yield a countably infinite partition into $F_A$-invariant Borel complete sections.

\begin{defn}[Aperiodic separation]
	For Borel sets $A, Y \subseteq X$, we say that $A$ aperiodically separates $Y$ if $f_A([Y]_G)$ is aperiodic (as an invariant subset of the shift $2^G$). If such $A$ exists, we say that $Y$ is aperiodically separable.
\end{defn}

\begin{prop}\label{E-invariant complete sections}
	For $A \in \Bfrak(X)$, if $A$ aperiodically separates $X$, then $X$ admits a countably infinite partition into Borel $F_A$-invariant complete sections.
\end{prop}
\begin{proof}
	Let $Y = \{y \in 2^G : |[y]_G| = \infty\}$ and hence $f_A(X)$ is a $G$-invariant subset of $Y$. By \cref{markers} applied to $Y$, there is a partition $\{B_n\}_{n \in \N}$ of $Y$ into Borel complete sections. Thus $A_n = f_{\I}^{-1}(B_n)$ is a Borel $F_A$-invariant complete section for $X$ and $\{A_n\}_{n \in \N}$ is a partition of $X$.
\end{proof}

Let $\Afrak$ denote the collection of all subsets of aperiodically separable Borel sets.

\begin{lemma}\label{A sigma-ideal}
	$\Afrak$ is a $\sigma$-ideal.
\end{lemma}
\begin{proof}
	We only have to show that if $Y_n$ are aperiodically separable Borel sets, then $Y = \bigcup_{n \in \N} Y_n \in \Afrak$.  Let $A_n$ be a Borel set aperiodically separating $Y_n$. Since $A_n$ also aperiodically separates $[Y_n]_G$ (by definition), we can assume that $Y_n$ is $G$-invariant. Furthermore, by taking $Y_n' = Y_n \setminus \bigcup_{k<n} Y_k$, we can assume that $Y_n$ are pairwise disjoint. Now letting $A = \bigcup_{n \in \N} (A_n \cap Y_n)$, it is easy to check that $A$ aperiodically separates $Y$.
\end{proof}

Let $\Sm$ denote the collection of all subsets of smooth sets. By a similar argument as the one above, $\Sm$ is a $\sigma$-ideal.

\begin{lemma}\label{smooth is in A}
	If $X$ is aperiodic, then $\Sm \subseteq \Afrak$.
\end{lemma}
\begin{proof}
	Let $S \in \Sm$ and hence there is a Borel transversal $T$ for $[S]_G$. Fix $x \in S$ and let $y \ne z \in [x]_G$. Since $T$ is a transversal, there is $g \in G$ such that $gy \in T$, and hence $gz \notin T$. Thus $f_T(y) \ne f_T(z)$, and so $f_T([x]_G)$ is infinite. Therefore $T$ aperiodically separates $[S]_G$.
\end{proof}

For the rest of this subsection, fix an enumeration $G = \{g_n\}_{n \in \N}$ and let $F_A^n$ be following equivalence relation:
$$
y F_A^n z \Leftrightarrow \forall k < n (g_k y \in A \leftrightarrow g_k z \in A).
$$
Note that $F_A^n$ has no more than $2^n$ equivalence classes and that $y F_A z$ if and only if $\forall n (y F_A^n z)$.

\begin{lemma}\label{criterion for aperiodically separating}
	For $A,Y \in \Bfrak(X)$, $A$ aperiodically separates $Y$ if and only if $(\forall x \in Y) (\forall n) (\exists y,z \in Y^{[x]_G}) [y F_A^n z \wedge \neg(y F_A z)]$.
\end{lemma}
\begin{proof}
	\noindent $\Rightarrow$: Assume that for all $x \in Y$, $f_A([x]_G)$ is infinite and thus $F_A \rest{[x]_G}$ has infinitely many equivalence classes. Fix $n \in \N$ and recall that $F_A^n$ has only finitely many equivalence classes. Thus, by the Pigeon Hole Principle, there are $y,z \in Y^{[x]_G}$ such that $y F_A^n z$ yet $\neg (y F_A z)$.
	
	\noindent $\Leftarrow$: Assume for contradiction that $f_A(Y^{[x]_G})$ is finite for some $x \in Y$. Then it follows that $F_A = F_A^n$, for some $n$, and hence for any $y,z \in Y^{[x]_G}$, $y F_A^n z$ implies $y F_A z$, contradicting the hypothesis.
\end{proof}

\begin{theorem}\label{X is aperiodically separable}
	If $X$ is an aperiodic Borel $G$-space, then $X \in \Afrak$.
\end{theorem}
\begin{proof}
	By \cref{markers}, there is a partition $\{A_n\}_{n \in \N}$ of $X$ into Borel complete sections. We will inductively construct Borel sets $B_n \subseteq C_n$, where $C_n$ should be thought of as the set of points \textit{colored} (black or white) at the $n^{th}$ step, and $B_n$ as the set of points colored \textit{black} (thus $C_n \setminus B_n$ is colored \textit{white}).
	
	Define a function $\# : X \rightarrow \N$ by $x \mapsto m$, where $m$ is such that $x \in A_m$. Fix a countable family $\{U_n\}_{n \in \N}$ of sets generating the Borel $\sigma$-algebra of $X$.
	
	Assuming that for all $k < n$, $C_k, B_k$ are defined, let $\bar{C}_n = \bigcup_{k<n} C_k$ and $\bar{B}_n = \bigcup_{k<n} B_k$. Put $P_n = \{x \in A_0 : \forall k < n (g_k x \in \bar{C}_n) \wedge g_n x \notin \bar{C}_n\}$ and set $F_n = F_{\bar{B}_n}^n \rest{P_n}$, that is for all $x,y \in P_n$,
	$$
	y F_n z \Leftrightarrow \forall k < n (g_k y \in \bar{B}_n \leftrightarrow g_k z \in \bar{B}_n).
	$$
	Now put $C'_n = \set{x \in P_n : \#(g_n x) = \min \#((g_nP_n)^{[x]_G})}$, $C''_n = \set{x \in C'_n : \exists y, z \in (C'_n)^{[x]_G} (y \ne z \wedge y F_n z)}$ and $C_n = g_n C''_n$. Note that it follows from the definition of $P_n$ that $C_n$ is disjoint from $\bar{C}_n$.
	
	Now in order to define $B_n$, first define a function $\n : X \rightarrow \N$ by
	$$
	x \mapsto \text{ the smallest $m$ such that there are } y,z \in C''_n \cap [x]_G \text{ with } y F_n z, y \in U_m \text{ and } z \notin U_m.
	$$
	Note that $\n$ is Borel and $G$-invariant. Lastly, let $B'_n = \{x \in C''_n : x \in U_{\n(x)}\}$ and $B_n = g_n B'_n$. Clearly $B_n \subseteq C_n$. Now let $B = \bigcup_{n \in \N} B_n$ and $D = \left[\bigcup_{n \in \N} (C'_n \setminus C''_n)\right]_G$. We show that $B$ aperiodically separates $Y := X \setminus D$ and $D \in \Sm$. Since $\Sm \subseteq \Afrak$ and $\Afrak$ is an ideal, this will imply that $X \in \Afrak$.
	
	\begin{claim}
		$D \in \Sm$.
	\end{claim}
	\begin{pfof}
		Since $\Sm$ is a $\sigma$-ideal, it is enough to show that for each $n$, $[C'_n \setminus C''_n]_G \in \Sm$, so fix $n  \in \N$. Clearly $(C'_n \setminus C''_n)^{[x]_G}$ is finite, for all $x \in X$, since there can be at most $2^n$ pairwise $F_n$-nonequivalent points. Thus, fixing some Borel linear ordering of $X$ and taking the smallest element from $(C'_n \setminus C''_n)^{[x]_G}$ for each $x \in C'_n \setminus C''_n$, we can define a Borel transversal for $[C'_n \setminus C''_n]_G$.
	\end{pfof}
	
	By \cref{criterion for aperiodically separating}, to show that $B$ aperiodically separates $Y$, it is enough to show that $(\forall x \in Y) (\forall n) (\exists y,z \in [x]_G) [y F_B^n z \wedge \neg(y F_B z)]$. Fix $x \in Y$.
	
	\begin{claim}
		$(\exists^{\infty} n) (C''_n)^{[x]_G} \neq \0$.
	\end{claim}
	\begin{pfof}
		Assume for contradiction that $(\forall^{\infty} n) (C''_n)^{[x]_G} = \0$. Since $x \notin D$, it follows that $(\forall^{\infty} n) P_n^{[x]_G} = \0$. Since $A_0$ is a complete section and $\bar{C}_0 = \0$, $P_0^{[x]_G} \neq \0$. Let $N$ be the largest number such that $P_N^{[x]_G} \ne \0$. Thus for all $n > N$, $C_n^{[x]_G} = \0$ and hence for all $n > N$, $\bar{C}_n^{[x]_G} = \bar{C}_{N+1}^{[x]_G}$. Because $C_N^{[x]_G} \ne \0$, there is $y \in A_0^{[x]_G}$ such that $\forall k \le N (g_k y \in \bar{C}_{N+1})$; but because $P_{N+1}^{[x]_G} = \0$, $g_{N+1} y$ must also fall into $\bar{C}_{N+1}$. By induction on $n > N$, we get that for all $n>N$, $g_n y \in \bar{C}_n$ and thus $g_n y \in \bar{C}_{N+1}$.
		
		On the other hand, it follows from the definition of $C'_n$ that for each $n$, $(C'_n)^{[x]_G}$ intersects exactly one of $A_k$. Thus $\bar{C}_{N+1}^{[x]_G}$ intersects at most $N+1$ of $A_k$ and hence there exists $K \in \N$ such that for all $k \geq K$, $\bar{C}_{N+1}^{[x]_G} \cap A_k = \0$. Since $\exists^{\infty} n (g_n y \in \bigcup_{k \geq K} A_k)$, $\exists^{\infty} n (g_n y \notin \bar{C}_{N+1})$, a contradiction.
	\end{pfof}
	
	Now it remains to show that for all $n \in \N$, $(C''_n)^{[x]_G} \neq \0$ implies that $\exists y,z \in [x]_G$ such that $y F_B^n z$ but $\neg (y F_B z)$. To this end, fix $n \in \N$ and assume $(C''_n)^{[x]_G} \neq \0$. Thus there are $y,z \in (C''_n)^{[x]_G}$ such that $y F_n z$, $y \in U_{\n(x)}$ and $z \notin U_{\n(x)}$; hence, $g_n y \in B_n$ and $g_n z \notin B_n$, by the definition of $B_n$. Since the $C_k$ are pairwise disjoint, $B_n \subseteq C_n$ and $g_n y, g_n z \in C_n$, it follows that $g_n y \in B$ and $g_n z \notin B$, and therefore $\neg (y F_B z)$. Finally, note that $F_n = F_B^n \rest{P_n}$ and hence $y F_B^n z$.
\end{proof}

It is worth explicitly stating the previous theorem in terms of the coding map $f_A$:

\begin{theorem}\label{map_to_aperiodic_part_of_shift}
	Any aperiodic Borel $G$-space admits a $G$-equivariant Borel map to the aperiodic part of the shift action $G \actson 2^G$.
\end{theorem}

\begin{cor}\label{G=Z implies separating orbits}
	For an aperiodic Borel $\Z$-space $X$, there is a Borel set $A \subseteq X$ such that $G \gen{A}$ separates points in each orbit, i.e. $f_A \rest{[x]_\Z}$ is one-to-one, for all $x \in X$.
\end{cor}
\begin{proof}
	Let $A$ be a Borel set aperiodically separating $X$ (exists by \cref{X is aperiodically separable}) and put $Y = f_A(X)$. Then $Y \subseteq 2^\Z$ is aperiodic and hence the action of $\Z$ on $Y$ is free. But this implies that for all $y \in Y$, $f_A^{-1}(y)$ intersects every orbit in $X$ at no more than one point, and hence $f_A$ is one-to-one on every orbit.
\end{proof}

From \cref{E-invariant complete sections,X is aperiodically separable} we immediately get the following strengthening of \cref{markers}.

\begin{cor}\label{E-invariant complete sections for X}
	Any aperiodic Borel $G$-space $X$ admits a countably infinite partition into Borel $F_A$-invariant complete sections, for some Borel set $A \subseteq X$.
\end{cor}

\subsection{Separating smooth-many invariant sets}

The following is a useful tool in constructing generators and we will apply it in the next section to the equivalence relation $E$ of being in the same component of the ergodic decomposition.

\begin{defn}[Smooth equivalence relations]\label{defn:smooth_eq_rel}
	An equivalence relation $E$ on a standard Borel space $X$ is called smooth if there is a Borel map $h : X \to \R$ reducing $E$ to the equality relation on $\R$, i.e. for each $x,y \in X$,
	$$
	x E y \iff h(x) = h(y).
	$$
\end{defn}

Note that since any two uncountable Polish spaces are Borel isomorphic (see \cite[Theorem 15.6]{bible}), in the above definition $\R$ can be replaced with any other uncountable Polish space. Also note that a smooth equivalence relation $E$ is automatically Borel (as a subset of $X^2$) because $E = \bar{h}^{-1}(\De(\R))$, where $\De(\R) = \set{(r,r) : r \in \R}$ and $\bar{h} : X^2 \to \R^2$ is defined by $(x,y) \mapsto (h(x), h(y))$.

It is a theorem of Burgess (see \cite{Burgess}) that for a Borel $G$-space $X$, the orbit equivalence relation $E_G$ is smooth if and only if the action $G \actson X$ is smooth in the sense of \cref{defn:smoothness_via_transversal}, that is: $X$ admits a Borel transversal.

For a Borel $G$-space $X$, an equivalence relation $E$ on $X$ is called \emph{$G$-invariant} if $E_G \subseteq E$; in other words, each $E$-class is a union of $G$-orbits.

\begin{theorem}\label{separating smooth-many invariant sets}
	Let $X$ be an aperiodic Borel $G$-space and let $E$ be a smooth $G$-invariant equivalence relation on $X$. There exists a partition $\Pa$ of $X$ into $4$ Borel sets such that $G \Pa$ separates any two $E$-nonequivalent points in $X$, i.e. for all $x, y \in X$, $[x]_E \ne [y]_E$ implies $f_{\Pa}(x) \neq f_{\Pa}(y)$.
\end{theorem}
\begin{proof}
	By \cref{E-invariant complete sections for X}, there is $A \in \Bfrak(X)$ and a Borel partition $\{A_n\}_{n \in \N}$ of $X$ into $F_A$-invariant complete sections. Fix an enumeration $G = \{g_n\}_{n \in \N}$, and for each $n \in \N$, define a function $\n : X \rightarrow \N$ by
	$$
	x \mapsto \text{the smallest $m$ such that } \exists x' \in A_0^{[x]_G} \text{ with } g_m x' \in A_n.
	$$
	Clearly $\n$ is Borel, and because all of $A_k$ are $F_A$-invariant, $\n$ is also $F_A$-invariant, i.e. for all $x,y \in X$, $x F_A y \rightarrow \n(x) = \n(y)$. Also, $\n$ is $G$-invariant by definition.
	
	Put $A'_n = \{x \in A_0 : g_{\n(x)} x \in A_n\}$ and note that $A'_n$ is $F_A$-invariant Borel since so are $\n$, $A_0$ and $A_n$. Moreover, $A'_n$ is clearly a complete section. Define $\ga_n : A'_n \rightarrow A_n$ by $x \mapsto g_{\n(x)} x$. Clearly, $\ga_n$ is Borel and one-to-one.
	
	Since $E$ is smooth, there is a Borel $h : X \rightarrow \R$ such that for all $x,y \in X$, $x E y \leftrightarrow h(x) = h(y)$. Let $\{V_n\}_{n \in \N}$ be a countable family of subsets of $\R$ generating the Borel $\sigma$-algebra of $\R$ and put $U_n = h^{-1}(V_n)$. Because each equivalence class of $E$ is $G$-invariant, so is $h$ and hence so is $U_n$.
	
	Now let $B_n = \ga_n(A'_n \cap U_n)$ and note that $B_n$ is Borel being a one-to-one Borel image of a Borel set. It follows from the definition of $\ga_n$ that $B_n \subseteq A_n$. Put $B = \bigcup_{n \in \N} B_n$ and $\Pa = \gen{A,B}$; in particular, $|\Pa| \leq 4$. We show that $\Pa$ is what we want. To this end, fix $x,y \in X$ with $\neg (x E y)$. If $\neg (x F_A y)$, then $G \gen{A}$ (and hence $G \Pa$) separates $x$ and $y$.
	
	Thus assume that $x F_A y$. Since $h(x) \neq h(y)$, there is $n$ such that $h(x) \in V_n$ and $h(y) \notin V_n$. Hence, by invariance of $U_n$, $gx \in U_n \wedge gy \notin U_n$, for all $g \in G$. Because $A'_n$ is a complete section, there is $g \in G$ such that $gx \in A'_n$ and hence $gy \in A'_n$ since $A'_n$ is $F_A$-invariant. Let $m = \n(gx)$ ($= \n(gy)$). Then $g_m gx \in B_n$ while $g_m gy \notin B_n$ although $g_m gy \in \ga_n(A'_n) \subseteq A_n$. Thus $g_m gx \in B$ but $g_m gy \notin B$ and therefore $G \Pa$ separates $x$ and $y$.
\end{proof}

\section{Potential dichotomy theorems}\label{section_dichotomy}

In this section we prove dichotomy theorems assuming Weiss's question has a positive answer for $G = \Z$. In the proofs we use the Ergodic Decomposition Theorem (see \cite{Farrell}, \cite{Varadarajan}) and a Borel/uniform version of Krieger's finite generator theorem, so we first state both of the theorems and sketch the proof of the latter.

For a Borel $G$-space $X$, let $\M_G(X)$ denote the set of $G$-invariant Borel probability measures on $X$ and let $\Erg_G(X)$ denote the set of ergodic ones among those. Clearly both are Borel subsets of $P(X)$ (the standard Borel space of Borel probability measures on $X$) and thus are themselves standard Borel spaces.

\begin{ergdecthm}[Farrell, Varadarajan]
	Let $X$ be a Borel $G$-space. If $\M_G(X) \ne \0$ (and hence $\Erg_G(X) \ne \0$), then there is a Borel surjection $x \mapsto e_x$ from $X$ onto $\Erg_G(X)$ such that:
	\begin{enumerate}[(i)]
		\item $x E_G y \Rightarrow e_x = e_y$;
		
		\item For each $e \in \Erg_G(X)$, if $X_e = \set{x \in X : e_x = e}$ (hence $X_e$ is invariant Borel), then $e(X_e) = 1$ and $e \rest{X_e}$ is the unique ergodic invariant Borel probability measure on $X_e$;
		
		\item For each $\mu \in \M_G(X)$ and $A \in \Bfrak(X)$, we have $\mu(A) = \int e_x(A) d\mu(x).$
	\end{enumerate}
\end{ergdecthm}

For the rest of the section, let $X$ be a Borel $\Z$-space.

For $e \in \ErgX$, if we let $h_e$ denote the entropy of $(X, \Z, e)$, then the map $e \mapsto h_e$ is Borel. Indeed, if $\set{\Pa_k}_{k \in \N}$ is a refining sequence of partitions of $X$ that generates the Borel $\sigma$-algebra of $X$, then, by \cite[4.1.2]{Downarowicz}, $h_e = \lim_{k \to \infty} h_e(\Pa_k, \Z)$, where $h_e(\Pa_k, \Z)$ denotes the entropy of $\Pa_k$. By \cite[17.21]{bible}, the function $e \mapsto h_e(\Pa_k)$ is Borel and thus so is the map $e \mapsto h_e$.

For all $e \in \ErgX$ with $h_e < \w$, let $N_e$ be the smallest integer such that $\log_2 N_e > h_e$. The map $e \mapsto N_e$ is Borel because so is $e \mapsto h_e$.

\begin{KriegersThm}[Uniform version]\label{uniform Krieger}
	Let $X$ be a Borel $\Z$-space. Suppose $\M_{\Z}(X) \ne \0$ and let $\rho$ be the map $x \mapsto e_x$ as in the Ergodic Decomposition Theorem. Assume also that all measures in $\ErgX$ have finite entropy and let $e \mapsto N_e$ be the map defined above. Then there is a partition $\set{A_n}_{n \le \infty}$ of $X$ into Borel sets such that
	\begin{enumerate}[(i)]
		\item $A_{\infty}$ is invariant and does not admit an invariant Borel probability measure;
		
		\item For each $e \in \ErgX$, $\set{A_n \cap X_e}_{n < N_e}$ is a generator for $X_e \setminus A_{\infty}$, where $X_e = \rho^{-1}(e)$.
	\end{enumerate}
\end{KriegersThm}
\begin{skofpf}
	Note that it is enough to find a Borel invariant set $X' \subseteq X$ and a Borel $\Z$-map $\phi : X' \to \N^{\Z}$, such that for each $e \in \ErgX$, we have
	\begin{enumerate}[(I)]
		\item $e(X \setminus X') = 0$;
		
		\item $\phi \rest{X_e \cap X'}$ is one-to-one and $\phi(X_e \cap X') \subseteq (N_e)^{\Z}$, where $(N_e)^{\Z}$ is naturally viewed as a subset of $\N^{\Z}$.
	\end{enumerate}
	Indeed, assume we had such $X'$ and $\phi$, and let $A_{\infty} = X \setminus X'$ and $A_n = \phi^{-1}(V_n)$ for all $n \in \N$, where $V_n = \set{y \in \N^{\Z} : y(0) = n}$. Then it is clear that $\set{A_n}_{n \in \N}$ satisfies (ii). Also, (I) and part (iii) of the Ergodic Decomposition Theorem imply that (i) holds for $A_{\infty}$.
	
	To construct such a $\phi$, we use the proof of Krieger's theorem presented in \cite[Theorem 4.2.3]{Downarowicz}, and we refer to it as Downarowicz's proof. For each $e \in \ErgX$, the proof constructs a Borel $\Z$-embedding $\phi_e : X' \to N_e^{\Z}$ on an $e$-measure $1$ set $X'$. We claim that this construction is uniform in $e$ in a Borel way and hence would yield $X'$ and $\phi$ as above.
	
	Our claim can be verified by inspection of Downarowicz's proof. The proof uses the existence of sets with certain properties and one has to check that such sets exist with the properties satisfied for all $e \in \ErgX$ at once. For example, the set $C$ used in \cite[proof of Lemma 4.2.5]{Downarowicz} can be chosen so that for all $e \in \ErgX$, $C \cap X_e$ has the required properties for $e$ (using the Shannon--McMillan--Brieman theorem). Another example is the set $B$ used in the proof of the same lemma, which is provided by Rohlin's lemma. By inspection of the proof of Rohlin's lemma (see \cite[2.1]{Glasner}), one can verify that we can get a Borel $B$ such that for all $e \in \ErgX$, $B \cap X_e$ has the required properties for $e$. The sets in these two examples are the only kind of sets whose existence is used in the whole proof; the rest of the proof constructs the required $\phi$ ``by hand''.
\end{skofpf}

\begin{theorem}[Dichotomy I]\label{dichotomyI}
	Suppose the answer to \cref{Weiss's question} is positive and let $X$ be an aperiodic Borel $\Z$-space. Then exactly one of the following holds:
	\begin{enumerate}[(1)]
		\item there exists an invariant ergodic Borel probability measure with infinite entropy;
		
		\item there exists a partition $\set{Y_n}_{n \in \N}$ of $X$ into invariant Borel sets such that each $Y_n$ has a finite generator.
	\end{enumerate}
\end{theorem}
\begin{proof}
	We first show that the conditions above are mutually exclusive. Indeed, assume there exist an invariant ergodic Borel probability measure $e$ with infinite entropy and a partition $\set{Y_n}_{n \in \N}$ of $X$ into invariant Borel sets such that each $Y_n$ has a finite generator. By ergodicity, $e$ would have to be supported on one of the $Y_n$. But $Y_n$ has a finite generator and hence the dynamical system $(Y_n, \Z, e)$ has finite entropy by the Kolmogorov--Sinai theorem (see \labelcref{Kolmogorov--Sinai}). Thus so does $(X, \Z, e)$ since these two systems are isomorphic (modulo $e$-$\NULL$), contradicting the assumption on $e$.
	
	Now we prove that at least one of the conditions holds. Assume that there is no invariant ergodic measure with infinite entropy. Now, if there was no invariant Borel probability measure at all, then, since the answer to \cref{Weiss's question} is assumed to be positive, $X$ would admit a finite generator, and we would be done. So assume that $\M_{\Z}(X) \ne \0$ and let $\set{A_n}_{n \le \infty}$ be as in \cref{uniform Krieger}. Furthermore, let $\rho$ be the map $x \mapsto e_x$ as in the Ergodic Decomposition Theorem. Set $X' = X \setminus A_{\infty}$, $Y_{\infty} = A_{\infty}$, and for all $n \in \N$,
	$$
	Y_n = \set{x \in X' : N_{e_x} = n},
	$$
	where the map $e \mapsto N_e$ is as above. Note that the sets $Y_n$ are invariant since $\rho$ is invariant, so $\set{Y_n}_{n \le \w}$ is a countable partition of $X$ into invariant Borel sets. Since $Y_{\infty}$ does not admit an invariant Borel probability measure, by our assumption, it has a finite generator.
	
	Let $E$ be the equivalence relation on $X'$ defined by $\rho$, i.e. $\forall x,y \in X'$,
	$$
	x E y \Leftrightarrow \rho(x) = \rho(y).
	$$
	By definition, $E$ is a smooth Borel equivalence relation with $E \supseteq E_{\Z}$ since $\rho$ respects the $\Z$-action. Thus, by \cref{separating smooth-many invariant sets}, there exists a partition $\Pa$ of $X'$ into $4$ Borel sets such that $\Z \Pa$ separates any two points in different $E$-classes.
	
	Now fix $n \in \N$ and we will show that $\I = \Pa \vee \set{A_i}_{i < n}$ is a generator for $Y_n$. Indeed, take distinct $x,y \in Y_n$. If $x$ and $y$ are in different $E$-classes, then $\Z \Pa$ separates them and hence so does $\Z \I$. Thus we can assume that $x E y$. Then $e := \rho(x) = \rho(y)$, i.e. $x,y \in X_e = \rho^{-1}(e)$. By the choice of $\set{A_i}_{i \in \N}$, $\set{A_n \cap X_e}_{n < N_e}$ is a generator for $X_e$ and hence $\Z \set{A_i}_{i < N_e}$ separates $x$ and $y$. But $n = N_e$ by the definition of $Y_n$, so $\Z \I$ separates $x$ and $y$.
\end{proof}

\begin{prop}\label{arbitrarily large entropy}
	Let $X$ be a Borel $\Z$-space. If $X$ admits invariant ergodic probability measures of arbitrarily large entropy, then it admits an invariant probability measure of infinite entropy.
\end{prop}
\begin{proof}
	For each $n \ge 1$, let $\mu_n$ be an invariant ergodic probability measure of entropy $h_{\mu_n} > n 2^n$ such that $\mu_n \ne \mu_m$ for $n \ne m$, and put
	$$
	\mu = \sum_{n \ge 1} {1 \over 2^n} \mu_n.
	$$
	It is clear that $\mu$ is an invariant probability measure, and we show that its entropy $h_{\mu}$ is infinite using an argument pointed out by the referee of the current paper (the author's original argument was less general).
	
	\begin{claim*}
		For any invariant probability measures $\nu, \mu$ on $X$ and $c \le 1$, if $c \nu \le \mu$ then $c h_\nu \le h_\mu$.
	\end{claim*}
	\begin{pfof}
		First note that the entropy of any atomic invariant probability measure $\rho$ is $0$: indeed, such a measure is supported on countably many orbits, each of which is finite, and hence for any partition $\I$, the sequence $h_\rho(\bigvee_{i=-n}^n T^i \I)$ converges as $n \to \w$; therefore, the time-average entropy (see \labelcref{eq:dynamic_entropy}) is $0$. It follows now that the nonatomic parts of $\nu,\mu$ do not contribute in the calculation of their entropies, so we may assume without loss of generality that $\nu,\mu$ are nonatomic.
		
		The function $g(x) = - x \log_2 x$ used in the calculation of the entropy is convex on $(0,1)$ and satisfies $c g(x) \le g(cx) \le g(y)$ for $0 < x,y < 1/2$ with $cx \le y$. In \labelcref{eq:static_entropy}, $x$ is equal to the measure of a piece of a partition; since, by the convexity of $g(x)$, refining a partition only increases its static entropy and our measures are nonatomic, in \labelcref{eq:sup_entropy} we can take the supremum merely over the partitions with pieces of measure less than $1/2$. This yields $c h_\nu \le h_{\mu}$.
	\end{pfof}
	
	For any $n \ge 1$, we have ${1 \over 2^n} \mu_n \le \mu$, so the above claim gives $h_\mu \ge {1 \over 2^n} h_{\mu_n} \ge n$ and hence $h_\mu = \w$.
\end{proof}

\begin{theorem}[Dichotomy II]\label{dichotomyII}
	Suppose the answer to \cref{Weiss's question} is positive and let $X$ be an aperiodic Borel $\Z$-space. Then exactly one of the following holds:
	\begin{enumerate}[(1)]
		\item there exists an invariant Borel probability measure with infinite entropy;
		
		\item $X$ admits a finite generator.
	\end{enumerate}
\end{theorem}
\begin{proof}
	The Kolmogorov--Sinai theorem implies that the conditions are mutually exclusive, and we prove that at least one of them holds. Assume that there is no invariant measure with infinite entropy. If there was no invariant Borel probability measure at all, then, by our assumption, $X$ would admit a finite generator. So assume that $\M_{\Z}(X) \ne \0$ and let $\set{A_n}_{n \le \infty}$ be as in \cref{uniform Krieger}. Furthermore, let $\rho$ be the map $x \mapsto e_x$ as in the Ergodic Decomposition Theorem. Set $X' = X \setminus A_{\infty}$ and $X_e = \rho^{-1}(e)$, for all $e \in \ErgX$.
	
	By our assumption, $A_{\w}$ admits a finite generator $\Pa$. Also, by \cref{arbitrarily large entropy}, there is $N \ge 1$ such that for all $e \in \ErgX$, $N_e \le N$ and hence $\Qa := \set{A_n}_{n < N}$ is a finite generator for $X_e$; in particular, $\Qa$ is a partition of $X'$. Let $E$ be the following equivalence relation on $X$:
	$$
	x E y \Leftrightarrow (x, y \in A_{\w}) \vee (x,y \in X' \wedge \rho(x) = \rho(y)).
	$$
	By definition, $E$ is a smooth equivalence relation with $E \supseteq E_{\Z}$ since $\rho$ respects the $\Z$-action and $A_{\w}$ is $\Z$-invariant. Thus, by \cref{separating smooth-many invariant sets}, there exists a partition $\J$ of $X$ into $4$ Borel sets such that $\Z \J$ separates any two points in different $E$-classes.
	
	We now show that $\I := \gen{\J \cup \Pa \cup \Qa}$ is a generator. Indeed, fix distinct $x,y \in X$. If $x$ and $y$ are in different $E$-classes, then $\Z \J$ separates them. So we can assume that $x E y$. If $x,y \in A_{\w}$, then $\Z \Pa$ separates $x$ and $y$. Finally, if $x,y \in X'$, then $x,y \in X_e$, where $e = \rho(x)$ ($= \rho(y)$), and hence $\Z \Qa$ separates $x$ and $y$.
\end{proof}

\begin{remark}
	It is likely that the above dichotomies are also true for any amenable group using a uniform version of Krieger's theorem for amenable groups (see \cite{DP}), but the author has not checked the details.
\end{remark}

\section{Finite generators on comeager sets}\label{section_finite gen modulo meager}

This section is devoted to the proof of the following:

\begin{theorem}\label{4-generator modulo a meager set}
	Any aperiodic Polish $G$-space admits a $4$-generator on an invariant comeager set.
\end{theorem}

Throughout this section, let $X$ be an aperiodic Polish $G$-space. We use the notation $\forall^* x$ to mean ``for comeager many $x$''.

Having advertised the Kuratowski--Ulam method in the introduction, let us point out that a ``blind'' application of it would not give us the statement of the above theorem. Indeed, assume for a moment that we have found a parametrized construction of finite partitions $\Pa_\al$, for $\al \in \N^\N$, and let
$$
\Phi(\Pa_\al, x, y) :\shortiff \text{``if $x \ne y$, then $G \Pa_\al$ separates $x$ and $y$''}.
$$
If we apply the Kuratowski--Ulam method to this $\Phi$, we will get that for comeager many $\al \in \N^\N$, we have:
$$
\forall^* (x,y) \in X^2 \ \Phi(\Pa_\al, x, y),
$$
while we want a comeager set $D \subseteq X$ such that
$$
\forall (x,y) \in D^2 \ \Phi(\Pa_\al, x, y).
$$
The problem is that a $2$-dimensional comeager set may not contain a square of a $1$-dimensional comeager set. To get around this, we transform our $2$-dimensional problem into two $1$-dimensional problems, and here is the first of them:

\begin{lemma}\label{separating orbits on a comeager set}
	There exists $A \in \Bfrak(X)$ such that $G \gen{A}$ separates points in each orbit of a comeager $G$-invariant set $D$, i.e. for each $x \in D$, the restriction of the coding map $f_{\gen{A}}$ to $[x]_G$ is one-to-one.
\end{lemma}
\begin{proof}
	Fix a countable basis $\{U_n\}_{n \in \N}$ for $X$ with $U_0 = \emptyset$ and let $\{A_n\}_{n \in \N}$ be a partition of $X$ provided by \cref{markers}. For each $\alpha \in \Br$ (the Baire space), define
	$$
	B_{\alpha} = \Uni (A_n \cap U_{\alpha(n)}).
	$$
	
	\begin{claim*}
		$\forall^* \alpha \in \Br \forall^* z \in X \forall x,y \in [z]_G (x \ne y \Rightarrow \exists g \in G (gx \in B_{\alpha} \nLeftrightarrow gy \in B_{\alpha}))$.
	\end{claim*}
	\begin{pfof}
		By Kuratowski--Ulam, it is enough to show the statement with places of the quantifiers $\forall^* \alpha \in \Br$ and $\forall^* z \in X$ switched. Also, since orbits are countable and a countable intersection of comeager sets is comeager, we can also switch the places of the quantifiers $\forall^* \alpha \in \Br$ and $\forall x,y \in [z]_G$. Thus we fix $z \in X$ and $x,y \in [z]_G$ with $x \neq y$ and show that $C = \{\alpha \in \Br : \exists g \in G \ (gx \in B_{\alpha} \nLeftrightarrow gy \in B_{\alpha})\}$ is dense open.
		
		To see that $C$ is open, take $\alpha \in C$ and let $g \in G$ be such that $gx \in B_{\alpha} \nLeftrightarrow gy \in B_{\alpha}$. Let $n,m \in \N$ be such that $gx \in A_n$ and $gy \in A_m$. Then for all $\beta \in \Br$ with $\beta(n) = \alpha(n)$ and $\beta(m) = \alpha(m)$, we have $gx \in B_{\beta} \nLeftrightarrow gy \in B_{\beta}$. But the set of such $\beta$ is open in $\Br$ and contained in $C$.
		
		For the density of $C$, let $s \in \QB$ and set $n = |s|$. Since $A_n$ is a complete section, $\exists g \in G$ with $gx \in A_n$. Let $m \in \N$ be such that $gy \in A_m$. Take any $t \in \N^{\max\{n,m\}+1}$ with $t \sqsupseteq s$ satisfying the following condition:
		
		\noindent Case 1: $n > m$. If $gy \in U_{s(m)}$ then set $t(n) = 0$. If $gy \notin U_{s(m)}$, then let $k$ be such that $gx \in U_k$ and set $t(n) = k$.
		
		\noindent Case 2: $n \leq m$. Let $k$ be such that $gx \in U_k$ but $gy \notin U_k$ and set $t(n) = t(m) = k$.
		
		Now it is easy to check that in any case $gx \in B_{\alpha} \nLeftrightarrow gy \in B_{\alpha}$, for any $\alpha \in \Br$ with $\alpha \sqsupseteq t$, and so $\alpha \in C$ and $\alpha \sqsupseteq s$. Hence $C$ is dense.
	\end{pfof}
	
	By the claim, $\exists \alpha \in \Br$ such that 
	$$
	D := \set{z \in X : \forall x,y \in [z]_G \text{ with } x \neq y, \ G \gen{B_{\alpha}} \text{ separates $x$ and $y$}}
	$$ 
	is comeager and clearly invariant, which completes the proof.
\end{proof}

The reader is invited to compare this last lemma with \cref{X is aperiodically separable} and \cref{G=Z implies separating orbits}. In fact, in the proof of \cref{4-generator modulo a meager set} below, we only use \cref{separating orbits on a comeager set} to deduce the conclusion of \cref{X is aperiodically separable} modulo $\MEAGER$, which we could also deduce from \cref{X is aperiodically separable} itself. However, we still included \cref{separating orbits on a comeager set} here to keep this section self-contained, and also because its proof is easier than that of \cref{X is aperiodically separable}.

\begin{proof}[Proof of \cref{4-generator modulo a meager set}]
	Let $A$ and $D$ be provided by \cref{separating orbits on a comeager set}. Throwing away an invariant meager set from $D$, we may assume that $D$ is dense $G_{\de}$ and hence Polish in the relative topology. Therefore, we may assume without loss of generality that $X = D$.
	
	Thus $A$ aperiodically separates $X$ and hence, by \cref{E-invariant complete sections}, there is a partition $\{A_n\}_{n \in \N}$ of $X$ into $F_A$-invariant Borel complete sections (the latter could be inferred directly from \cref{E-invariant complete sections for X} without using \cref{separating orbits on a comeager set}). Fix an enumeration $G = \{g_n\}_{n \in \N}$ and a countable basis $\{U_n\}_{n \in \N}$ for $X$. Denote $\Brr = (\N^2)^{\N}$ and for each $\alpha \in \Brr$, define
	$$
	B_{\alpha} = \Uone (A_n \cap g_{(\alpha(n))_0}U_{(\alpha(n))_1}).
	$$
	
	\begin{claim*}
		$\forall^* \alpha \in \Brr \forall^* x \in X \forall l \in \N \exists n,k \in \N (\alpha(n) = (k,l) \wedge g_k x \in A_n)$.
	\end{claim*}
	\begin{pfof}
		By Kuratowski--Ulam, it is enough to show that $\forall x \in X$ and $\forall l \in \N$, $C = \{\alpha \in \Brr : \exists k,n \in \N (\alpha(n) = (k,l) \wedge g_k x \in A_n)\}$ is dense open.
		
		To see that $C$ is open, note that for fixed $n,k,l \in N$, $\alpha(n) = (k,l)$ is an open condition in $\Brr$.
		
		For the density of $C$, let $s \in (\N^2)^{<\N}$ and set $n = |s|$. Since $A_n$ is a complete section, $\exists k \in \N$ with $g_k x \in A_n$. Any $\alpha \in \Brr$ with $\alpha \sqsupseteq s$ and $\alpha(n) = (k,l)$ belongs to $C$. Hence $C$ is dense.
	\end{pfof}
	
	By the claim, there exists $\alpha \in \Brr$ such that $Y = \{x \in X : \forall l \in \N \ \exists k,n \in \N \ (\alpha(n) = (k,l) \wedge g_k x \in A_n)\}$ is comeager.
	
	Let $\I = \gen{A, B_{\alpha}}$, and so $|\I| \le 4$. We show that $\I$ is a generator on $Y$. Fix distinct $x,y \in Y$. If $x$ and $y$ are separated by $G \gen{A}$ then we are done, so assume otherwise, that is $x F_A y$. Let $l \in \N$ be such that $x \in U_l$ but $y \notin U_l$. Then there exists $k,n \in \N$ such that $\alpha(n) = (k,l)$ and $g_k x \in A_n$. Since $g_k x F_A g_k y$ and $A_n$ is $F_A$-invariant, $g_k y \in A_n$. Furthermore, since $g_k x \in A_n \cap g_k U_l$ and $g_k y \notin A_n \cap g_k U_l$, $g_k x \in B_{\alpha}$ while $g_k y \notin B_{\alpha}$. Hence $G \gen{B_{\alpha}}$ separates $x$ and $y$, and thus so does $G\I$. Therefore $\I$ is a generator.
\end{proof}

\begin{cor}
	Let $X$ be a Polish $G$-space. If $X$ is aperiodic, then it is $2$-compressible modulo $\MEAGER$.
\end{cor}
\begin{proof}
	By \cite[Theorem 13.1]{KM}, $X$ is compressible modulo $\MEAGER$. Also, by the above theorem, $X$ admits a $4$-generator modulo $\MEAGER$. Thus \cref{n-generator => log(n)-compressible} implies that $X$ is $2$-compressible modulo $\MEAGER$.
\end{proof}

\section{The nonexistence of non-meager weakly wandering sets}

Throughout this section, let $X$ be a Polish $\Z$-space and $T$ be the homeomorphism corresponding to the action of $1 \in \Z$.

\subsection{An arithmetic criterion}

\begin{obs}\label{closure properties of ww sets}
Let $A \subseteq X$ be weakly wandering with respect to $H \subseteq \Z$. Then $A$ is weakly wandering with respect to
\begin{enumerate}[(a)]
\item any subset of $H$;
\item $r+H$, $\forall r \in \Z$;
\item $-H$.
\end{enumerate}
\end{obs}

\begin{defn}
Let $d \geq 1$ and $F = \set{n_i}_{i<k} \subseteq \Z$, where $n_0 < n_1 < ... < n_{k-1}$ are increasing. $F$ is called $d$-syndetic if $n_{i+1} - n_i \leq d$ for all $i < k-1$. In this case we say that the length of $F$ is $n_{k-1}-n_0$ and denote it by $||F||$.
\end{defn}

\begin{lemma}\label{php}
Let $d \geq 1$ and $F \subseteq \Z$ be a $d$-syndetic set. For any $H \subseteq \Z$, if $|H| = d+1$ and $\max(H) - \min(H) < ||F|| + d$, then $F$ is not weakly wandering with respect to $H$ (viewing $\Z$ as a $\Z$-space).
\end{lemma}
\begin{proof}
Using (b) and (c) of \cref{closure properties of ww sets}, we may assume that $H$ is a set of non-negative numbers containing $0$. Let $F = \set{n_i}_{i<k}$ with $n_i$ increasing.

\begin{claim*}
$\forall h \in H$, $(h + F) \cap [n_{k-1}, n_{k-1} + d) \neq \0$.
\end{claim*}
\begin{pfof}
Fix $h \in H$. Since $0 \leq h < ||F|| + d$,
$$
n_0 + h < n_0 + (||F|| + d) = n_{k-1} + d.
$$
We prove that there is $0 \leq i \leq k-1$ such that $n_i + h \in [n_{k-1}, n_{k-1} + d)$. Otherwise, because $n_{i+1} - n_i \leq d$, one can show by induction on $i$ that $n_i + h < n_{k-1}, \forall i < k$, contradicting $n_{k-1} + h \geq n_{k-1}$.
\end{pfof}

Now $|H| = d+1 > d = |\Z \cap [n_{k-1}, n_{k-1} + d)|$, so by the Pigeon Hole Principle there exists $h \neq h' \in H$ such that $(h + F) \cap (h' + F) \neq \0$ and hence $F$ is not weakly wandering with respect to $H$.
\end{proof}

\begin{defn}
Let $d,l \geq 1$ and $A \subseteq X$. We say that $A$ contains a $d$-syndetic set of length $l$ if there exists $x \in X$ such that $\{n \in \Z : T^n(x) \in A\}$ contains a $d$-syndetic set of length $\geq l$. This is equivalent to $\bigcap_{n \in F} T^n(A) \neq \0$, for some $d$-syndetic set $F \subseteq \Z$ of length $\geq l$.
\end{defn}

For $A \subseteq X$, define $s_A : \N \rightarrow \N \cup \{\infty\}$ by
$$
d \mapsto \sup\{l \in \N : A \text{ contains a } d\text{-syndetic set of length } l\}.
$$
Also, for infinite $H \subseteq \Z$, define a width function $w_H : \N \rightarrow \N$ by
$$
d \mapsto \min\{\max(H') - \min(H') : H' \subseteq H \wedge |H'| = d+1\}.
$$

\begin{prop}\label{syndeticity inequality for ww sets}
If $A \subseteq X$ is weakly wandering with respect to an infinite $H \subseteq \Z$ then $\forall d \in \N, s_A(d) + d \leq w_H(d)$.
\end{prop}
\begin{proof}
Let $H$ be an infinite subset of $\Z$ and $A \subseteq X$, and assume that $s_A(d) + d > w_H(d)$ for some $d \in \N$. Thus $\exists x \in X$ such that $\{n \in \Z : T^n(x) \in A\}$ contains a $d$-syndetic set $F$ of length $l$ with $l + d > w_H(d)$ and $\exists H' \subseteq H$ such that $|H'| = d+1$ and $\max(H') - \min(H') = w_H(d)$. By \cref{php} applied to $F$ and $H'$, $F$ is not weakly wandering with respect to $H'$ and hence neither is $A$. Thus $A$ is not weakly wandering with respect to $H$.
\end{proof}

\begin{cor}\label{arbitrarily long syndetic => not ww}
If $A \subseteq X$ contains arbitrarily long $d$-syndetic sets for some $d \geq 1$, then it is not weakly wandering.
\end{cor}
\begin{proof}
If $A$ and $d$ are as in the hypothesis, then $s_A(d) = \infty$ and hence, by \cref{syndeticity inequality for ww sets}, $A$ is not weakly wandering with respect to any infinite $H \subseteq \Z$.
\end{proof}

\begin{theorem}\label{syndetic open sets => no ww set}
Let $X$ be a Polish $G$-space. Suppose for every nonempty open $V \subseteq X$ there exists $d \geq 1$ such that $V$ contains arbitrarily long $d$-syndetic sets, i.e. $\bigcap_{n \in F} T^n(V) \neq \0$ for arbitrarily long $d$-syndetic sets $F \subseteq \Z$. Then $X$ does not admit a non-meager Baire measurable weakly wandering subset.
\end{theorem}
\begin{proof}
Let $A$ be a non-meager Baire measurable subset of $X$. By the Baire property, there exists a nonempty open $V \subseteq X$ such that $A$ is comeager in $V$. By the hypothesis, there exists arbitrarily long $d$-syndetic sets $F \subseteq \Z$ such that $\bigcap_{n \in F} T^n(V) \neq \0$. Since $A$ is comeager in $V$ and $T$ is a homeomorphism, $\bigcap_{n \in F} T^n(A)$ is comeager in $\bigcap_{n \in F} T^n(V)$, and hence $\bigcap_{n \in F} T^n(A) \neq \0$ for any $F$ for which $\bigcap_{n \in F} T^n(V) \neq \0$. Thus $A$ also contains arbitrarily long $d$-syndetic sets and hence, by \cref{arbitrarily long syndetic => not ww}, $A$ is not weakly wandering.
\end{proof}

\begin{cor}\label{finite intersection property for open sets => no ww set}
Let $X$ be a Polish $G$-space. Suppose for every nonempty open $V \subseteq X$ there exists $d \geq 1$ such that $\{T^{nd}(V)\}_{n \in \N}$ has the finite intersection property. Then $X$ does not admit a non-meager Baire measurable weakly wandering subset.
\end{cor}
\begin{proof}
Fix nonempty open $V \subseteq X$ and let $d \geq 1$ such that $\{T^{nd}(V)\}_{n \in \N}$ has the finite intersection property. Then for every $N$, $F = \{kd : k \leq N\}$ is a $d$-syndetic set of length $Nd$ and $\bigcap_{n \in F} T^n(V) \neq \0$. Thus \cref{syndetic open sets => no ww set} applies.
\end{proof}

\subsection{A negative answer to the Eigen--Hajian--Nadkarni question}

\begin{lemma}\label{generic lww implies generic ww}
Let $X$ be a generically ergodic Polish $G$-space. If there is a non-meager Baire measurable locally weakly wandering subset then there is a non-meager Baire measurable weakly wandering subset.
\end{lemma}
\begin{proof}
Let $A$ be a non-meager Baire measurable locally weakly wandering subset. By generic ergodicity, we may assume that $X = [A]_G$. Throwing away a meager set from $A$ we can assume that $A$ is $G_{\de}$. Then, by (d) of \cref{travel guide for lww}, there exists a $\s$-measurable (and hence Baire measurable) $G$-invariant travel guide $\ga : A \rightarrow \GN$. By generic ergodicity, $\ga$ must be constant on a comeager set, i.e. there is $(g_n)_{n \in \N} \in \GN$ such that $Y := \ga^{-1}((g_n)_{n \in \N})$ is comeager. But then $W := A \cap Y$ is non-meager and is weakly wandering with respect to $\set{g_n}_{n \in \N}$.
\end{proof}

Let $X = \{\alpha \in 2^{\N} : \alpha \text{ has infinitely many 0-s and 1-s}\}$ and $T$ be the odometer transformation on $X$. We will refer to this $\Z$-space as the odometer space.

\begin{cor}\label{no generic lww for odometer}
The odometer space does not admit a non-meager Baire measurable locally weakly wandering subset.
\end{cor}
\begin{proof}
Let $\{U_s\}_{s \in 2^{<\N}}$ be the standard basis. Then for any $s \in 2^{<\N}$, $T^{d}(U_s) = U_s$, where $d = 2^{|s|}$. Thus $\{T^{nd}(U_s)\}_{n \in \N}$ has the finite intersection property, in fact $\bigcap_{n \in \N} T^{nd}(U_s) = U_s$. Hence, we are done by \cref{finite intersection property for open sets => no ww set,generic lww implies generic ww}.
\end{proof}

The following corollary shows the failure of the analogue of the Hajian--Kakutani--It\^{o} theorem in the context of Baire category as well as gives a negative answer to \cref{EHN's question}.
\begin{cor}\label{failure of having lww mod meager}
There exists a generically ergodic Polish $\Z$-space $Y$ (namely an invariant dense $G_{\de}$ subset of the odometer space) with the following properties:
\begin{enumerate}[(i)]
\item there does not exist an invariant Borel probability measure on $Y$;

\item there does not exist a non-meager Baire measurable locally weakly wandering set;

\item there does not exist a Baire measurable countably generated partition of $Y$ into invariant sets, each of which admits a Baire measurable weakly wandering complete section.
\end{enumerate}
\end{cor}
\begin{proof}
By the Kechris--Miller theorem (see \labelcref{Kechris-Miller thm}), there exists an invariant dense $G_{\de}$ subset $Y$ of the odometer space that does not admit an invariant Borel probability measure. Now (ii) is asserted by \cref{no generic lww for odometer}. By generic ergodicity of $Y$, for any Baire measurable countably generated partition of $Y$ into invariant sets, one of the pieces of the partition has to be comeager. But then that piece does not admit a Baire measurable weakly wandering complete section since otherwise it would be non-meager, contradicting (ii).
\end{proof}


\begin{bibdiv}
\begin{biblist}

\bib{BK}{book}
{
author = {Becker, H.}
author = {Kechris, A. S.}
title = {The Descriptive Set Theory of Polish Group Actions}
date = {1996}
publisher = {Cambridge Univ. Press}
series = {London Math. Soc. Lecture Note Series}
volume = {232}
}

\bib{Burgess}{article}
{
	author = {Burgess, J. P.}
	title = {A selection theorem for group actions}
	date = {1979}
	journal = {Pac. J. Math.}
	volume = {80}
	pages = {333--336}
}

\bib{CKM}{article}
{
	author = {Conley, C. T.}
	author = {Kechris, A. S.}
	author = {Miller, B. D.}
	title = {Stationary probability measures and topological realizations}
	date = {2013}
	journal = {Israel. J. Math.}
	volume = {198}
	number = {1}
	pages = {333--345}
}

\bib{DP}{article}
{
author = {Danilenko, A. I.}
author = {Park, K. K.}
title = {Generators and Bernoullian factors for amenable actions and cocycles on their orbits}
date = {2002}
journal = {Ergodic Theory and Dynamical Systems}
volume = {22}
pages = {1715--1745}
}

\bib{Downarowicz}{book}
{
author = {Downarowicz, T.}
title = {Entropy in Dynamical Systems}
date = {2011}
publisher = {Cambridge Univ. Press}
series = {New Mathematical Monographs Series}
volume = {18}
}

\bib{EHN}{article}
{
author = {Eigen, S.}
author = {Hajian, A.}
author = {Nadkarni, M.}
title = {Weakly wandering sets and compressibility in a descriptive setting}
date = {1993}
journal = {Proc. Indian Acad. Sci.}
volume = {103}
number = {3}
pages = {321--327}
}

\bib{Farrell}{article}
{
author = {Farrell, R. H.}
title = {Representation of invariant measures}
date = {1962}
journal = {Illinois J. Math.}
volume = {6}
pages = {447--467}
}

\bib{Glasner}{book}
{
author = {Glasner, E.}
title = {Ergodic Theory via Joinings}
date = {2003}
publisher = {American Mathematical Society}
series = {Mathematical Surveys and Monographs}
volume = {101}
}

\bib{GW}{article}
{
author = {Glasner, E.}
author = {Weiss, B.}
title = {Minimal actions of the group $S(\Z)$ of permutations of the integers}
date = {2002}
journal = {Geom. Funct. Anal.}
volume = {12}
pages = {964--988}
}

\bib{HI}{article}
{
author = {Hajian, A. B.}
author = {It\^{o}, Y.}
title = {Weakly wandering sets and invariant measures for a group of transformations}
date = {1969}
journal = {Journal of Math. Mech.}
volume = {18}
pages = {1203--1216}
}

\bib{HK}{article}
{
author = {Hajian, A. B.}
author = {Kakutani, S.}
title = {Weakly wandering sets and invariant measures}
date = {1964}
journal = {Trans. Amer. Math. Soc.}
volume = {110}
pages = {136--151}
}

\bib{JKL}{article}
{
author = {Jackson, S.}
author = {Kechris, A. S.}
author = {Louveau, A.}
title = {Countable Borel equivalence relations}
date = {2002}
journal = {Journal of Math. Logic}
volume = {2}
number = {1}
pages = {1--80}
}

\bib{bible}{book}
{
author = {Kechris, A. S.}
title = {Classical Descriptive Set Theory}
date = {1995}
publisher = {Springer}
series = {Graduate Texts in Mathematics}
volume = {156}
}

\bib{KM}{book}
{
author = {Kechris, A. S.}
author = {Miller, B.}
title = {Topics in Orbit Equivalence}
date = {2004}
publisher = {Springer}
series = {Lecture Notes in Math.}
volume = {1852}
}

\bib{Krieger}{article}
{
author = {Krieger, W.}
title = {On entropy and generators of measure-preserving transformations}
date = {1970}
journal = {Trans. Amer. Math. Soc.}
volume = {149}
pages = {453--464}
}

\bib{Krengel}{article}
{
author = {Krengel, U.}
title = {Transformations without finite invariant measure have finite strong generators}
conference = {
		title = {First Midwest Conference, Ergodic Theory and Probability}
	}
	book = {
        series = {Springer Lecture Notes}
        volume = {160}
	}
	date = {1970}
	pages = {133--157}
}

\bib{Kuntz}{article}
{
author = {Kuntz, A. J.}
title = {Groups of transformations without finite invariant measures have strong generators of size 2}
date = {1974}
journal = {Annals of Probability}
volume = {2}
number = {1}
pages = {143--146}
}

\bib{Miller_thesis}{book}
{
author = {Miller, B. D.}
title = {PhD Thesis: Full groups, classification, and equivalence relations}
date = {2004}
publisher = {University of California at Los Angeles}
}

\bib{Munroe}{book}
{
author = {Munroe, M. E.}
title = {Introduction to Measure and Integration}
date = {1953}
publisher = {Addison--Wesley}
}

\bib{Nadkarni}{article}
{
author = {Nadkarni, M. G.}
title = {On the existence of a finite invariant measure}
date = {1991}
journal = {Proc. Indian Acad. Sci. Math. Sci.}
volume = {100}
pages = {203--220}
}

\bib{Rudolph}{book}
{
author = {Rudolph, D.}
title = {Fundamentals of Measurable Dynamics}
date = {1990}
publisher = {Oxford Univ. Press}
}

\bib{Varadarajan}{article}
{
author = {Varadarajan, V. S.}
title = {Groups of automorphisms of Borel spaces}
date = {1963}
journal = {Trans. Amer. Math. Soc.}
volume = {109}
pages = {191--220}
}

\bib{Wagon}{book}
{
author = {Wagon, S.}
title = {The Banach--Tarski Paradox}
date = {1993}
publisher = {Cambridge Univ. Press}
}

\bib{Weiss}{article}
{
author = {Weiss, B.}
title = {Countable generators in dynamics-universal minimal models}
date = {1987}
journal = {Measure and Measurable Dynamics, Contemp. Math.}
volume = {94}
pages = {321--326}
}
\end{biblist}
\end{bibdiv}
\end{document}